\documentclass[a4paper,10pt]{amsart}
\usepackage[english]{babel}
\usepackage[utf8]{inputenc}
\usepackage[T1]{fontenc}
\usepackage{csquotes}
\usepackage[style=numeric,
    useprefix,%
    giveninits=true,%
    hyperref,%
    doi=false,%
    url=false,%
    isbn=false,%
    backend=bibtex,%
    maxbibnames=99%
    ]{biblatex}
\usepackage{soul}
    
\bibliography{./BIB}

\usepackage{amssymb}
\usepackage{mathrsfs}
\usepackage{hyperref}
\usepackage[all]{xy} 
\usepackage[usenames,dvipsnames]{xcolor}
\hypersetup{colorlinks,%
citecolor=Black,%
filecolor=Black,%
linkcolor=Black,%
urlcolor=Black}
\usepackage{enumitem}   
\usepackage{mathtools} 
\mathtoolsset{showonlyrefs=true} 
\newcommand{\scr}[1]{\mathscr{#1}}
\newcommand{\frk}[1]{\mathfrak{#1}}
\newcommand{\bb}[1]{\mathbb{#1}}
\newcommand{\cal}[1]{\mathcal{#1}}
\newcommand{\N}{\mathbb{N}} 
\newcommand{\Q}{\mathbb{Q}} 
\newcommand{\Z}{\mathbb{Z}} 
\newcommand{\C}{\mathbb{C}} 
\newcommand{\R}{\mathbb{R}} 
\newcommand{\Id}{\mathrm{Id}}  
\newcommand{\dd}{\,\mathrm{d}}  
\newcommand{\did}{\,\mathrm{d}}  
\newcommand{\de}{\partial}      
\newcommand{\into}{\hookrightarrow}        
\newcommand{\THEN}{\Rightarrow} 
\newcommand{\IFF}{\Leftrightarrow}  

\newcommand{\dist}{\operatorname{dist}}
\newcommand{\spt}{\operatorname{spt}}
\newcommand{\capacity}{\operatorname{Cap}}

\newcommand{\interior}{\operatorname{int}}
\newcommand{\closure}{\operatorname{cl}}
\newcommand{\Length}{\operatorname{Length}}

\newcommand{\Per}{\operatorname{Per}}
\newcommand{\loc}{\mathrm{loc}}
\newcommand{\Diff}{\mathrm{D}}	

\newcommand{\dimpar}{\dim_{\rm par}}
\newcommand{\dimgr}{\dim_{\rm grt}}
\newcommand{\dimisp}{\dim_{\rm isp}}

\usepackage{bbm}
\newcommand{\ferrdHyp}{\mathbbmss h}
\newcommand{\ferrg}{\scr P}
\newcommand{\ferrdPar}{\mathbbmss p}

\usepackage{dsfont}
\newcommand{\one}{{\mathds 1\!}}

\def\Xint#1{\mathchoice
      {\XXint\displaystyle\textstyle{#1}}%
      {\XXint\textstyle\scriptstyle{#1}}%
      {\XXint\scriptstyle\scriptscriptstyle{#1}}%
      {\XXint\scriptscriptstyle\scriptscriptstyle{#1}}%
      \!\int}
   \def\XXint#1#2#3{{\setbox0=\hbox{$#1{#2#3}{\int}$}
        \vcenter{\hbox{$#2#3$}}\kern-.5\wd0}}
   \def\fint{\Xint-}

\theoremstyle{plain}
\newtheorem{proposition}{Proposition}[section]
\newtheorem{theorem}[proposition]{Theorem}
\newtheorem{lemma}[proposition]{Lemma}
\newtheorem{corollary}[proposition]{Corollary}
\newtheorem{thm}{Theorem}[section]

\theoremstyle{definition}
\newtheorem{definition}[proposition]{Definition}
\newtheorem{remark}[proposition]{Remark}
\newtheorem{question}{Question}

\theoremstyle{remark}

\title[Quasi-conformal VS quasi-isometric equivalence]{Quasi-conformal VS quasi-isometric equivalence in spaces with controlled growth}

\author[Fässler]{Katrin Fässler}
\address[Fässler]{Department Of Mathematics And Statistics, P.O. Box 35 (Mad), Fi-40014 University Of Jyväskylä, Finland}
\email{katrin.s.fassler@jyu.fi}

\author[Le Donne]{Enrico Le Donne}
\address[Le Donne]{Department of Mathematics, University of Fribourg, Ch. du Musée 23, 1700 Fribourg}
\email{enrico.ledonne@unifr.ch}

\author[Nicolussi Golo]{Sebastiano Nicolussi Golo}
\address[Nicolussi Golo]{Department of Mathematics, University of Fribourg, Ch. du Musée 23, 1700 Fribourg, Orcid ID: https://orcid.org/0000-0002-3773-6471}
\email{sebastiano2.72@gmail.com}

\author[Ottazzi]{Alessandro Ottazzi}
\address[Ottazzi]{School of Mathematics and Statistics, University of New South Wales, UNSW Sydney 2052, Australia}
\email{a.ottazzi@unsw.edu.au}

\author[Pansu]{Pierre Pansu}
\address[Pansu]{Laboratoire de Mathématiques d’Orsay, CNRS, Université Paris-Saclay, 91405 Orsay, France}
\email{pierre.pansu@math.u-psud.fr}

\date{\today. \IfFileExists{./.gittex}{\input{./.gittex}}{}}

\begin{document}
\maketitle

\begin{abstract}
	We study conditions under which quasi-conformal homeomorphisms are quasi-isometries.
	We show that if two nilpotent geodesic Lie groups are quasi-conformally homeomorphic, then they are quasi-isometrically equivalent.
	We also give more general results beyond the nilpotent case.
	In particular, we show that quasi-conformal homeomorphisms between geodesic Lie groups are quasi-isometries whenever the spaces have strict parabolic or hyperbolic conformal type.
	As a consequence, quasi-conformal homeomorphisms between geodesic Lie groups with infinite fundamental group are quasi-isometries.
	The statements for Lie groups are deduced from a more general study on metric measure spaces with uniformly locally bounded geometry.
\end{abstract}

\setcounter{tocdepth}{2}
\phantomsection
\addcontentsline{toc}{section}{Contents}
\tableofcontents

\section{Introduction}

\subsection{Overview}
In Geometry and Analysis, quasi-conformal homeomorphisms and quasi-isometries play a widespread role; see for instance~\cite{zbMATH01599991,zbMATH01965442,zbMATH05076257,zbMATH06729338,MR1683160,MR1654771,MR1334873,MR0236383,MR979599,zbMATH06696882,zbMATH02076292,MR3310079,MR3912638,MR3581902}.
Quasi-conformality (QC) is an infinitesimal property, while quasi-isometries (QI) preserve large-scale structures.
A typical interplay between these two types of maps is found in a crucial step of Mostow Rigidity Theorem: there, a quasi-isometry between hyperbolic spaces induces a quasi-conformal map between the visual boundaries.

It turns out there is a deeper connection:
In \cite[Theorem 9.8]{zbMATH01599991}, Bonk, Heinonen, and Koskela proved
that, under specific geometric conditions\footnotemark, quasi-conformal homeomorphisms between Gromov hyperbolic spaces are themselves quasi-isometries.
As an application, they obtained a boundary extension result for quasi-conformal maps between Gromov hyperbolic $n$-dimensional Hadamard manifolds ($n\ge 2$) with Ricci bounded geometry; see \cite[Theorem 1.15]{zbMATH01599991}. 
A special case of \cite[Theorem 9.8]{zbMATH01599991} (for quasi-conformal diffeomorphisms between Hadamard manifolds of pinched negative sectional curvature) was obtained earlier in~\cite{zbMATH03951479}.
\footnotetext{The geometric conditions on the space are: having bounded geometry, being roughly starlike, and having a Gromov boundary that is a nondegenerate continuum.}
The Bonk--Heinonen--Koskela result suggests a ``QC-implies-QI Principle'' for some types of geodesic spaces.
More recently, Pansu studied a notion of large-scale conformal maps and showed that such maps are quasi-isometries in the setting of manifolds with bounded geometry; see \cite[Theorem 2]{zbMATH07548718}.

We investigate similar results for Lie groups endowed with left-invariant (Riemannian or non-Riemannian) geodesic distances.
We show the validity of the ``QC-implies-QI Principle'' in a large class of geodesic Lie groups; see Theorem~\ref{thm68760387}.
On each geodesic Lie group, the top-dimensional Hausdorff measure is a Haar measure, which is locally Ahlfors regular and supports a local Poincaré inequality.
Hence, we consider each geodesic Lie group as a metric measure space with \emph{uniformly locally bounded geometry}, in the sense of PI spaces; see, for instance, \cite{MR1683160,MR1654771,MR1869604,MR3363168}, and Definition~\ref{def6803c8e2} and Proposition~\ref{prop682dc2dc} below.

The common core of our arguments is to obtain bounds on the conformal capacity of connected sets, compact or unbounded, which lead to bornologous comparisons between the original distance and some \emph{conformal gauges}, that is, distance-like functions that are bi-Lipschitz preserved by quasi-conformal maps.
Our main inspiration is the profound work by J.~Ferrand on Riemannian manifolds; see for instance~\cite{zbMATH00883879,zbMATH03431466}, and~\cite{MR1749853} for a survey.
Once we have such comparisons, we obtain that quasi-conformal maps are bi-bornologous.
The conclusion that quasi-conformal maps are quasi-isometries follows from the observation that bi-bornologous maps between geodesic spaces are quasi-isometries.

The strategy we have just outlined clearly cannot work in every geodesic Lie group.
For example, the Euclidean plane admits plenty of quasi-conformal homeomorphisms that are not quasi-isometries.
This example suggests that we need additional conditions on the space, such as being Gromov hyperbolic, as mentioned in the result by Bonk, Heinonen, and Koskela above.
In the non-Riemannian case, curvature bounds are too restrictive.
However, other large-scale conditions allow us to run the appropriate capacity estimates.
These extra conditions are of two kinds, depending on whether the space has \emph{parabolic} or \emph{hyperbolic conformal type}, as in Definition~\ref{def687743ae}.
In the parabolic case, we use bounds on the volume growth at infinity and the existence of \emph{quasi-straight sequences} of points, as in Definition~\ref{def68766881}, to obtain bounds on the conformal capacity of condensers made by pairs of unbounded connected sets.
In the hyperbolic case, we use the isoperimetric inequality at large scale to obtain bounds on the conformal capacity of so-called Grötzsch condensers, that is, compact connected sets paired with $\infty$.

In geodesic Lie groups, these different conditions are governed by the difference between the exponent $N$ of volume growth and the Hausdorff dimension $Q$; see, for instance, \cite{MR1749853}, or \cite{zbMATH01309448,zbMATH00999672,zbMATH01517666}.
Our methods apply whenever $N\neq Q$, while they leave open the situation when $N=Q$, as expected.
The Euclidean space $\R^n$ is an example where $N=Q=n$.\\


\subsection{Main results}
We next present our main results in detail.
Our focus will be on geodesic Lie groups.
However, we try to pinpoint minimal assumptions for each argument: for this reason, we prove most of our results in the setting of metric measure spaces.
Besides geodesic Lie groups, other spaces to which our Theorems~\ref{thm685ea294} and~\ref{thm6877508e} apply are locally compact isometrically homogeneous geodesic spaces, such as metric quotients of Lie groups, or Lie groups endowed with quasi-geodesic left-invariant distances, only to mention a few examples.

We call \emph{geodesic Lie group} $(G,d,\mu)$ every Lie group $G$ endowed with a geodesic left-invariant distance $d$, inducing the manifold topology, and a left Haar measure $\mu$.
From the work of Berestowski~\cite{zbMATH04098226}, every geodesic left-invariant distance on a Lie group is a Carnot-Carathéodory distance, also referred to as a sub-Finsler metric.
Examples of geodesic Lie groups are Riemannian and sub-Riemannian connected Lie groups.

For nilpotent Lie groups, we show that quasi-conformal equivalence implies quasi-isometric equivalence, with no extra conditions.

\begin{thm}[see Theorem~\ref{thm67f935d5}]\label{thm1}
	Let $G, H$ be nilpotent geodesic Lie groups.
	If there is a metrically quasi-conformal map $G\to H$, then there is a quasi-isometry $G\to H$.
\end{thm}
To prove Theorem~\ref{thm1}, we distinguish two cases: spaces of parabolic and of hyperbolic conformal type.
A metric measure space has \emph{parabolic conformal type} if all compact sets have zero $Q$-capacity, $Q$ being the Hausdorff dimension of the space; otherwise the space has \emph{hyperbolic conformal type}.
The conformal type depends on the metric measure structure and, for spaces with bounded geometry, it is preserved under quasi-conformal maps.

In the nilpotent case, the starting point is to show that a simply connected nilpotent geodesic Lie group has parabolic conformal type if and only if it is a Carnot group; see Proposition~\ref{prop67f92fca}.
By the Differentiation Theorem \cite{MR1334873}, at every point of intrinsic differentiability, the differential of a quasi-conformal map $G\to H$ between Carnot groups is a bi-Lipschitz map $G\to H$, and thus, in particular, we recover a quasi-isometry $G\to H$ between the same spaces.

In the hyperbolic case, we show that every quasi-conformal map $G\to H$ is a quasi-isometry, while the non-simply connected case is studied by passing to the universal covering.
In fact, the hyperbolic case is treated in a more general setting, beyond nilpotent Lie groups.

The distinction between hyperbolic conformal type and parabolic conformal type in geodesic Lie groups needs to be made finer.
On a geodesic metric measure space $(X,d,\mu)$, one may define the $p$-capacity $\capacity_p(E)$ of each subset $E$ of $X$, for every $p\in[1,+\infty)$; see Definition~\ref{def6803ccb2}.
If $(X,d,\mu)$ has locally $Q$-bounded geometry, then the $Q$-capacity is preserved by quasi-conformal maps: for this reason, the distinction between hyperbolic and parabolic conformal type is done in terms of the $Q$-capacity.
For $p\neq Q$, the $p$-capacity is not necessarily preserved by quasi-conformal maps.
If we define the \emph{parabolic dimension} of $X$ by
\begin{equation}
	\dimpar(X) := \inf\{p\in[1,+\infty):\capacity_p(E)=0,\ \forall E\Subset X\} ,
\end{equation} then this number may not be a quasi-conformal invariant.
Surprisingly, it turns out that, for geodesic Lie groups, the value $\dimpar(X)$ is a quasi-conformal invariant; see Corollary~\ref{cor67f6897a}.

We prove that a geodesic Lie group with Hausdorff dimension $Q$ has parabolic conformal type if and only if $\dimpar(G) \le Q$, and has hyperbolic conformal type if and only if $Q<\dimpar(G)$; see Corollary~\ref{cor67f056ce}.
We say that $G$ has \emph{strictly parabolic conformal type} if $\dimpar(G) < Q$, and we call the threshold case when $\dimpar(G) = Q$ as \emph{liminal parabolic conformal type}.
We are in the latter case exactly when the geodesic Lie group is globally Ahlfors regular.

\begin{thm}[see Theorem~\ref{thm68780db5}]\label{thm68760387}
	Quasi-conformal maps between non-compact geodesic Lie groups preserve both the Hausdorff and the parabolic dimensions of the groups; in particular, they maintain the distinction between hyperbolic, parabolic, and strictly parabolic conformal types.
	
	Moreover, quasi-conformal maps between geodesic Lie groups of hyperbolic or strictly parabolic conformal type are quasi-isometries.
\end{thm}


Theorem~\ref{thm68760387} implies in particular that the sub-Riemannian Heisenberg group and the sub-Riemannian rototranslation group are not quasi-conformally equivalent, as was already shown in~\cite{MR3428952}.

In the second part of Theorem~\ref{thm68760387}, we have not included the liminal parabolic case, because it would not be true.
An example of a geodesic Lie group with liminal parabolic conformal type is the Euclidean space $\R^n$: we know that in this case, there are plenty of quasi-conformal maps that are not quasi-isometries.
One can find more examples among Carnot groups: in particular, the sub-Riemannian Heisenberg group, which also admits quasi-conformal transformations that are not quasi-isometries.

An example of a geodesic Lie group with strictly parabolic conformal type is the direct product $\R^n\times\bb S^1$, since the Hausdorff dimension is $n+1$, while its parabolic dimension is $n$, as shown in Theorem~\ref{thm6862ad98} below.
Theorem~\ref{thm68760387} implies that every quasi-conformal map from $\R^n\times\bb S^1$ onto itself is a quasi-isometry.
This latter fact is also a consequence of \cite{zbMATH03703044,zbMATH06156153} because $\R^n\times\bb S^1$ is isometric to $\C^* \times\R^{n-1} \subseteq \R^{n+1}$ when equipped with the quasi-hyperbolic metric, and the latter distance is conformally equivalent to the Euclidean one. 
We generalize this example as follows:

\begin{thm}[see Theorem~\ref{thm68d55353}]
	Let $f: G_1\to G_2$ be a quasi-conformal map between geodesic Lie groups.
	If the fundamental group of $G_1$ is infinite, then $f$ is a quasi-isometry.
\end{thm}

Standard results from the theory of metric measure spaces give the following corollaries of Theorem~\ref{thm68760387}.
See Section~\ref{sec6a91e2a8} for additional comments and the proofs.

\begin{thm}[see Theorem~\ref{thm6a79e253}]
	Quasi-conformal homeomorphisms between geodesic Lie groups of dimension greater than 1 are global quasisymmetries.
\end{thm}

\begin{thm}[see Theorem~\ref{thm6a70fdfc}]
	Let $G$ be a geodesic Lie group that is not liminal parabolic.
	Then, all 1-quasi-conformal maps $G\to G$ are isometries.
\end{thm}

For maps between nilpotent geodesic Lie groups, Theorem~\ref{thm68760387} with a few algebraic and metric observations implies Theorem~\ref{thm1}.
Indeed, for simply connected nilpotent geodesic Lie groups, the Hausdorff dimension $Q$ and the degree $N$ of volume growth at large scale are computable algebraically from the structure of the Lie algebra.
With this fact, we show in Proposition~\ref{prop67f92fca} that $Q\le N$ on simply connected nilpotent geodesic Lie groups, with equality holding only for Carnot groups.
See Section~\ref{subs68768d93} for more details.

However, beyond nilpotent Lie groups, it is not clear that QC equivalence implies QI equivalence for geodesic Lie groups of liminal parabolic conformal type. 
An example is the $\ell^2$ product $(G,d) = (\bb H^1,d_R)\times (\scr R,d_{sR})$ of the Riemannian Heisenberg group $(\bb H^1,d_R)$ and the sub-Riemannian rototranslation group $(\scr R,d_{sR})$: $G$ has parabolic and Hausdorff dimension 7, so, it is a liminal parabolic geodesic Lie group but it is not nilpotent.
This scenario will be studied in a forthcoming work.

Theorem~\ref{thm68760387} is proven in several steps, which we describe in the following sections.

\subsection{Conformal type, volume growth, and isoperimetric inequalities}

Quasi-conformal maps between spaces with $Q$-bounded geometry preserve both topological properties and $Q$-capacities.
It is reasonable, then, to define distance-like functions on metric measure spaces that combine these two types of information.
We will use in particular two of them.
In the hyperbolic case, we use what we call the \emph{hyperbolic Ferrand distance} $\ferrdHyp$:
for every two points $x,y$, the quantity $\ferrdHyp(x,y)$ is the infimum of $Q$-capacity of closed and connected sets $E$ that contain $x$ and $y$; see Definition~\ref{def of ferrdHyp}.
These sets $E$ are also known as \emph{Grötzsche capacitors}.
In the parabolic case, we use instead what we call the \emph{parabolic Ferrand distance} $\ferrdPar$:
for every two points $x,y$, the quantity $\ferrdPar(x,y)$ is the $\frac1Q$-power of the inverse of the infimum of $Q$-capacity of capacitors $(E,F)$ made of closed connected unbounded sets with $x\in E$ and $y\in F$; see Definition~\ref{eq68da4e38}.
These capacitors $(E,F)$ are also known as \emph{Teichmüller\footnotemark capacitors}.
\footnotetext{Disclaimer: Teichmüller was an active Nazi.}

To get estimates of Ferrand distances, we use two geometric properties: the isoperimetric profile at large scale and the volume growth at large scale.

We say that a metric measure space $(X,d,\mu)$ supports the \emph{isoperimetric inequality at large scale of order $\frac{N-1}{N}$}, for some $N\ge1$, if there exists $C\in\R$ such that,
for every $E\subset X$ measurable,
\begin{equation}
	\mu(E) \ge 1
	\qquad\THEN\qquad
	\mu(E)^{\frac{N-1}{N}} \le C \Per(E) ,
\end{equation}
where $\Per(E)$ denotes the perimeter, i.e., the total variation of the characteristic function of $E$.
The supremum of all such $N$ is the \emph{isoperimetric dimension} $\dimisp(X)$ of $X$.

We say that a metric measure space $(X,d,\mu)$ has \emph{degree of large-scale growth at most~$N$}, for some $N\ge1$ if there exists $C\in\R$ such that
\begin{equation}\label{eq67c86af8_intro}
	\mu(B(x,R)) \le C R^N ,
	\qquad\forall x\in X,\ \forall R>1 .
\end{equation}
The \emph{growth dimension} $\dimgr(X)$ of $X$ is the infimum of all such $N$.
If $\dimgr(X)<\infty$, then $X$ has \emph{polynomial (volume) growth}.
Under some mild additional hypothesis on the spaces, we have that $\dimgr(X)$ is a quasi-isometric invariant; see Proposition~\ref{prop6866b133}.

Using the coarea inequality, one easily checks that $\dimisp(X) \le \dimgr(X)$.
We will show in Corollary~\ref{cor6863f5f0} and Proposition~\ref{prop67f0559a} that,
{\it if $X$ is a metric measure space with compact sets of arbitrarily large measure, then
\begin{equation}\label{eq68765d7d}
	\dimisp(X) \le \dimpar(X) \le \dimgr(X) .
\end{equation}
}


To prove Theorem~\ref{thm68760387}, we first characterize the conformal type of geodesic Lie groups, in terms of their growth and isoperimetric dimensions.
Sobolev and isoperimetric inequalities and their relation with heat diffusion and volume growth on Riemannian Lie groups have been studied thoroughly; see, for instance~\cite{MR839109,zbMATH05521653, MR1749853}.
In particular, Pittet proved in~\cite[Theorem 2.1]{zbMATH01782641} 
that the isoperimetric profile at large scale of a Riemannian Lie group
 is determined by its volume-growth type and whether it is amenable and unimodular or not.
We need a similar statement for geodesic Lie groups:
since every two geodesic left-invariant distances on a Lie group are quasi-isometric equivalent to each other, our task is to show that the isoperimetric dimension is preserved under quasi-isometries.
We will do this following a strategy of Kanai from~\cite{zbMATH03883148}; see Theorem~\ref{thm684c4f0a}.
As a consequence, we obtain the following characterization of isoperimetric profiles at large scale for geodesic Lie groups.

\begin{thm}[see Theorem~\ref{thm68655128}]\label{thm6862ad98}
	Let $(G,d,\mu)$ be a non-compact geodesic Lie group.
	Then exactly one of the following cases happens:
	\begin{enumerate}[label=(\alph*)]
	\item
	$G$ is non-amenable or non-unimodular, it has exponential volume growth, 
	and there is $C$ such that, for every measurable $E\subset G$, 
	if $\mu(E) \ge 1$ then $\mu(E) \le C \Per(E)$.
	In particular, $\dimisp(X) = \dimgr(X) = \infty$.
	\item
	$G$ is amenable, unimodular, it has exponential volume growth,
	and for every $N>1$ there is $C$ such that, for every measurable $E\subset G$,
	if $\mu(E) \ge 1$ then $\mu(E)^{\frac{N-1}{N}}  \le C \Per(E)$.
	In particular, $\dimisp(X) = \dimgr(X) = \infty$.
	\item
	$G$ is amenable, unimodular, it has polynomial growth of degree $N$, for some $N\in\N$, 
	and there is $C$ such that, for every measurable $E\subset G$, 
	if $\mu(E) \ge 1$, then $\mu(E)^{\frac{N-1}{N}} \le C \Per(E)$.
	In particular, $\dimisp(X) = \dimgr(X) < \infty$.
	\end{enumerate}
	As a consequence, for each non-compact geodesic Lie group $G$, one has equalities in~\eqref{eq68765d7d}, that is,
	\begin{equation}
		\dimisp(G) = \dimpar(G) = \dimgr(G) ,
	\end{equation}
\end{thm}

Recall that all Lie groups of polynomial growth are amenable and unimodular; see~\cite[II.4.5]{MR2000440}.
Note that Theorem~\ref{thm6862ad98} provides a complete characterization of the parabolic conformal type of geodesics Lie groups in terms of their volume growth.

\begin{thm}[see Corollary~\ref{cor67f056ce}]\label{thm685e39f2}
	A geodesic Lie group $(G,d,\mu)$ has parabolic conformal type if and only if there exists $C>0$ such that
	\begin{equation}
		\mu(B(x,r)) \le C r^Q ,
		\qquad \forall r\ge 1,\ \forall x\in G,
	\end{equation}
	where $Q$ is the Hausdorff dimension of $(G,d)$.
\end{thm}

Once we have established the conformal type of geodesic Lie groups, we can use the quasi-conformal invariance of capacity to prove Theorem~\ref{thm68760387}.
The two cases, parabolic and hyperbolic, are treated quite differently.

\subsection{QC maps in the hyperbolic conformal type}
As we have described earlier, on each metric measure space, one can consider the hyperbolic Ferrand distance $\ferrdHyp$.
By construction, the function $\ferrdHyp$ satisfies the triangular inequality and quasi-conformal maps are bi-Lipschitz maps for $\ferrdHyp$.
In the parabolic case, we have $\ferrdHyp\equiv0$.
However, temporarily omitting some extra condition, if $(X,d,\mu)$ is a metric measure space with hyperbolic conformal type,
the identity map $(X,d)\to(X,\ferrdHyp)$ is \emph{bi-bornologous}, that is, $\lim_{n\to\infty}d(x_n,y_n)=\infty$ if and only if $\lim_{n\to\infty}\ferrdHyp(x_n,y_n)=\infty$,
for every sequences $\{x_n\}_{n\in\N}$ and $\{y_n\}_{n\in\N}$ in $X$.

In all crucial capacity estimates, we make use of \emph{monotone functions}.
Following the strategy outlined by Ferrand in \cite{zbMATH00883879}, we observe that the capacity can be obtained by infimising over monotone functions; see Proposition~\ref{prop67e505c3}.
Moreover, we recover from the work \cite{MR1683160} of Haj\l{}asz and Koskela upper bounds for the oscillation of monotone functions; see Section~\ref{subs67ef8b4f}.

At this point, we can regard each quasi-conformal map $f:(X,d)\to (Y,d)$ as the composition of bornologous functions
\begin{equation}
	(X,d) \overset{\Id}\longrightarrow 
	(X,\ferrdHyp) \overset{f}\longrightarrow 
	(Y,\ferrdHyp) \overset{\Id}\longrightarrow 
	(Y,d) .
\end{equation}
Consequently, the map $f:(X,d)\to (Y,d)$ is bornologous.
If $X$ and $Y$ are quasi-geodesic spaces, bornologous functions are large-scale Lipschitz;
see Proposition~\ref{prop6803b35d}.

The above strategy needs extra assumptions on the space $(X,d,\mu)$, which will be satisfied for geodesic Lie groups.
See Section~\ref{sec6866e144} for the definition of uniformly locally bounded geometry.
A more general result is the following:	

\begin{thm}[see Theorem~\ref{thm685cec65}]\label{thm685ea294}
	Let $X_1$ and $X_2$ be metric measure spaces such that:
	\begin{enumerate}
	\item 
	they are quasi-geodesic and boundedly compact;
	\item
	they have uniformly locally bounded geometry, with Hausdorff dimension $Q>1$;
	\item
	for each $r>0$, the measure of balls of radius $r$ is uniformly bounded from above.
	\end{enumerate}
	If $\dimisp(X_1)>Q$ and $\dimisp(X_2)>Q$, 
	then every quasi-conformal map $X_1\to X_2$ is a quasi-isometry.
\end{thm}

\subsection{QC maps in the parabolic conformal type}

In the parabolic case, the above hyperbolic Ferrand distance $\ferrdHyp$ is the zero function.
So, we use the parabolic Ferrand distance $\ferrdPar$.
We show that, under some additional hypothesis on the space, for every pair of sequences $\{x_n\}_{n\in\N}$ and $\{y_n\}_{n\in\N}$, 
we have $\lim_{n\to\infty} d(x_n,y_n) = \infty$ if and only if $\lim_{n\to\infty} \ferrdPar(x_n,y_n) = \infty$.
Hence, a strategy similar to the one applied in the hyperbolic case allows us to show that quasi-conformal maps are bornologous, with bornologous inverse, and thus quasi-isometries.

For the argument, we need both lower and upper bounds to $\ferrdPar$.
Upper bounds are based on properties of monotone functions; see Proposition~\ref{prop68720f23}.
For the lower bounds, we need to construct, for every pair of points $x$ and $y$, two appropriate connected and unbounded sets $E$ and $F$ such that $x\in E$ and $y\in F$.
The intuitive idea is to take a straight line passing through the points and consider the two unbounded half-lines separated by the points.

To do this, we use what we call \emph{quasi-straight sequences}:
roughly speaking, a quasi-straight sequence is a function $\zeta:\Z\to X$ unbounded in both directions, so that its image is made of roughly evenly distributed points that are almost aligned, quantitatively; see Definition~\ref{def68766881}.
We note that quasi-straight sequences form a larger class than infinite quasi-geodesics; see Remark~\ref{rem6876686d}.

A metric space is \emph{quasi-straightenable} if every two points lie close to a quasi-straight sequence, quantitatively with constants that are independent of the points;
see Definition~\ref{def68766881}.
We can show that every non-compact geodesic Lie group with polynomial growth is quasi-straightenable, because simply connected nilpotent Lie groups are quasi-straightenable (using straight lines in exponential coordinates), and because every geodesic Lie group with polynomial growth is quasi-isometric to a simply connected nilpotent Lie group.

With these notions, we can prove the second part of Theorem~\ref{thm68760387}, when both spaces are strictly parabolic.

\begin{thm}[see Theorem~\ref{thm68755ab4}]\label{thm6877508e}
	Let $X_1$ and $X_2$ be metric measure spaces such that:
	\begin{enumerate}
	\item 
	they are quasi-geodesic and boundedly compact;
	\item
	they have uniformly locally bounded geometry, with Hausdorff dimension $Q>1$;
	\item
	they are quasi-straightenable.
	\end{enumerate}
	If $\dimgr(X_1)<Q$ and $\dimgr(X_2)<Q$, 
	then every quasi-conformal map $X_1\to X_2$ is a quasi-isometry.
\end{thm}

Observe that the above Theorem~\ref{thm6877508e} can only be applied to strictly parabolic geodesic Lie groups.
This hypothesis is used in the lower bound for $\ferrdPar$; see Proposition~\ref{prop68720d05}.
However, we can show that geodesic Lie groups that are liminal parabolic are globally Loewner spaces and thus the parabolic Ferrand distance $\ferrdPar$ is $0$ at every pair of points;
see Proposition~\ref{prop687658f1}.
Therefore, a strictly parabolic geodesic Lie group cannot be quasi-conformal equivalent to a liminal parabolic geodesic Lie group.
The latter argument is shadowing ideas from~\cite{MR3428952}.


\subsection{Open questions}\label{sub:open-questions}
The dichotomy underlying our results is governed by the comparison between the
Hausdorff dimension $Q$ and the growth dimension $N=\dim_{\mathrm{grt}}(G)$ of a
geodesic Lie group: our methods show that quasi-conformal maps are
quasi-isometries whenever $N\neq Q$, that is, in the strictly parabolic and in
the hyperbolic case, and they break down precisely at the liminal threshold
$N=Q$, where the group is globally Ahlfors regular. A number of natural questions
remain open, both at this threshold and beyond it. We collect here those we find most appealing; some would extend the results obtained in this paper, while others connect our results to classical themes in quasi-conformal geometry.\\

In every case in which we prove that each quasi-conformal map is a quasi-isometry,
the conclusion is purely large-scale and qualitative.
It is therefore natural to ask whether more is true. 
We are led to ask the following question.

\begin{question}\label{q:bounded-distance-biLip}
Let $f\colon G_1\to G_2$ be a quasi-conformal map between non-compact geodesic
Lie groups that are not globally Ahlfors regular. Is $f$ at bounded distance from
a bi-Lipschitz map?
\end{question}

The obstruction to upgrading quasiisometries to biLipschitz maps has both a large scale and a local character. The large-scale obstruction is described by K. Whyte in \cite{zbMATH01425235}. For the local aspects, according to D. Sullivan \cite{zbMATH03751837}, the TOP, QC and LIP categories are essentially equivalent, except in dimension 4. This means that quasi-conformal maps can be approximated with biLipschitz maps on compact manifolds. A relative version can be found in \cite{zbMATH01566038}.
Is this a hint that quasiisometric quasi-conformal maps could be approximated with biLipschitz maps on manifolds of bounded geometry?
In dimension 4, TOP differs from QC, \cite{zbMATH04155459}, but the question whether QC=LIP was open in 1990; see \cite{zbMATH00057898}.

Regarding Ahlfors regular groups, in \cite{MR3581902}, Xiangdong Xie conjectured that if $G$ is a Carnot group that is neither Abelian nor one of the Heisenberg groups, then every quasi-conformal self-map $G\to G$ is bi-Lipschitz. 
We stress that in every Riemannian manifold, in fact, it is easy to construct by local radial stretching quasi-conformal maps that are not bi-Lipschitz maps, and moreover are the identity outside some compact set. 
Similar but less obvious is the construction of such maps in every contact manifold; see \cite{zbMATH01640258}.
We therefore extend Xie's conjecture to liminal groups:

\begin{question}
	{\it Xie's generalized conjecture}:
	If $G_1$ and $G_2$ are globally Ahlfors regular geodesic Lie groups, each of them neither Riemannian nor locally bi-Lipschitz to a Heisenberg group, is then every quasi-conformal map $G_1\to G_2$ bi-Lipschitz? 
\end{question}

Regarding some recent progress on the above question, we mention \cite{zbMATH06503627, zbMATH08022182, arXiv:2112.01866}.\\

%

\subsection{Acknowledgments}
The authors thank Jeremy Tyson, Gioacchino Antonelli, Abhishek Pandey, and Xiangdong Xie for fruitful discussions.
We also thank Elia Bubani and Andrea Tettamanti for proofreading a preliminary version of this work.

KF was partially supported by the Research Council of Finland (formerly, Academy of Finland), grants 321696, 328846, 352649.
ELD and SNG were partially supported 
by the Swiss National Science Foundation (grant 200021-204501 ‘{\it Regularity of sub-Riemannian geodesics and applications}’), 
by the European Research Council (ERC Starting Grant 713998 GeoMeG `{\it Geometry of Metric Groups}'), 
by the Academy of Finland 
(grant 288501 `{\it Geometry of subRiemannian groups}', 
grant 322898 `{\it Sub-Riemannian Geometry via Metric-geometry and Lie-group Theory}'. 

\section{Spaces and maps}\label{sec6866e144}
In this section, we introduce various properties of metric spaces and of their maps.

A metric space is said to be \emph{locally compact} if every point has a compact neighborhood.
A metric space is \emph{boundedly compact} if bounded sets are pre-compact.
If $(X,d)$ is a metric space, we denote the \emph{(open) ball} with center $x\in X$ and radius $r>0$ as $B(x,r) := \{y\in X: d(x,y)<r\}$,
and the \emph{closed ball} as $\bar B(x,r) := \{y\in X: d(x,y)\le r\}$.
If $E$ is a subset of a metric space and $r>0$, we define
\begin{equation}
	B(E,r):= \bigcup_{x\in E} B(x,r) .
\end{equation}
A subset $\scr K$ of a metric space $(X,d)$ is \emph{$r$-separated} for some $r\in(0,+\infty)$ if $d(x,y)\ge r$ for all $x,y\in\scr K$ distinct.
An $r$-separated subset $\scr K\subset X$ is \emph{maximal} if it is maximal among $r$-separated sets with respect to inclusion of sets.
Equivalently, an $r$-separated subset $\scr K\subset X$ is maximal if for every $z\in X$ there exists $x\in\scr K$ with $d(x,z)<r$.
In every metric space, for every $r\in(0,+\infty)$, there exists a maximal $r$-separated set by Zorn's Lemma.

A \emph{metric measure space} is a triple $(X,d,\mu)$ where $(X,d)$ is a metric space and $\mu$ is a Borel regular measure on $X$.
Given a metric measure space $(X,d,\mu)$, the \emph{volume functions of $X$} are
$\mu_X^+,\mu_X^-:(0,+\infty) \to [0,+\infty]$, defined by
\begin{equation}\label{eq6866b26f}
\begin{aligned}
	\mu_X^+(r) := \sup\{\mu(B(x,r)): x\in X\} ,\\
	\mu_X^-(r) := \inf\{\mu(B(x,r)): x\in X\} ,
\end{aligned}
\end{equation}
for all $r\in (0,+\infty)$.

A Borel function $g:X\to[0,+\infty]$ is an \emph{upper gradient} of a function $u:X\to\R$ if, for every rectifiable curve $\gamma:[0,1]\to X$, we have
\begin{equation}\label{eq6804a338}
	|u(\gamma(1)) - u(\gamma(0))| \le \int_\gamma g.
\end{equation}
Notice that the constant function $g\equiv+\infty$ is an upper gradient for every measurable function $X\to\R$.
If $u: A \to \mathbb{R}$ is defined only on a Borel subset $A$ of $X$, we speak about upper gradients of $u$ with self-explanatory meaning, considering $A$ as a metric measure space itself. In this way, if $g$ is an upper gradient for a function  $u$ on $X$, then $g|_A$ is an upper gradient for $u|_A$.

\begin{definition}[Uniformly locally bounded geometry]\label{def6803c8e2}
	A metric measure space $(X,d,\mu)$ has \emph{locally $Q$-bounded geometry} with data 
	$(Q,R,C_{\rm A}, C_{\rm P},\sigma)$,
	where $Q\ge1$, $R:X\to(0,+\infty]$, $C_{\rm A}\ge1$, $C_{\rm P}\geq0$, and $\sigma\ge1$,
	if $X$ is separable, pathwise connected, and locally compact, and if the following two conditions hold:
	First,  
	\begin{equation}\label{eq67f646a8}
		C_{\rm A}^{-1} r^Q \le \mu(B(x,r)) \le C_{\rm A} r^Q ,
		\qquad\forall x\in X,\ \forall r\in(0,R(x)) .
	\end{equation}
	Second, for every $x\in X$ and for all $r\in(0,R(x)/\sigma)$,	if $u:B(x,\sigma r)\to\R$ is a locally integrable function with upper gradient $g$, then
	\begin{equation}\label{eq67f646d6}
		\fint_{B(x,r)} |u-u_B| \did\mu \le C_{\rm P} r \left( \fint_{B(x,\sigma r)} g^Q \did \mu\right)^{1/Q} ,
	\end{equation}
	where $u_B := \fint_{B(x,r)} u \did\mu$.
	
	In addition, we say that $(X,d,\mu)$ has \emph{uniformly locally $Q$-bounded geometry},
	if it has locally $Q$-bounded geometry with a constant function $R$.
\end{definition}

The above definition is in line with the literature; see also \cite[p.~150]{zbMATH01787110}, \cite[Def.~9.1]{MR1869604}, and in particular the definition of $Q$-BG space from~\cite[Chapter~9]{zbMATH01599991}.

By the previous remark about the restriction of functions and upper gradients, if $(X,d,\mu)$ has uniformly locally $Q$-bounded geometry in the sense of Definition~\ref{def6803c8e2}, then~\eqref{eq67f646d6} holds (with the indicated collection of balls) for all $u: X\to \R$ with upper gradient $g$ on $X$.

\begin{definition}\label{def687378f2}
    A metric space $(X,d)$ is \emph{$C$-quasi-convex} for some $C\ge1$ if for every $x,y\in X$ there is a curve $\gamma$ from $x$ to $y$    such that the length $\Length(\gamma)$ is at most $Cd(x,y)$. 
    It is \emph{uniformly locally $C$-quasi-convex} if 
    for every $x,y\in X$ with $d(x,y) \le 1/C$ there exists a curve $\gamma$ from $x$ to $y$ with $\Length(\gamma) \le C d(x,y)$.
\end{definition}

\begin{proposition}\label{prop686fb4ae-uniform}
    Let $(X,d,\mu)$ be a boundedly compact metric measure space with uniformly locally bounded geometry   with data $(Q,R,C_{\rm A}, C_{\rm P},\sigma)$, as in Definition~\ref{def6803c8e2}.  
    Then there is $C\ge1$, depending only on the data of $(X,d,\mu)$, such that $(X,d)$ is uniformly locally $C$-quasi-convex.
	If $R=\infty$, then $(X,d)$ is quasi-convex.
\end{proposition}

The first part of Proposition \ref{prop686fb4ae-uniform} holds according to \cite[Lemma 4.9]{zbMATH06951611};
see also \cite[Lemma 2.17]{zbMATH08075610}. 
If $R=\infty$, then $(X,d)$ is quasiconvex; see for instance \cite[Proposition 4.4]{MR1683160} or \cite[Theorem 8.3.2]{MR3363168}.

\begin{definition}\label{def6803ccb2}
	Let $(X,d,\mu)$ be a metric measure space.
	A \emph{capacitor} is a pair $(E;F)$ of subsets $E,F\subset X$.
	The set of \emph{admissible functions for the capacitor $(E;F)$} is
	\begin{equation}
		\cal A(E;F) :=
		\left\{ u:X\to\R : \text{$u$ measurable, } E\subset\{u\ge1\} \text{ and }F\subset\{u\le 0\} \right\} .
	\end{equation}
	For $p\in[1,\infty)$, the \emph{$p$-capacity} of a capacitor $(E;F)$ is
	\begin{equation}
		\capacity_p(E;F)
		:= \inf\left\{
		\int_X g^p \did\mu : 
		\text{$g$ upper gradient of $u\in\cal A(E;F)$}
		\right\} .
	\end{equation}
	The set of \emph{admissible functions for the set $E\subset X$} is
	\begin{equation}
		\cal A(E;\infty) :=
		\left\{ u:X\to\R : 
		\begin{array}{c}
			\text{$u$ measurable, with compact support,} \\
			\text{and } E\subset\{u\ge1\}
		\end{array}
		\right\} .
	\end{equation}
	For $p\in[1,\infty)$, the \emph{$p$-capacity (at infinity)} of a set $E\subset X$ is
	\begin{eqnarray*}
		\capacity_p(E) &:= \capacity_p(E;\infty) &:=  \inf\left\{
		\int_X g^p \did\mu : 
		\text{$g$ upper gradient of $u\in\cal A(E;\infty)$}
		\right\} .\\
		&&:=  \inf\left\{
		\capacity_p(E;F) : X\setminus F
		\text{ is compact}
		\right\} .
	\end{eqnarray*}
\end{definition}

\begin{remark}\label{rem67f8c4ee}
	The capacity $E\mapsto \capacity_p(E)$ is sub-additive and monotone.
	Indeed, if $E_1,E_2\subset X$ and $u_1,u_2:X\to \R$ are functions with compact support and $E_j\subset\{u_j\ge 1\}$, then $u=\max\{u_1,u_2\}$ is such that $E_1\cup E_2\subset\{u\ge1\}$ and an upper gradient of $u$ is $g=\max\{g_1,g_2\}$, where $g_j$ is an upper gradient of $u_j$.
	Thus $\capacity_p(E_1\cup E_2) \le \capacity_p(E_1) + \capacity_p(E_2)$.
	The monotonicity is clear from the definition as an infimum: $E_1\subset E_2$ implies $\capacity_p(E_1) \le \capacity_p(E_2)$.
\end{remark}

We use the capacity to introduce a first notion of quasi-conformal maps.

\begin{definition}[Quasi-conformal map: geometric definition]\label{def68cc0a5c}
	A map $f:X\to Y$ between metric measure spaces is \emph{geometrically quasi-conformal of index $Q$} or, simply, \emph{quasi-conformal},	if it is a homeomorphism 	and if there exists $K\in\R$ such that,	for every $E,F\subset X$,
	\begin{equation}\label{bi-Lipschitz preserve the capacity}
		\frac1K \capacity_Q(E;F) \le
		\capacity_Q(f(E);f(F)) 
		\le K \capacity_Q(E;F) .
	\end{equation}
\end{definition}

A related and more elementary definition is the so-called \emph{metric definition of quasi-conformal maps}:

\begin{definition}[Quasi-conformal map: metric definition]\label{QC metric definition}
	A map $f:X\to Y$ between metric spaces is \emph{metrically quasi-conformal} if it is a homeomorphism 	and if there exists $C\in\R$ such that,	for every $x\in X$,
	\begin{equation}
		\limsup_{r\to0} \frac{\sup\{ d(f(x),f(x')) : x'\in B(x,r) \} }{ \inf\{ d(f(x),f(x')) : x'\notin B(x,r) \}  } \le C .
	\end{equation}
\end{definition}

\begin{remark}\label{equiv_def_QC}In the literature, it is known that the metric definition of quasi-conformal map implies the above geometric definition for homeomorphisms between spaces with $Q$-bounded geometry; see~\cite[Theorem~9.10]{MR1869604} or \cite[Theorem~9.8]{MR1869604} with \cite[Proposition~2.17]{MR1654771} and \cite[Proposition~9.4]{zbMATH01599991}.
See also \cite{zbMATH06028338}.
The above Definition~\ref{def68cc0a5c} is a priori weaker than a more standard geometric definition, which is stated in terms of the modulus of curve families. Our definition would correspond to considering only families of curves connecting two sets (instead of arbitrary curve families); for the equivalence in the Euclidean setting, see \cite[Corollary~36.2]{zbMATH03350289}. 
In our work, we will mostly use the geometric Definition~\ref{def68cc0a5c}.
Later, in the case of nilpotent geodesic Lie groups (i.e., in Theorem~\ref{thm67f935d5}), we will consider the metric Definition~\ref{QC metric definition} and make use of the fact that it implies the geometric Definition~\ref{def68cc0a5c}.
\end{remark}

\begin{definition}[Quasi-isometry]
	A function $f:X\to Y$ between metric spaces is \emph{large-scale Lipschitz} if there exist $L,C\in\R$ such that
	\begin{equation}\label{eq67f620da}
		d(f(p),f(q)) \le L d(p,q) + C,
		\qquad\forall p,q\in X.
	\end{equation}
	A function $f:X\to Y$ between metric spaces is a \emph{quasi-isometry} if there exist $L,C\in\R$ such that
	\begin{equation}\label{eq68ce76a1}
		\frac1L d(x_1,x_2) - C \le
		d(f(x_1),f(x_2)) 
		\le L d(x_1,x_2) + C ,
		\qquad \forall x_1,x_2\in X ,
	\end{equation}
	and such that, for every $y\in Y$ there exists $x\in X$ with $d(f(x),y)\le C$.
\end{definition}

\begin{definition}[Quasi-geodesic space]\label{def6874b30e}
	A metric space $(X,d)$ is \emph{$(\ell,r)$-quasi-geodesic} for some $\ell,r\in\R$,
	if, for every $p,q\in X$ there are $x^0 = p, x^1,\dots,x^k = q$ with $d(x^{j-1},x^j) \le r$ 
	and $\sum_{j=1}^k d(x^{j-1},x^j) \le \ell d(p,q)$.
\end{definition}

\begin{definition}[Bornologous function]
	A function $f:X\to Y$ between metric spaces is \emph{bornologous} if the following holds:
	if $\{p_n\}_{n\in\N},\{q_n\}_{n\in\N}\subset X$ are sequences in $X$ such that $\lim_{n\to\infty}d(f(p_n),f(q_n))=\infty$, then $\lim_{n\to\infty}d(p_n,q_n)=\infty$.
	For extra information, see~\cite[p.6]{zbMATH02012373}.
\end{definition}
Compositions of bornologous maps are bornologous.
Large-scale Lipschitz maps are bornologous.
It is classic that for geodesic spaces the inverse is also valid; see~\cite[Lemma~1.10]{zbMATH02012373}.
For completeness, we include a proof of this exercise.

\begin{proposition}\label{prop6803b35d}
	Let $X$ and $Y$ be metric spaces.
	If $X$ is quasi-geodesic and	if $f: X\to Y$ is bornologous, then $f$ is large-scale Lipschitz.
\end{proposition}
\begin{proof}
	Suppose $X$ is $(\ell,r)$-quasi-geodesic but $f$ does not satisfy~\eqref{eq67f620da} for every $L,C>0$.
	Then, for every $n\in\N$ there are $p_n,q_n\in X$ such that
	\begin{equation}\label{eq67f62298}
		d(f(p_n),f(q_n)) > n d(p_n,q_n) + n .
	\end{equation}
	
	Since $X$ is $(\ell,r)$-geodesic, 
	for each $n\in\N$ there are $k_n\in\N$ and
	$x_n^0 = p_n, x_n^1,\dots,x_n^{k_n} = q_n$ such that $d(x_n^{j-1},x_n^j) \le r$
	for all $j\in\{1,\dots,k_n\}$,
	and $\sum_{j=1}^{k_n} d(x_n^{j-1},x_n^j) \le \ell d(p_n,q_n)$.
	
	Since $f$ is bornologous and by the assumption~\eqref{eq67f62298}, we know that $d(p_n,q_n)\to\infty$ as $n\to\infty$: in particular, up to passing to a subsequence, we can assume $d(p_n,q_n) \ge r$ for all $n$.
	
	If, for some $j$, we have $d(x_n^{j-1},x_n^j)<r/2$, then we can skip a few points and still obtain a chain of points so that
	$r/2\le d(x_n^{j-1},x_n^j) \le 3r$ for all $j\in\{1,\dots,k_n\}$.
	It follows that we can assume
	\begin{equation}\label{eq67f62053}
		\ell d(p_n,q_n) \ge \frac{k_n r}{2} .
	\end{equation}
	Further, by the triangle inequality,
	\begin{equation}
		\sum_{j=1}^{k_n} d\left(f(x_n^{j-1}),f(x_n^j)\right)
		\ge d(f(p_n),f(q_n)) .
	\end{equation}
	Consequently, for each $n$ there exists $j_n$ such that $d(f(x_n^{j_n-1}),f(x_n^{j_n})) \ge \frac{d(f(p_n),f(q_n))}{k_n}$.
	For such points, we have
	\begin{equation}
		d(f(x_n^{j_n-1}),f(x_n^{j_n}))
		\ge \frac{d(f(p_n),f(q_n))}{k_n} 
		\overset{\eqref{eq67f62298}}> n \frac{d(p_n,q_n)}{k_n} + \frac{n}{k_n} 
		\overset{\eqref{eq67f62053}}\ge \frac{rn}{2\ell} + \frac{n}{k_n} .
	\end{equation}
	We conclude that 
	$\lim_{n\to\infty}d(f(x_n^{j_n-1}),f(x_n^{j_n})) = \infty$
	even though $d(x_n^{j_n-1},x_n^{j_n}) \le 3r$,
	in contradiction with $f$ being bornologous.
\end{proof}

\begin{corollary}\label{cor67f68188}
	A bijection between quasi-geodesic metric spaces that is bornologous and with inverse bornologous, is a quasi-isometry.
\end{corollary}

\section{Monotone functions}

\subsection{Topological notions}
Given a subset $Y$ of a topological space $X$, we denote by 
by $\closure(Y)$ the closure of $Y$ within $X$,
by $\interior Y$ the interior of $Y$ within $X$, and 
by $\de Y$ the topological boundary of $Y$ within $X$.

\begin{definition}
	Let $X$ be a topological space.
	A subset $K\subset X$ is a \emph{continuum} if it is compact and connected. 
	A subset $C\subset X$ is an \emph{$\infty$-continuum} if it is closed and every connected component of $C$ is not compact.
	Equivalently, a closed non-compact set $C\subset X$ is an {$\infty$-continuum} if the closure of $C$ in the Alexandrov 1-point compactification of $X$ is a continuum.
\end{definition}

In \cite{zbMATH00883879}, Ferrand called $\infty$-continua with the name ``r-continuum'', where ``r'' stands for ``relative''.
From Ferrand's work, we isolated the following remark, which will be useful to study monotone functions.

\begin{lemma}\label{lem67ed2a8f}
	Let $X$ be a Hausdorff topological space and $C\subset X$ an $\infty$-continuum.
	If $Y\subset X$ is compact with $\interior Y\cap C\neq\emptyset$, then $\de Y\cap C\neq\emptyset$.
\end{lemma}
\begin{proof}
	Up to passing to a connected component of $C$ that contains a point of $\interior Y\cap C$, we can assume that $C$ is connected.
	Define $Y_1 = \interior Y\cap C$ and $Y_2 = C \setminus Y = C\cap (X\setminus Y)$.
	Since $Y$ is a compact subset of a Hausdorff space, then $Y$ is closed;
	thus, $Y_1$ and $Y_2$ are both open in $C$.
	The set $Y_1$ is not empty by assumption.
	The set $Y_2$ is not empty, because otherwise $C=C\cap Y$ would be a closed set in a compact set, thus compact; this is impossible since $\infty$-continua are assumed not compact.
	Moreover, $Y_1\cap Y_2=\emptyset$.
	Since $C$ is connected, then $C\neq Y_1\cup Y_2$, i.e., there exists $c\in C\setminus(Y_1\cup Y_2)$.
	This means that $c\in \de Y$,
	because $X = \interior Y \sqcup \de Y \sqcup (X\setminus Y)$.
\end{proof}

\subsection{Monotone functions on topological spaces}\label{subs68556338}

\begin{definition}
	Let $X$ be a topological space.
	A function $u:X\to\R$ is \emph{monotone} if, for every compact $K\subset X$,
	\begin{equation}\label{eq67ed2b45}
		\sup_{\de K} u = \sup_K u
		\quad\text{ and }\quad
		\inf_{\de K} u = \inf_K u .
	\end{equation}
\end{definition}

Further references for monotone functions are \cite{MR0236383,zbMATH00883879,zbMATH03431466}.
Notice that, since $\sup\emptyset=-\infty$, if there exists $K\subset X$ compact with $\de K=\emptyset$, then there are no monotone functions $X\to\R$.

The following two lemmas are due to Ferrand, who proved the analog for Riemannian manifolds.
For completeness, we include her argument showing the validity for topological spaces.

\begin{lemma}[Ferrand]\label{lem67e51b03}
	Let $X$ be a Hausdorff topological space.
	Let $C\subset X$ be an $\infty$-continuum and $u:X\to\R$ continuous.
	Suppose that $u$ is monotone in $X\setminus C$ and that $C\subset\{u=\sup_Xu\}\cup\{u=\inf_Xu\}$.
	Then, the function $u$ is monotone on $X$.
\end{lemma}
\begin{proof}
	The argument for this proof is inspired by \cite[(2.4)]{zbMATH00883879}.
	We start considering the case $C\subset\{u=\sup_Xu\}$.
	Set $M:=\sup_Xu$ and let $K\subset X$ compact.
	
	First, assume $\de K\cap C = \emptyset$.
	Then $K\cap C = \emptyset$ by the contrapositive of Lemma~\ref{lem67ed2a8f}.
	So, we have the identities~\eqref{eq67ed2b45} on $K$ because $u$ is monotone on $X\setminus C$.
	
	Second, assume $\de K\cap C \neq\emptyset$.
	Since $M = \sup_X u \ge \sup_K u \ge \sup_{\de K}u \ge \sup_{\de K\cap C} u = M$, then we have $\sup_Ku = \sup_{\de K}u$.
	
	We need to show that $\inf_Ku = \inf_{\de K}u$.
	If $\inf_Ku = M$, then $K\subset\{u=M\}$ and thus $\inf_{\de K}u = M = \inf_Ku$.
	So, suppose that $\inf_Ku < M-\epsilon$ for some $\epsilon>0$.
	Define $K_\epsilon := K\cap\{u\le M-\epsilon\}$.
	Since $u$ is continuous, then $K_\epsilon$ is the intersection of a compact set and a closed set, so $K_\epsilon$ is compact.
	Since $\de K_\epsilon\subset \{u\le M-\epsilon\}$ and $\{u\le M-\epsilon\}\cap C = \emptyset$,
	then $\de K_\epsilon \cap C = \emptyset$ and thus, by the first part of the proof, 
	$\inf_K u = \inf_{K_\epsilon} u = \inf_{\de K_\epsilon} u$.
	Let $x\in \de {K_\epsilon}$ with $u(x) = \inf_{\de K_\epsilon} u$.
	If $x\in\interior(K)$, then 
	$x\in \interior(K)\cap\{u<M-\epsilon\}\subset \interior(K_\epsilon)$: contradiction.
	Thus, $x\in\de K$ and we conclude that $\inf_Ku = \inf_{\de K} u$.
	
	We have shown the lemma whenever $C\subset\{u=\sup_Xu\}$.
	Clearly, we directly obtain the same statement for $C\subset\{u=\inf_Xu\}$,
	by substituting $u$ with $-u$.
	If $C\subset\{u=\sup_Xu\}\cup\{u=\inf_Xu\}$, then $C=C_{\rm s}\sqcup C_{\rm i}$ with 
	$C_{\rm s}:= C\cap\{u=\sup_Xu\}$ and $C_{\rm i}:=C\cap\{u=\inf_Xu\}$.
	Then, if a continuous function $u:X\to\R$ is monotone in $X\setminus C = (X\setminus C_{\rm i})\setminus C_{\rm s}$,
	then it is monotone on $X\setminus C_{\rm i}$ be the first part of the proof, so $u$ is monotone on $X$ by the subsequent observation.
\end{proof}

\begin{lemma}[Ferrand]\label{lem67e51b8c}
	Let $X$ be a Hausdorff topological space.
	Let $K\subset X$ be a continuum and $u:X\to\R$ continuous such that $K\subset\{u=\sup_Xu\}$ or $K\subset\{u=\inf_Xu\}$.
	If $u$ is monotone in $X\setminus K$, then, for every $x\in K$, the function $u$ is monotone in $X\setminus\{x\}$.
\end{lemma}
\begin{proof}
	The argument is taken from \cite[(2.5)]{zbMATH00883879}.
	We claim that, in the topological space $X\setminus\{x\}$, the set $K\setminus \{x\}$ is an $\infty$-continuum.
	Indeed, let $C$ be a connected component of $ K\setminus\{x\}$.
	Since $X$ is Hausdorff, then $X\setminus\{x\}$ is open and thus $C$, which is open in $ K\setminus\{x\}$, is open in $K$.
	If $C$ were compact, it would be closed in $K$ because $K$ is Hausdorff: since $K$ is connected, we would conclude $C=K$, but $x\in K\setminus C$.
	Therefore, we conclude that $C$ is not compact, and this shows the claim.
	The statement finally follows from Lemma~\ref{lem67e51b03}.	
\end{proof}

\subsection{Estimate of the oscillation for monotone functions}\label{subs67ef8b4f}
\newcommand{\oscillation}{\operatorname{osc}}
\begin{definition}
	Given a space $X$,	define the \emph{oscillation} of a function $u:X\to \R$ over a set $D\subset X$ as
	\begin{equation}
		\oscillation_D u := \sup_D u - \inf_D u = {\rm diam}(u(D)) .
	\end{equation}
\end{definition}

We will prove the following result, which is an adaptation of~\cite[Theorem~7.2]{MR1683160}.

\begin{theorem}[after Hajłasz and Koskela]\label{thm685442cd}
    Let $(X,d,\mu)$ be a  metric measure space with uniformly locally $Q$-bounded geometry with data $(Q,R,C_{\rm A}, C_{\rm P},\sigma)$ and with $Q>1$.
    There is a constant $C_1\ge 1$ such that the following holds.
    For every $x_0\in X$, for every $r\in(0,R/C_1)$, for every $u:B(x_0,C_1r)\to\R$ continuous with upper gradient $g:B(x_0,C_1r)\to[0,+\infty]$,
    there is $t\in[r/2,r]$ such that for every $x_1,x_2\in X$ with $d(x_0,x_1)=d(x_0,x_2)=t$
    \begin{equation}\label{eq68544262}
        |u(x_1)-u(x_2)| \le C_1 \left( \frac{d(x_1,x_2)}{r} \right)^{1/Q}
        \left(\int_{B(x_0,C_1r)\setminus B(x_0,r/C_1)} g^Q \did\mu \right)^{\frac1Q} .
    \end{equation}
\end{theorem}

Before the proof of Theorem~\ref{thm685442cd}, we state a special case of \cite[Theorem~7.2]{MR1683160} (for continuous $u$, with $Q=p=s$), and then discuss the adaptations needed to deduce Theorem \ref{thm685442cd}.

\begin{theorem}[Theorem 7.2 in \cite{MR1683160}]\label{thmHK7-2}
	Let $(X,d,\mu)$ be a metric measure space, and $Q>1$. 
	Assume that $\mu$ is a doubling measure with the property that there is a constant $C_{\rm b}$ such that
	\begin{equation}\label{eq:DoublingLower}
	    \mu(B(z,s))\ge C_{\rm b} \mu(B(z_0,s_0)) \left(\frac{s}{s_0}\right)^Q,
	\end{equation}
	for all $z_0 \in X$, $s_0>0$, $z\in B(z_0,s_0)$, and $0<s\le s_0$. Let $u:X\to\R$ be a continuous function and $g:X\to [0,+\infty]$ a measurable function such that there exist constants $C_{\rm P}>0$ and $\sigma\ge 1$ with
	\begin{equation}\label{eq:PIPair}
		\fint_{B(z,s)} |u-u_{B(z,s)}| \did\mu \le C_{\rm P} s \left( \fint_{B(z,\sigma s)} g^Q \did\mu\right)^{1/Q} ,
		\qquad\forall z\in X,\ \forall s>0 .
    \end{equation}	
	Fix $x_0\in X$ and $r_0>0$.
	Assume that there is a constant $C$ such that every pair of points $x,y\in B(x_0,r_0)\setminus B(x_0,r_0/2)$ can be joined by a continuum $F$ in $B(x_0,C r_0)\setminus B(x_0,r_0/C)$ with $\operatorname{diam}(F)\le C d(x,y)$. 
	Then there exist a radius $r\in [r_0/2,r_0]$ and a constant $C_1$, depending only on $Q$, $C_{\rm b}$, $C_{\rm P}$, and the doubling constant of $\mu$, such that
	\begin{equation}
	    |u(x_1)-u(x_2)| \le C_1 d(x_1,x_2)^{\frac{1}{Q}}\,r_0^{1-\frac{1}{Q}} \left(\fint_{B(x_0,C_1r_0)\setminus B(x_0,r_0/C_1)} g^Q \did\mu \right)^{\frac1Q} .
	\end{equation}
	for every $x_1,x_2\in X$ with $d(x_0,x_1)=d(x_0,x_2)=r$.
\end{theorem}

\begin{remark}\label{r:HaKoLocal}  
	Theorem \ref{thmHK7-2} is local in nature. 
	Since the conclusion concerns a fixed ball centered at $x_0$, the assumptions can be suitably localized, as we now discuss. 
	
	First, an inspection of the proofs of Theorems 7.1 and 7.2 in \cite{MR1683160} reveals that the doubling property of $\mu$ and the mass bound \eqref{eq:DoublingLower} are only applied to balls centered in $B(x_0, C r_0)$, where $C\ge 2$ is a constant for which the connectivity assumption holds. 
	The radii of the relevant balls all have upper bounds of the form $C(\sigma)r_0$.
	Thus,  if we require the doubling property of $\mu$ and the mass bound~\eqref{eq:DoublingLower} to hold only for balls up to a certain radius $R_0$, the conclusion of Theorem \ref{thmHK7-2} still remains valid provided that we assume that $r_0 < R_0/C_1$, for a large enough constant $C_1$.
	
	Similarly, also the Poincaré inequality~\eqref{eq:PIPair} is only applied with balls of radii at most $2r_0$ on the left-hand side of the inequality, and with centers in $B(x_0, C r_0)$. 
	Therefore, for the conclusion of Theorem~\ref{thmHK7-2}, the values of $u$ and $g$ will only be relevant in a ball $B(x_0, C_1 r_0)$ for a sufficiently large $C_1$, depending on $C$. 
	Moreover, if we only require~\eqref{eq:PIPair} to hold for radii $s\le R_0/\sigma$, the conclusion of the theorem remains valid provided that  $r_0 < R_0/C_1$ for large enough $C_1$. 
\end{remark}

Theorem \ref{thm685442cd} is an application of Theorem \ref{thmHK7-2}, with localized assumptions as explained in Remark \ref{r:HaKoLocal}. 
The uniform locally $Q$-bounded geometry of $(X,d,\mu)$ and a local version of \cite[Proposition 4.5]{MR1683160} will ensure that the connectivity requirement in Theorem \ref{thmHK7-2} is met.
 
\begin{proof}[Proof of Theorem~\ref{thm685442cd}] 
	According to Proposition \ref{prop686fb4ae-uniform},
	the metric space $(X,d)$ is uniformly locally quasi-convex (or quasi-convex, if $R=\infty$) with constants depending only on the data of $(X,d,\mu)$.

	Assume now that $x_0\in X$ and $r\in (0, R/C_1)$ for a large enough constant $C_1\ge 1$ to be determined depending only on the data of $(X,d,\mu)$. 
	In particular, we may, and will, assume that $C_1$ is so large that the quasi-convexity condition holds for points in $B(x_0,r)$.

	Using verbatim the proof of \cite[Proposition 4.5]{MR1683160}, one can then further choose $C \ge 4$ large enough, depending only on the data of $(X,d,\mu)$,
	such that the following \emph{annular quasiconvexity} condition holds.
	If  $x,y\in B(x_0,r)\setminus B(x_0,r/2)$, then $x,y$ are joined in $B(x_0,Cr)\setminus B(x_0,r/C)$ by a curve whose length does not exceed $C d(x,y)$. 

	This yields one of the assumptions needed for the application of Theorem~\ref{thmHK7-2}. 
	Regarding the other assumptions, we note that the local Ahlfors regularity implies that the measure $\mu$ is uniformly locally doubling and condition~\eqref{eq:DoublingLower} holds uniformly for small enough radii, depending on the data of $(X,d,\mu)$. 
	Similarly,~\eqref{eq:PIPair} holds for $u$ and $g$ on $B(x_0,C_1 r)$ as in the assumptions of Theorem~\ref{thm685442cd}.

	It follows by Theorem~\ref{thmHK7-2} and Remark~\ref{r:HaKoLocal} that we can choose $C_1$ large enough, depending only on the data of $(X,d,\mu)$, so that there is a radius $t\in [r/2,r]$, with the property that 
	\begin{equation}
    |u(x_1)-u(x_2)| \le C_1 d(x_1,x_2)^{\frac{1}{Q}} r^{1-\frac{1}{Q}} \left(\fint_{B(x_0,C_1r)\setminus B(x_0,r/C_1)} g^Q \did\mu \right)^{\frac1Q} .
	\end{equation}
	for every $x_1,x_2\in X$ with $d(x_0,x_1)=d(x_0,x_2)=t$.
 
	To conclude the proof, it remains to observe that $\mu(B(x_0, C_1 r)\setminus B(x_0,r/C_1))$ can be controlled from below in terms of $r^Q$, so that the average integral on the right-hand side can be rewritten to obtain the desired inequality. 
	To bound the measure of the annulus from below, we note that the path-connectedness of $(X,d)$ ensures that $B(x_0,r)\setminus B(x_0,r/2)$ is non-empty. 
	Therefore, the annulus $B(x_0,C_1 r )\setminus B(x_0,r/C_1 )$ contains a small ball of radius comparable to $r$. 
	The claim then follows from the uniform local Ahlfors regularity assumption of $\mu$.
\end{proof}

The following proposition is a generalization of Ferrand's result \cite[(1.4)]{zbMATH00883879} to metric measure spaces.

\begin{theorem}\label{thm68545111}
	Let $(X,d,\mu)$ be a metric measure space with uniformly locally $Q$-bounded geometry with data $(Q,R,C_{\rm A}, C_{\rm P},\sigma)$, with $Q>1$.
	There is a constant $C\ge1$ such that, for every $x_0\in X$ and $r_0\in (0,R/C)$,
	and for every $u:B(x_0, Cr_0)\to\R$ monotone with upper gradient $g:B(x_0,Cr_0)\to[0,+\infty]$, 
	\begin{equation}\label{eq685442fd}
		\int_0^{r_0} \left( \oscillation_{S(x_0,t)} u \right)^Q \frac{\did t}{t} 
		\le C \int_{B(x_0, C r_0)} g^Q \did \mu .
	\end{equation}
	In particular, if $r\in(0, r_0]$, then
	\begin{equation}\label{eq685451d6}
		\oscillation_{S(x_0,r)} u 
		\le C (\log(r_0/r))^{-\frac1Q} \| g \|_{L^Q(B(x_0,Cr_0))} ,
	\end{equation}
	where $S(x_0,r) := \{y\in X:d(x_0,y)=r\}$.
\end{theorem}
\begin{proof}
	For $x_0\in X$ and $r,s>0$, define the annulus
	\begin{equation}
		A(x_0;r,s) := \{y\in X: r\le d(x_0,y) < s \} .
	\end{equation}
	
	Let $C_1$ be the constant given by Theorem~\ref{thm685442cd}.
	Let $r_0>0$ be such that $C_1r_0<R$ and let $u:B(x_0,C_1r_0)\to\R$ be monotone with upper gradient $g:B(x_0,C_1r_0)\to[0,+\infty]$.
	
	By Theorem~\ref{thm685442cd},
	for every $r\in(0,r_0]$,
	there is $t\in[r/2,r]$ such that,
	for every $x_1,x_2\in S(x_0,t)$
	\begin{equation}
	\begin{aligned}
		|u(x_1)-u(x_2)| 
		&\le C_1 \left( \frac{d(x_1,x_2)}{r} \right)^{1/Q} 
		\left(\int_{ A(x_0;r/C_1,rC_1) } g^Q \did\mu \right)^{\frac1Q} \\
		&\le C_1 2^{1/Q} \left(\int_{ A(x_0;r/C_1,rC_1) } g^Q \did\mu \right)^{\frac1Q} .
	\end{aligned}
	\end{equation}
	It follows that for every $r\in(0,r_0]$ there is $t\in[r/2,r]$ such that
	\begin{equation}\label{eq68544b4d}
		\oscillation_{S(x_0,r/2)}u
		\overset{({\rm m})}\le \oscillation_{S(x_0,t)}u
		\le C_1 2^{1/Q} \left(\int_{ A(x_0;r/C_1,rC_1) } g^Q \did\mu \right)^{\frac1Q} ,
	\end{equation}
	where in the first inequality, and later in the proof, we mark with $({\rm m})$
	bounds that are consequences of the monotonicity of $u$ and the fact that
	$\de B(x_0,r) \subset S(x_0,r)$, for all $x_0\in X$ and $r>0$.
	We then estimate
	\begin{align}
		\int_0^{r_0} \left( \oscillation_{S(x_0,t)} u \right)^Q \frac{\did t}{t} 
		&= \sum_{k=0}^\infty \int_{2^{-k-1}r_0}^{2^{-k}r_0} \left( \oscillation_{S(x_0,t)} u \right)^Q \frac{\did t}{t} \\
		&\overset{({\rm m})}\le \sum_{k=0}^\infty \int_{2^{-k-1}r_0}^{2^{-k}r_0} \left( \oscillation_{S(x_0,2^{-k}r_0)} u \right)^Q \frac{\did t}{2^{-k-1}r_0} \\
		&= \sum_{k=0}^\infty \frac{2^{-k-1}r_0}{2^{-k-1}r_0} \left( \oscillation_{S(x_0,2^{-k}r_0)} u \right)^Q \\
		&\overset{\eqref{eq68544b4d}}\le \sum_{k=0}^\infty C_1^Q 2 \int_{ A(x_0;2^{-k+1}r_0/C_1,2^{-k+1}r_0C_1) } g^Q \did\mu .
	\end{align}
	Notice that, since for every $t>0$ we have
	\begin{equation}
		\#\{k\in\N: t\in[ 2^{-k}r_0/C_1 , 2^{-k}r_0C_1 ] \}
		\le 2\log_2(C_1) + 1 ,
	\end{equation}
	the family of annuli $\{A(x_0;2^{-k+1}r_0/C_1,2^{-k+1}r_0C_1)\}_{k\in\N}$ has multiplicity bounded by $2\log_2(C_1) + 1$.
	Therefore, 
	\begin{align}
		\int_0^{r_0} \left( \oscillation_{S(x_0,t)} u \right)^Q \frac{\did t}{t}
		&\le \sum_{k=0}^\infty C_1^Q 2 \int_{ A(x_0;2^{-k+1}r_0/C_1,2^{-k+1}r_0C_1) } g^Q \did\mu \\
		&\le C_1^Q 2 (2\log_2(C_1) + 1) \int_{ B(x_0,2C_1r_0) } g^Q \did\mu .
	\end{align}
	We have thus proven~\eqref{eq685442fd} with $C := C_1^Q 2 (2\log_2(C_1) + 1)$, 
	which is larger than $C_1$ and $2C_1$.
	For~\eqref{eq685451d6}, we simply estimate
	\begin{align}
		\left( \oscillation_{S(x_0,r)} u  \right)^Q \log(r_0/r)
		&= \left( \oscillation_{S(x_0,r)} u  \right)^Q \int_r^{r_0} \frac{\did t}{t} \\
		&\overset{({\rm m})}\le \int_r^{r_0} \left( \oscillation_{S(x_0,t)} u  \right)^Q \frac{\did t}{t} \\
		&\overset{\eqref{eq685442fd}}\le C \| g \|_{L^Q(B(x_0,Cr_0))}^Q .
	\end{align}
	So, we obtain~\eqref{eq685451d6}.
\end{proof}

\begin{corollary}\label{cor67efffbc}
	Let $(X,d,\mu)$ be a metric measure space with uniformly locally $Q$-bounded geometry with data $(Q,R,C_{\rm A}, C_{\rm P},\sigma)$, with $Q>1$.
	There is a constant $C\ge1$ such that, for every $x_0\in X$ and $r\in(0,R/C)$,	and for every $u:B(x_0,Cr)\to\R$ monotone with upper gradient $g:B(x_0,Cr)\to[0,+\infty]$, 
	\begin{equation}\label{eq67efd746}
		\oscillation_{B(x,r)} u 
		\le C \left( \int_{B(x,Cr)} g^Q \did \mu \right)^{1/Q} .
	\end{equation}
\end{corollary}
\begin{proof}
	If $C$ is the constant from Theorem~\ref{thm68545111}, take the constant $2C$ and apply~\eqref{eq685451d6} to $r:=r_0/2$, so that $\log(r_0/r)=\log(2)$.
\end{proof}

\subsection{Lebesgue Straightening Lemma}
The aim of this subsection is Proposition~\ref{prop67e3cff4}, which is an integration of the Lebesgue Straightening Lemma as proven by Mostow in~\cite{MR0236383} with a statement about upper-gradients.
For this, we need to revise the notions introduced by Mostow, so as to prove Lemma~\ref{lem6803e6c4}.
Ferrand proved Proposition~\ref{prop67e3cff4} in the context of Riemannian manifolds; see \cite[\S3.5]{zbMATH03431466} and \cite[\S1.8]{zbMATH00883879}.

Let $X$ be a topological space, $U\subset X$ open, and $u:U\to\R$ continuous.
For $a\in\R$, denote $aUu$ the union of all connected components of $U\setminus u^{-1}(a)$ whose closures lie in $U$.
For $a\in\R$, define the function $u.a:U\to\R$,
\begin{equation}\label{eq6803e67d}
	(u.a)(x) := \begin{cases}
		a &\text{ if }x\in aUu ,\\
		u(x) &\text{ if }x\notin aUu .
	\end{cases}
\end{equation}

\begin{lemma}\label{lem6803e6c4}
	Let $X$ be a metric space, $u:X\to\R$ continuous, and $U\subset X$ open.
	Let $a\in\R$ and $(u.a):X\to\R$ be the function modified as in~\eqref{eq6803e67d} on $U$.
	If $g:X\to[0,+\infty]$ is an upper gradient for $u$, then $g$ is an upper gradient also for~$(u.a)$. 
\end{lemma}
\begin{proof}
	Let $\gamma:[0,1]\to X$ be a rectifiable curve.
	We need to show~\eqref{eq6804a338} for $u.a$, that is,
	\begin{equation}\label{eq6804a3e8}
		|(u.a)(\gamma(1)) - (u.a)(\gamma(0))| \le \int_\gamma g .
	\end{equation}
	We have three cases: 
	If $\gamma(1),\gamma(0)\in aUu$, then~\eqref{eq6804a3e8} is trivial, because the left-hand side of~\eqref{eq6804a3e8} is zero.
	If $\gamma(1),\gamma(0)\notin aUu$, then~\eqref{eq6804a3e8} is trivial, because $g$ is an upper gradient of $u$.
	Up to reversing the direction of $\gamma$, the only remaining case is $\gamma(0)\in aUu$ and $\gamma(1)\notin aUu$.
	Then, there is $t\in[0,1]$ such that $\gamma(t) \in \de(aUu)$, so that $(u.a)(\gamma(0)) = a = u(\gamma(t))$.
	Therefore,
	\[
	|(u.a)(\gamma(1)) - (u.a)(\gamma(0))|
	= |u(\gamma(1)) - u(\gamma(t))|
	\le \int_{\gamma|_{[t,1]}} g
	\le \int_\gamma g .
	\]
\end{proof}

Mostow proved the following statement.

\begin{lemma}[Lebesgue Straightening Lemma -- {\cite[p.67]{MR0236383}}]\label{lem6803e50c}
	Let $U$ be an open proper subset of a connected locally connected topological space and let $u:\closure U\to [m, M]$ be a continuous function, for some $m, M\in\R$.
	Let $a:\N\to\Q\cap[m, M]$ be an enumeration of the rational numbers in $[m, M]$.
	For every $n\in\N$, set
	\begin{equation}\label{eq6803e602}
	u_n = (\dots ((u.a_1).a_2)\dots).a_n .
	\end{equation}
	Then $u_n$ converges uniformly on $\closure U$ to a monotone function.
\end{lemma}

Lemma~\ref{lem6803e50c} can be strengthened by keeping track of the upper gradients of the functions $u_n$.

\begin{proposition}[Lebesgue Straightening Lemma -- version 2]\label{prop67e3cff4}
	Let $(X,d,\mu)$ be a connected locally connected metric measure space.
	Let $U\subset X$ be open. Let $u: X\to\R$ be continuous, and bounded on $U$.
	Then there exists $v: X\to\R$ continuous such that
	\begin{enumerate}
	\item
	$v$ is monotone in $U$;
	\item
	$u=v$ on $X\setminus U$;
	\item
	upper gradients of $u$ are upper gradients of $v$ too.
	\end{enumerate}
\end{proposition}
\begin{proof}
	Let $v$ be the function given by Lemma~\ref{lem6803e50c}.
	Then $v$ is monotone in $U$ and $u=v$ on $X\setminus U$.
	We claim that, if $g: X\to[0,+\infty]$ is an upper gradient for $u$, then $g$ is an upper gradient also for $v$.
	Indeed, by Lemma~\ref{lem6803e6c4}, the function $g$ is an upper gradient for each approximating function $u_n$ defined in~\eqref{eq6803e602}.
	Since $u_n$ converges uniformly to $v$ on $X$, then $g$ is an upper gradient of $v$ as well.
	This proves the last part of the proposition.	
\end{proof}

Using Proposition~\ref{prop67e3cff4}, we show that the capacity can be estimated using monotone admissible functions.
For this reason, define for $E\subset X$ and $F\subset X$ or $F=\infty$,
\begin{equation}
	\cal A_{\rm m}(E;F) := \{ u\in \cal A(E;F) : u\text{ monotone on }X\setminus (E\cup F) \} .
\end{equation}
\begin{proposition}\label{prop67e505c3}
	Let $(X,d,\mu)$ be a metric measure space and $p\in[1,+\infty)$.
	
	For every $E,F\subset X$ closed, or $E\subset X$ closed and $F=\infty$,
	\begin{equation}
		\capacity_p(E;F) = \inf\left\{
		\int_X g^p \did\mu : 
		\text{$g$ upper gradient of $u\in\cal A_{\rm m}(E;F)$}
		\right\} .
	\end{equation}
\end{proposition}
\begin{proof}
	Let $u\in\cal A(E;F)$.
	Up to substituting $u$ with $x\mapsto \min\{ 1, \max\{ 0, u(x)\}\}$, which is an element of $\cal A(E;F)$ with possible smaller upper gradients,
	we assume $u:X\to[0,1]$.
	Let $U = X\setminus (E\cup F)$.
	We apply Proposition~\ref{prop67e3cff4} to the function $u$
	on the open set $U = X\setminus (E\cup F)$.
	We get a function $v\in\cal A_{\rm m}(E;F)$ so that upper gradients of $u$ are also upper gradients of $v$.
\end{proof}

\begin{remark}\label{rem68e4b268}
	Among the properties of quasi-conformal maps, we will strongly use two properties:
	First, they bi-Lipschitz preserve the capacity as by \eqref{bi-Lipschitz preserve the capacity}.
Second, they are homeomorphisms, and hence they preserve the topology. 
	 	In particular, quasi-conformal maps preserve monotone functions.
	For this reason, we could define the capacity using only monotone admissible functions, and consider maps that preserve this reduced capacity.
	Proposition~\ref{prop67e505c3} ensures that this choice is not restrictive.
\end{remark}

\section{Conformal type}

\subsection{Parabolicity and hyperbolicity}\label{subs67c86ab7}

For more information about the following notions, we suggest Troyanov's work \cite[p.20]{MR1749853}.
\begin{definition}\label{def687743ae}
	Let $X$ be a metric measure space and $p\in[1,\infty)$.
	We say that $X$ is \emph{$p$-parabolic} if $\capacity_p(E)=0$ for all bounded $E\subset X$.
	We say that a metric space $X$ is \emph{$p$-hyperbolic} if it not $p$-parabolic.
	The \emph{parabolic dimension} of a metric measure space $X$ is
	\begin{equation}
		\dimpar(X) 
		:= \inf\{p\in[1,\infty):\text{ $X$ is $p$-parabolic}\} .
	\end{equation}
	A metric measure space $X$ with Hausdorff dimension $Q$ is of \emph{parabolic conformal type} if $X$ is $Q$-parabolic, that is, if $\capacity_Q(E)=0$ for all bounded $E\subset X$.
	A metric measure space $X$ is of \emph{hyperbolic conformal type} if it is not of parabolic type, that is, if there exists $E\subset X$ bounded with $\capacity_Q(E)>0$.
\end{definition}

\begin{remark}\label{rem689ee744}
	By the geometric Definition~\ref{def68cc0a5c} of QC map, with index $Q$ equal to the Hausdorff dimension, the conformal type is a quasi-conformal invariant;
see also~\cite[Theorem 4.5]{zbMATH01943361}.
\end{remark}

\subsection{Elementary estimate of the capacity of a ball}\label{subs67f8e0f2}

The following Lemma~\ref{lem67f8cd3b} is a standard exercise, and we will use this result for capacity estimates.
We include the proof for the sake of completeness.

\begin{lemma}\label{lem67f8cd3b}
	Let $(X,d,\mu)$ be a metric measure space and $o\in X$.
	Let $0<\bar r<\bar R<\infty$ and suppose that there are $q>1$ and $C>0$ such that
	\begin{equation}\label{eq67f8cd13}
		\mu(B(o,r)) \le C r^q ,
		\qquad \forall r\in[\bar r,\bar R] .
	\end{equation}
	Then
	\begin{align*}
		\capacity_q(B(o,\bar r) ; X \setminus B(o,\bar R) ) 
			&\le C \frac{1+ q(\log(\bar R)-\log(\bar r)) }{(\log(\bar R) - \log(\bar r))^q }  
			&& \text{ and} \\
		\capacity_p(B(o,\bar r) ; X \setminus B(o,\bar R) ) 
			&\le C \frac{ \left( - \frac{q}{p-q} \bar R^{q-p}  + \frac{p}{p-q} \bar r^{q-p} \right)  }{ (\log(\bar R) - \log(\bar r))^p  }  ,
			&& \forall p\in[1,+\infty)\setminus\{q\} .
	\end{align*}
\end{lemma}
\begin{proof}
	Define the function
	\begin{equation}
		u(x) := 
		\begin{cases}
			1 & \text{ if }x\in B(o,\bar r) , \\
			\frac{\log \bar R - \log(d(o,x)) }{ \log \bar R - \log \bar r } & \text{ if }x\in B(o,\bar R)\setminus B(o,\bar r) , \\
			0 & \text{ if }x\in X\setminus B(o,\bar R) .
		\end{cases}
	\end{equation}
	The function $u$ is an admissible function for the capacitor $B(o,\bar r) ; X \setminus B(o,\bar R)$.
	Using the chain rule and the fact that the distance function is 1-Lipschitz, 
	one can easily show that an upper gradient of $u$ is the function $g:X\to[0,+\infty]$,
		\begin{equation}
		g(x) := 
		\begin{cases}
			0 & \text{ if }x\in B(o,\bar r) , \\
			\frac{1}{ \log \bar R - \log\bar r } \frac{1}{d(o,x)} & \text{ if }x\in B(o,\bar R)\setminus B(o,\bar r) , \\
			0 & \text{ if }x\in X\setminus B(o,\bar R) .
		\end{cases}
	\end{equation}
	Set $L:= \frac{1}{ \log \bar R - \log \bar r }$ and let $p\in[1,\infty)$.
	Notice that $(L/\bar R)^p \le g^p \le (L/\bar r)^p$ on $B(o,\bar R) \setminus B(o,r)$,
	and that $\{x:g(x)^p>t\} = B(o,L/t^{1/p})$ for every $t\in [(L/\bar R)^p , (L/\bar r)^p)$.
	Therefore
	\begin{align*}
		&\int_X g(x)^p \did\mu(x) \\
		[\text{Cavalieri}]
		&= \int_0^\infty \mu(g^p>t) \did t \\
		&= \int_0^{L^p/\bar R^p} (\mu(B(o,\bar R)) - \mu(B(o,\bar r)) \did t + \int_{L^p/\bar R^p}^{L^p/\bar r^p} \mu(B(o,L/t^{1/p})) \did t \\
		[\text{by~\eqref{eq67f8cd13}}]
		&\le C L^p \bar R^{q-p} + C L^q \int_{L^p/\bar R^p}^{L^p/\bar r^p} t^{-q/p} \did t \\
		[\text{set }s^p = t] 
		&= C L^p \bar R^{q-p} + pC L^q \int_{L/\bar R}^{L/\bar r} s^{p-1-q} \did s \\
		&= C L^p \bar R^{q-p} + pC L^q \begin{cases}
			\log(\bar R/\bar r) & \text{ if }p=q  \\
			\frac{L^{p-q}}{p-q} (1/\bar r^{p-q} - 1/\bar R^{p-q}) & \text{ otherwise} 
		\end{cases} \\
		&= \begin{cases}
		C \frac{1+ q(\log(\bar R)-\log(\bar r)) }{(\log(\bar R) - \log(\bar r))^q } & \text{ if }p=q , \\
		\frac{C}{ (\log(\bar R) - \log(\bar r))^p  } \left( - \frac{q}{p-q} \bar R^{q-p}  + \frac{p}{p-q} \bar r^{q-p} \right)  & \text{ otherwise.}
		\end{cases}
	\end{align*}
\end{proof}

\subsection{Volume growth VS parabolicity}\label{subs67dbe653}
There are several results on Riemannian manifolds that link volume growth at large scale with parabolicity.
See for instance \cite[\S5.2]{zbMATH00999672}, or \cite[Cor.5.6]{MR1749853}.
The next result is a general one:

\begin{proposition}\label{prop67f0559a} 
	Suppose that $(X,d,\mu)$ is a metric measure space with
	degree of large-scale growth at most~$N\in(1,+\infty)$,
	that is, 
	there exist $C,N>1$ such that
	\begin{equation}
		\mu(B(x,R)) \le C R^N 
		\qquad\forall x\in X, \forall R>1 .
	\end{equation}
	Then $X$ is $p$-parabolic for every $p\ge N$,	and thus $\dimpar(X)\le N$.
\end{proposition}
\begin{proof}
	Let $x\in X$, $r>0$, and $p\ge N$.
	We give an upper bound to $\capacity_p(B(x,r))$
	using the function given by Lemma~\ref{lem67f8cd3b}.
	If $p=N>1$, then 
	\begin{align}
		\capacity_p(B(o,r))
		\le \liminf_{R\to \infty} C \frac{1+ N(\log(R)-\log(r)) }{(\log(R) - \log(r))^N }  
		= 0.
	\end{align} 
	If $p>N$, i.e., $N-p<0$, then 
	\begin{align}
		\capacity_p(B(o,r))
		\le \liminf_{R\to \infty} C \frac{ \left( R^{N-p} \frac{N}{N-p} - r^{N-p}\frac{p}{N-p} \right) }{ (\log(R) - \log(r))^p  }  
		= 0 .
	\end{align}
\end{proof}

\subsection{Sobolev inequality VS hyperbolicity}\label{subs6862aabb}

We give a sufficient condition for hyperbolicity of a metric measure space.
In Proposition~\ref{prop67f055d9}, we show that a Sobolev inequality at large scale implies hyperbolicity.
Since we can deduce a Sobolev inequality from the isoperimetric inequality, we will show in Corollary~\ref{cor6863f5f0} that a large-scale isoperimetric inequality implies hyperbolicity.

\begin{proposition}\label{prop67f055d9}
	Let $(X,d,\mu)$ be a metric measure space
	with compact sets of arbitrarily large measure.
	Suppose also that there are $N\ge1$, $C_N>0$, and $C_\infty\ge0$ such that, 
	for every $u:X\to\R$ measurable with upper gradient $g:X\to[0,+\infty]$,
	\begin{equation}\label{eq67f05467}
		\left( \int_X |u|^{\frac{N}{N-1}} \did\mu \right)^{\frac{N-1}{N}}
		 \le C_N \int_X g \did\mu + C_\infty \|u\|_{L^\infty} .
	\end{equation}
	Then $X$ is $p$-hyperbolic for every $p\in[1,N)$, and thus $\dimpar(X)\ge N$.
\end{proposition}
\begin{proof}
    Let $K\Subset X$.
	Let $u:X\to[0,1]$ with $K\subset\{u=1\}$, and $g:X\to[0,+\infty]$ an upper gradient of $u$.
    For $q\ge1$,
    we apply~\eqref{eq67f05467} to $u^q$,
	noticing that an upper gradient of $u^q$ is $q u^{q-1}g$.
	Therefore, for $q\in(1,+\infty)$ and $p\in(1,N)$,
    \begin{align*}
    \left( \int_X (u^q)^{\frac{N}{N-1}} \dd\mu \right)^{\frac{N-1}{N}}
    &\overset{\eqref{eq67f05467}}\le C_N \int_X q u^{q-1} g \dd\mu + C_\infty\\
    [\text{by Hölder inequality}]
    &\le C_N q \left( \int_X u^{(q-1)\frac{p}{p-1}} \dd\mu \right)^{\frac{p-1}{p}} \left(\int_X g^p \dd\mu \right)^{\frac1p} + C_\infty.
    \end{align*}
    We want $(q-1)\frac{p}{p-1} = q\frac{N}{N-1}$, which means
    \[
    q = \frac{1}{1-\frac{N(p-1)}{(N-1)p}}
    = \frac{(N-1)p}{(N-1)p - N(p-1)} 
    = \frac{(N-1)p}{N-p} .
    \]
    Since $p\in[1,N)$, then $q\ge1$.
    So, we get
    \[
    \left( \int_X u^{\frac{qN}{N-1}} \dd\mu \right)^{\frac{N-1}{N} - \frac{p-1}{p}}
    \le C_N q \left(\int_X g^p \dd\mu \right)^{\frac1p} 
    	+ C_\infty \left( \int_X u^{\frac{qN}{N-1}} \dd\mu \right)^{-\frac{p-1}{p}} .
    \]
    Since 
    \[
    \int_X u^{\frac{qN}{N-1}} \dd\mu \ge \mu(K) ,
    \]
    we obtain
    \[
    \mu(K)^{\frac{N-p}{Np}} - C_\infty \frac{1}{\mu(K)^{\frac{p-1}{p}}} 
    \le C_N q \left(\int_X g^p \dd\mu \right)^{\frac1p} .
    \]
    Since we can choose $K$ with arbitrarily large measure,
    we conclude that, for every $p\in[1,N)$, there exists $K\subset X$ compact with $\capacity_p(K)>0$.
\end{proof}

\begin{remark}
	If~\eqref{eq67f05467} holds with $C_\infty=0$, then the proof of Proposition~\ref{prop67f055d9}
	does not need compact sets of arbitrarily large measure,
	but only one compact set with positive measure.
\end{remark}

\section{The parabolic case}

\subsection{Volume growth and growth dimension}

In this subsection, we turn to volume growth conditions and their interaction with quasi-isometries.

\begin{definition}[Growth at infinity]\label{def6863fac9}
	A metric measure space $(X,d,\mu)$ has \emph{degree of large-scale growth at most~$N$}, for $N\in[0,\infty)$, if there exists $C\in\R$ such that
	\begin{equation}\label{eq67c86af8}
		\mu(B(x,R)) \le C R^N ,
		\qquad\forall x\in X,\ \forall R>1 .
	\end{equation}
	The \emph{growth dimension} of $X$ is 
	\begin{equation}
		\dimgr(X) 
		:= \inf\left\{N\ge0 : \text{$X$ has degree of large-scale growth at most~$N$}\right\} .
	\end{equation}
\end{definition}

In Proposition~\ref{prop67f0559a}, we have seen that, if a metric measure space $X$ has degree of large-scale growth at most~$N\in(1,\infty)$, then $X$ is $p$-parabolic for every $p\in[1,N]$.
In particular,
\begin{equation}\label{eq68cf049a}
	  \dimpar(X) \le \dimgr(X) .
\end{equation}

In Proposition~\ref{prop6866b133} below, we will show that, under certain geometric conditions, the growth dimension is a quasi-isometric invariant.
This is a standard exercise, and we include the proof for completeness.

\begin{proposition}\label{prop6866b133}
	Let $X$ and $Y$ be separable metric measure spaces,
	and suppose that their volume functions, defined as in~\eqref{eq6866b26f}, are valued in $(0,+\infty)$.
	Assume $X$ and $Y$ are quasi-isometric and let $N\in[1,+\infty)$.
	Then, the space $X$ has degree of large-scale growth at most~$N$ if and only if $Y$ does.
\end{proposition}

The proof is postponed after the following lemma.

\begin{lemma}\label{lem6837521d}
	Let $(X,d_X,\mu_X)$ and $(Y,d_Y,\mu_Y)$ be separable metric measure spaces.
	Let $f:X\to Y$ be a map such that there are $L,C\in\R$ with 
	\begin{equation}\label{eq6866a4ee}
		d_Y(f(a),f(b)) \le L d_X(a,b) + C ,
		\qquad \forall a,b\in X .
	\end{equation}
	For $r,t\in(0,\infty)$, define
	\begin{equation}\label{eq6837593c}
	\begin{aligned}
		\cal C_f(r,t) 
		&:= \sup\left\{ \frac{ \mu_Y(B(f(x),t)) }{ \mu_X(B(x,r)) } : x\in X \right\} .
	\end{aligned}
	\end{equation}
	Then, for every $E\subset X$ and every $r,s\in(0,\infty)$,
	\begin{equation}\label{eq68377070}
		\mu_Y(B(f(E),s)) \le \cal C_f(r, 2Lr+C+s ) \mu_X(B(E,r)) .
	\end{equation}
\end{lemma}
\begin{proof}
	Fix $E\subset X$ and $r\in(0,+\infty)$.
	Let $\scr K\subset E$ be a maximal set of points that are $2r$-separated.
	Since $X$ is separable, then $\scr K$ is countable. 
	Since the points are $2r$-separated, then
	\begin{equation}\label{eq683755ad}
		\bigsqcup_{k\in\scr K} B(k,r) 
		\subset B(E,r) .
	\end{equation}
	By maximality, we have,
	\begin{equation}\label{eq68ce54d8}
		E\subset \bigcup_{k\in\scr K} B(k,2r) .
	\end{equation}
	Hence, for every $s\ge0$ we have
	\begin{equation}\label{eq68375526}
	\begin{aligned}
		B(f(E),s)
		&\overset{\eqref{eq68ce54d8}}\subset  \bigcup_{k\in\scr K} B(f(B(k,2r)),s) 
		\overset{\eqref{eq6866a4ee}}\subset \bigcup_{k\in\scr K} B(B(f(k),2Lr+C),s) \\
		&\subset \bigcup_{k\in\scr K} B( f(k), 2Lr+C+s ) .
	\end{aligned}
	\end{equation}
	In what follows, we assume that, for all $k\in\scr K$, we have $\mu_Y( B( f(k),2Lr+C+s ) )<\infty$ and $\mu_X(B(k,r))>0$,
	because such special cases imply directly~\eqref{eq68377070} using the convention $0\cdot\infty=\infty$.
	Therefore,
	\begin{align}
		\mu_Y(B(f(E),s))
		&\overset{\eqref{eq68375526}}\le \sum_{k\in\scr K} \mu_Y( B( f(k),2Lr+C+s ) ) \\
		&= \sum_{k\in\scr K} \frac{ \mu_Y( B( f(k),2Lr+C+s ) ) }{ \mu_X(B(k,r)) } \mu_X(B(k,r)) \\
		&\le \cal C_f(r, 2Lr+C+s ) \sum_{k\in\scr K} \mu_X(B(k,r)) \\
		&\overset{\eqref{eq683755ad}}\le \cal C_f(r, 2Lr+C+s ) \mu_X(B(E,r)) . \qedhere
	\end{align}
\end{proof}

\begin{proof}[Proof of Proposition~\ref{prop6866b133}]
	Let $f:X\to Y$ be a $(L,C)$-quasi-isometry.
	It is immediate to show that, for every $x\in X$ and $R>0$,
	\begin{equation}\label{eq6866ab6e}
		B(f(x),R) \subset B(f(B(x,L(R+2C))) , C) .
	\end{equation}
	Indeed, if $z\in B(f(x),R)$, then there exists $y\in X$ with $d(f(y),z)\le C$.
	It follows that $d(x,y) \le Ld(f(x),f(y)) +LC \le L(d(f(x),z) + d(z,f(y))) + LC \le L(R + C) + LC = L(R + 2C)$.
	
	Using Lemma~\ref{lem6837521d}, we get
	\begin{equation}\label{eq6866afe4}
	\begin{aligned}
		\mu_Y(B(f(x),R)) 
		&\overset{\eqref{eq6866ab6e}}\le \mu_Y(B(f(B(x,L(R+2C))) , C)) \\
		&\overset{\eqref{eq68377070}}\le \cal C_f(1, 2L + 2C) \mu_X( B(x,L(R+2C)) ) \\
		&\le \cal C_f(1, 2L + 2C) \mu_X^+(L(R+2C)) .
	\end{aligned}
	\end{equation}
	Since the volume functions are valued in $(0,+\infty)$, $\cal C_f(1, 2L + 2C)<\infty$.
	
	If $z\in Y$ and $R>0$, then there is $x\in X$ with $d(f(x),z)\le C$.
	Hence, 
	\begin{equation}
		\mu_Y(B(z,R)) 
		\le \mu_Y(B(f(x),R+C)) 
		\overset{\eqref{eq6866afe4}}\le \cal C_f(1, 2L + 2C) \mu_X^+(L(R+3C)) .
	\end{equation}
	We conclude that 
	\begin{equation}
		\mu_Y^+(R) \le \cal C_f(1, 2L + 2C) \mu_X^+(LR + 3LC) ,
		\qquad\forall R\in(0,\infty).
	\end{equation}
	The latter estimate easily implies that if $X$ has degree of large-scale growth at most~$N\in[1,+\infty)$, then so does $Y$.
	Since quasi-isometries have quasi-inverses, the reversed implication also holds.
\end{proof}

\subsection{The parabolic Ferrand distance}

We study quasi-conformal maps between spaces of parabolic conformal type using what we call the parabolic Ferrand distance $\ferrdPar$, which we soon introduce.
This gauge-type function is defined in terms of capacitors made by pairs of $\infty$-continua.
The parabolic Ferrand distance is a power of a function,    denoted below by $\ferrg$, that has been studied by Ferrand in \cite{zbMATH03431466,zbMATH00883879}.

\begin{definition}\label{def686fc1ad}
	Let $(X,d,\mu)$ be a metric measure space with Hausdorff dimension $Q\ge1$.
	For $x,y\in X$, we define the \emph{Ferrand gauge} by
	\begin{equation}
		\ferrg(x,y) := \inf\left\{ \capacity_Q(E;F) : 
		\begin{array}{c}
		x\in E,\ y\in F,\\
		\text{$E$ and $F$ $\infty$-continua} 
		\end{array}
		 \right\} .
	\end{equation}
	For $x,y\in X$, we define the \emph{parabolic Ferrand distance} by
	\begin{equation}\label{eq68da4e38}
		\ferrdPar(x,y) 
		:= \frac1{\ferrg(x,y)^{\frac1Q} }
		= \sup\left\{ \frac{1}{\capacity_Q(E;F)^{\frac1Q}} : 
		\begin{array}{c}
		x\in E,\ y\in F,\\
		\text{$E$ and $F$ $\infty$-continua} 
		\end{array}
		 \right\} .
	\end{equation}
\end{definition}

\begin{remark}
	Despite the name, we do not claim that the function $\ferrdPar:X\times X\to [0,+\infty]$ in Definition~\ref{def686fc1ad} is a distance in complete generality.
	Ferrand proved for Riemannian manifolds in \cite[Théorème 7.1 \& \S8.2]{zbMATH03431466} that,
	if $X$ is a Riemannian manifold of dimension $n$, then $\ferrdPar:(x,y) \mapsto \frac1{\ferrg(x,y)^{1/n}}$ satisfies the triangle inequality.
	In \cite[Proposition 3.2]{zbMATH00950191}, Ferrand improved the result showing that 
	$\ferrdPar^{\frac{n}{n-1}}$ satisfies the triangle inequality.	
\end{remark}

We will show that $\ferrdPar$ is bi-Lipschitz preserved by quasi-conformal maps,
and then, under certain conditions on the space, we will give an upper bound for $\ferrdPar$
in Proposition~\ref{prop68720f23} and later a lower bound to $\ferrdPar$ as a consequence of the upper bounds in Proposition~\ref{prop68720d05}.
The final result will be to show that, under specific conditions on the spaces, quasi-conformal maps are bornologous; see Theorem~\ref{thm68755ab4}.

\begin{proposition}\label{prop68755c04}
	Let $X$ and $Y$ be metric measure spaces with locally $Q$-bounded geometry,
	and let $f:X\to Y$ be a quasi-conformal map.
Then there exists $K>1$ such that the parabolic Ferrand distance $\ferrdPar$ satisfies
	\begin{equation}\label{eq68d143da}
		\frac1{K } \ferrdPar(x_1,x_2)
		\le \ferrdPar(f(x_1),f(x_2)) 
		\le K \ferrdPar(x_1,x_2),\qquad \forall x_1,x_2\in X .
	\end{equation}
\end{proposition}
\begin{proof} We equivalently show that 	  there exists $K>1$ such that for the Ferrand gauge $\ferrg$ we have
	\begin{equation}\label{eq68755d66}
		\frac1K \ferrg(x_1,x_2)
		\le \ferrg(f(x_1),f(x_2)) 
		\le K \ferrg(x_1,x_2) ,\qquad \forall x_1,x_2\in X.
	\end{equation}
	We focus on the second inequality in~\eqref{eq68755d66}.
	If $\ferrg(x_1,x_2)=\infty$, then the inequality is true.
	Suppose $\ferrg(x_1,x_2)<\infty$ and let $\epsilon>0$.
	Then there are $\infty$-continua $E,F\subset X$ such that $x\in E$, $y\in F$ and  
	$\ferrg(x_1,x_2) + \epsilon \ge \capacity_Q(E;F)$.

	Since $f$ is a homeomorphism, then $f(E),f(F)\subset Y$ are $\infty$-continua
	with $f(x_1)\in f(E)$, $f(x_2)\in f(F)$.
	Therefore, $\ferrg(f(x_1),f(x_2))\le \capacity_Q(f(E);f(F))$.
	
	By the Geometric Definition~\ref{def68cc0a5c} of quasi-conformal map, there is $K>0$ such that
	\begin{equation}
		\ferrg(f(x_1),f(x_2))
		\le \capacity_Q(f(E);f(F))
		\le K \capacity_Q(E;F)
		\le K \ferrg(x_1,x_2) + K\epsilon .
	\end{equation}
	Since $\epsilon$ is an arbitrary positive number, we obtain the second inequality in~\eqref{eq68755d66}.
	The first inequality in~\eqref{eq68755d66} follows from the fact that the above statement also applies to $f^{-1}$.
\end{proof}

The main application of the Ferrand distance will be Theorem~\ref{thm68755ab4}, which is itself based on the above Proposition~\ref{prop68755c04} and Proposition~\ref{prop68755b1a}.
One of the hypotheses of Proposition~\ref{prop68755b1a}, which then reverberate in Theorem~\ref{thm68755ab4}, is that the growth dimension $N$ is strictly smaller than the Hausdorff dimension $Q$.
The case $N=Q$ has to be singled out with the following proposition, which is a direct application of the study of Loewner spaces made in \cite{MR1654771}.
In these spaces, we have a global Poincaré inequality and a global two-sided volume-growth bound.

\begin{proposition}\label{prop687658f1}
	Let $(X, d, \mu)$ be a metric measure space with $Q$-bounded geometry, as
	in Definition~\ref{def6803c8e2} with $R \equiv +\infty$.
	Suppose also that $X$ is boundedly compact and that $Q>1$.
	Then, the parabolic Ferrand distance $\ferrdPar$ satisfies
	\begin{equation}\label{eq689ee2d3}
		\ferrdPar(x,y) = 0,
		\qquad\forall x,y\in X.
	\end{equation}
\end{proposition}
\begin{proof}
	Every boundedly compact space with $Q$-bounded geometry, $Q>1$, is quasi-convex by
Proposition~\ref{prop686fb4ae-uniform} and, moreover, it is a Loewner space by \cite[Theorem 5.7]{MR1654771}; see also \cite[Theorem 9.10]{MR1800917}.
	Let $\psi_Q:(0,+\infty)\to[0,+\infty]$ be the Loewner function of $X$, that is,
	\begin{equation}\label{eq67cb0444}
		\psi_Q(t) = 
		\inf\left\{
			\capacity_Q(E;F) : 
			\begin{array}{c}
			E,F\subset X \text{ continua with } \\
			0 < \dist(E,F) \le t \min \{\operatorname{diam} E,\operatorname{diam} F\}
			\end{array}
		\right\} .
	\end{equation}
	By \cite[Theorem 3.6]{MR1654771}, $\psi_Q(t) \approx \log\frac1t$ as $t\approx 0$; see also \cite[Theorem 8.23]{MR1800917}.
	
	Let $x,y\in X$ and $E,F\subset X$ be $\infty$-continua with $x\in E$ and $y\in F$.
	For every $R>1$, let $E_R$ and $F_R$ be the connected components of $E\cap B(x,R)$ and $F\cap B(y,R)$ containing $x$ and $y$, respectively.
	Then $\dist(E_R,F_R) \le d(x,y)$ for all $R$, and $\lim_{R\to\infty} \operatorname{diam} E_R = \lim_{R\to\infty} \operatorname{diam} F_R = \infty$.
	It follows that
	\begin{align}
		\capacity_Q(E;F)
		&\ge \limsup_{R\to\infty} \capacity_Q(E_R;F_R) \\
		&\ge \limsup_{R\to\infty} \psi_Q\left( \frac{ \dist(E_R,F_R) }{ \min \{\operatorname{diam} E_R,\operatorname{diam} F_R\} } \right)
		= \infty .
	\end{align}
	We infer that $			\ferrg(x,y) = +\infty$, for all $x,y\in X$, and hence we conclude~\eqref{eq689ee2d3}.
\end{proof}

\subsection{Upper bound to the parabolic Ferrand distance}

\begin{proposition}\label{prop68720f23}
	Let $(X,d,\mu)$ be a metric measure space with uniformly locally $Q$-bounded geometry, with $Q>1$.
	Suppose that $X$ is quasi-convex.
	Then, there exists $C>0$ such that
	\begin{equation}\label{eq686fb287_bis}
		\ferrdPar(x,y) \leq C \max\left\{ 1,  {d(x,y) } \right\},\qquad \forall x,y\in X .
	\end{equation}
	The constant $C$ only depends on the geometric data of $X$.
\end{proposition}
\begin{proof}To show \eqref{eq686fb287_bis} we equivalently show 
	\begin{equation}\label{eq686fb287}
		\ferrg(x,y) \ge C \min\left\{ 1, \frac1{d(x,y)^Q} \right\} .
	\end{equation}
	By Corollary~\ref{cor67efffbc}, there exist $C_0,r_0>0$ depending only on the data of $X$ such that, for every $u:X\to\R$ monotone with upper gradient $g$ we have
	\begin{equation}\label{eq686fb595}
		\oscillation_{B(z,r)} u \le C_0 \|g\|_{L^Q(X)} ,
		\qquad\forall z\in X,\ \forall r\in(0,r_0) .
	\end{equation}
	
	Let $x,y\in X$ and assume $\ferrg(x,y)<\infty$.
	Then there are $\infty$-continua $E,F\subset X$ with $x\in E$ and $y\in F$
	such that, $\ferrg(x,y) \ge \capacity_Q(E;F) /2$.
	By Proposition~\ref{prop67e505c3}, there is $u:X\to[0,1]$ with upper gradient $g$ such that $E\subset\{u\le 0\}$, $F\subset\{u\ge1\}$, $u$ is monotone on $X\setminus(E\cup F)$, and $\capacity_Q(E;F) \ge \frac12 \int_X g^Q \did\mu$.
	In particular,
	$u$ is monotone on $X$ by Lemma~\ref{lem67e51b03} and 
	\begin{equation}\label{eq686fbccd}
		\ferrg(x,y) \ge \frac14 \|g\|_{L^Q(X)}^Q .
	\end{equation}
	
	We separate two cases, depending on the size of $d(x,y)$.
	First case: $d(x,y)<r_0$.
	Then
	\begin{align}
		1&= u(y) - u(x)
		\overset{\eqref{eq686fb595}}\le C_0 \|g\|_{L^Q(X)} 
		\overset{\eqref{eq686fbccd}}\le C_0 4^{1/Q} \ferrg(x,y)^{1/Q} .
	\end{align}
	This shows~\eqref{eq686fb287} with $C=\frac1{4C_0^Q}$.
	
	Second case: $d(x,y)\ge r_0$.
	Since $X$ is assumed to be $C_1$-quasi-convex for some $C_1\ge1$,
	there is a Lipschitz $\gamma:[0,1]\to X$ from $x$ to $y$ with length at most $\Length(\gamma) \le C_1d(x,y)$.
	Let $n = \lceil \frac{\Length(\gamma)}{ r_0 } \rceil\ge 1$, so that
	\begin{equation}\label{eq68737758}
		n
		\le \frac{\Length(\gamma)}{ r_0 } + 1
		\le  \frac{C_1 d(x,y)}{ r_0 } + \frac{d(x,y)}{r_0}
		= d(x,y) \frac{C_1+1}{r_0},
	\end{equation}
	and let $0=t_0<t_1<\dots< t_n=1$ such that, for each $j\in\{1,\dots,n\}$, 
	we have $d(\gamma(t_{j-1}) , \gamma(t_j) ) \le \Length(\gamma|_{[t_{j-1},t_j]}) < r_0$.
	Then
	\begin{align}
		1&= u(\gamma(1)) - u(\gamma(0))
		\le \sum_{j=1}^n |u(\gamma(t_{j-1})) - u(\gamma(t_j))| 
		\overset{\eqref{eq686fb595}}\le n C_0 \|g\|_{L^Q(X)} \\
		&\overset{\eqref{eq68737758}}\le d(x,y) C_0 \frac{C_1+1}{r_0} \|g\|_{L^Q(X)} 
		\overset{\eqref{eq686fbccd}}\le d(x,y) C_0\frac{C_1+1}{r_0} 4^{1/Q} \ferrg(x,y)^{1/Q} .
	\end{align}
	We conclude~\eqref{eq686fb287} with $C = \frac14\left(\frac{r_0}{C_0(C_1+1) }\right)^Q$, which only depends on the geometric data of $X$.
\end{proof}

\begin{remark}
	By Proposition~\ref{prop686fb4ae-uniform}, a space with locally bounded geometry is locally $C$-convex for some $C$.
	However, in the above Proposition~\ref{prop68720f23}, we need $X$ to be globally $C$-convex.
	In fact, this hypothesis is used in the proof exactly for the case when the two points are far apart.
	
	Quasi-convexity can be replaced with $(\ell,r)$-quasi-geodesity as in Definition~\ref{def6874b30e}, for appropriate $\ell$ and $r$.
	Indeed, in the proof of Proposition~\ref{prop68720f23},
	$C$-quasi-convexity is used to recover a sequence $\{x_j\}_{j=0}^n$ with $x_0=x$, $x_n=y$, $d(x_{j-1},x_j) < r_0$ for all $j\in\{1,\dots,n\}$, and $n\le C_2 d(x,y)$,
	where $r_0$ is the radius given by Corollary~\ref{cor67efffbc}.
	In our case, we had $C_2=\frac{C_1+1}{r_0}$; see~\eqref{eq68737758}.
	So, we need $X$ to be $(\ell,r)$-quasi-geodesic for some $r\in(0,r_0)$ and for some $\ell\in\R$.
\end{remark}

\subsection{Quasi-straight sequences}\label{subs68cf0748}
In order to give a lower bound to the Ferrand distance $\ferrdPar(x,y)$ of two points $x$ and $y$, we need to find a ``good'' pair of $\infty$-continua $E$ and $F$ containing $x$ and $y$, respectively; see Proposition~\ref{prop68720d05}.
We will construct these continua from bi-infinite sequences of points that are quasi-aligned to straight lines, and that pass close to $x$ and $y$.
We call such sequences \emph{quasi-straight sequences}; see Definition~\ref{def68766881}.
We call \emph{quasi-straightenable} a space where every pair of points lies close to a quasi-straight sequence, quantitatively and uniformly; see Definition~\ref{def68ce7443}.
The advantage of quasi-straight sequences is that their existence is (quantitatively) preserved by quasi-isometries.
For this reason, we can prove their existence in geodesic Lie groups with polynomial growth; see Proposition~\ref{prop687570ac}.

\begin{definition}\label{def68766881}
	Let $(X,d)$ be a metric space.
	A (bi-infinite) sequence $\zeta:\Z\to X$ is a \emph{$K$-quasi-straight sequence}, with $K\ge0$,
	if both $\zeta(\Z_{\le0})$ and $\zeta(\Z_{\ge0})$ are unbounded sets in $X$,
	\begin{equation}\label{eq6870e51e}
		d(\zeta_k , \zeta_{k+1}) \le 1 + K, 
		\qquad\forall k\in\Z ,
	\end{equation}
	and
	\begin{equation}\label{eq6870e535}
	\begin{gathered}
		d(\zeta_i,\zeta_j) + d(\zeta_j, \zeta_k) - d(\zeta_i,\zeta_k)
		\le K d(\zeta_i,\zeta_k) + K , \\
		\qquad\hspace{5cm}\forall i,j,k\in\Z \text{ with } i\le j\le k .
	\end{gathered}
	\end{equation}
\end{definition}
Notice that~\eqref{eq6870e535} is equivalent to 
\begin{equation}\label{eq68ce7545}
\tag{\ref{eq6870e535}'}
\begin{gathered}
	d(\zeta_i,\zeta_j) + d(\zeta_j, \zeta_k)
	\le (1+K) d(\zeta_i,\zeta_k) + K , \\
	\qquad\hspace{5cm}\forall i,j,k\in\Z \text{ with } i\le j\le k ,
\end{gathered}
\end{equation}
which we will use often.

\begin{remark}\label{rem6876686d}
	If a function $\zeta:\Z\to X$ is a quasi-isometry, where $\Z$ is endowed with the Euclidean distance, then $\zeta$ is a quasi-geodesic and a quasi-straight sequence.
	However, there are quasi-straight sequences that do not arise from quasi-geodesics.
	For example, consider $\Z\into\R$ where $\R$ is endowed with the snowflaked Euclidean distance: $d(x,y) := \sqrt{|x-y|}$, for $x,y\in\R$.
	One can easily check that $\Z$ is a quasi-straight sequence in $(\R,d)$.
	Such example appears crucially in the first Heisenberg group.
	See Lemma~\ref{lem6874bbe0} for a generalization of this example.
\end{remark}

\begin{lemma}\label{lem68711041}
	Quasi-isometries preserve quasi-straight sequences, quantitatively.
\end{lemma}
\begin{proof}
	Let $(X,d)$ and $(Y,d)$ be metric spaces, $f:X\to Y$ a $(L,C)$-quasi-isometry
	and $\zeta:\Z\to X$ a $K$-straight sequence, for some $L,C,K\ge0$.
	
	First, for every $m\in \Z$ we have
	\begin{equation}
		d(f(\zeta_m),f(\zeta_0))
		\overset{\eqref{eq68ce76a1}}\ge \frac1L d(\zeta_m,\zeta_0) - C .
	\end{equation}
	Since $\limsup_{m\to+\infty}d(\zeta_m,\zeta_0) = +\infty$
	and $\limsup_{m\to-\infty}d(\zeta_m,\zeta_0) = +\infty$,
	then $f\circ\zeta(\Z_{\ge0})$ and $f\circ\zeta(\Z_{\le0})$ are unbounded.
	
	Second, for every $k\in \Z$ we estimate
	\begin{align}
		d(f(\zeta_k) , f(\zeta_{k+1})) 
		\overset{\eqref{eq68ce76a1}}\le L d(\zeta_k,\zeta_{k+1}) + C
		\overset{\eqref{eq6870e51e}}\le L(1+K) + C.
	\end{align}
	
	Third, if $i,j,k\in \Z$ with $i\le j\le k$, then
	\begin{align}
		&d(f(\zeta_i),f(\zeta_j)) + d(f(\zeta_j),f(\zeta_k)) \\
		&\overset{\eqref{eq68ce76a1}}\le L d(\zeta_i,\zeta_j) + L d(\zeta_j, \zeta_k) + 2C \\ 
		&\overset{\eqref{eq68ce7545}}\le L (K+1) d(\zeta_i,\zeta_k) + LK + 2C \\
		&\overset{\eqref{eq68ce76a1}}\le L (K+1) (L d(f(\zeta_i),f(\zeta_k)) + LC) + LK + 2C \\
		&= L^2 (K+1) d(f(\zeta_i),f(\zeta_k)) + L^2C (K+1) + LK + 2C .
	\end{align}
	
	We conclude that $f\circ\zeta:\Z\to Y$ is a $K'$-quasi-straight sequence with
	\begin{equation}
		K' = \max\{ L(1+K) + C , L^2 (K+1) - 1 , L^2C (K+1) + LK + 2C \} .
		\qedhere
	\end{equation}
\end{proof}

\begin{definition}\label{def68ce7443}
	A metric space $(X,d)$ is \emph{$K$-quasi-straightenable} with $K\ge0$ if, for every $x,y\in X$, there exists an $K$-quasi-straight sequence $\zeta:\Z\to X$ 	with $x,y\in B(\zeta(\Z),K)$.
	We say that $X$ is \emph{quasi-straightenable} if it is $K$-quasi-straightenable for some $K\ge0$.
\end{definition}

\begin{lemma}\label{lem6874bc72}
	Quasi-straightenability is preserved under quasi-isometry, quantitatively.
\end{lemma}
\begin{proof}
	Let $(X,d)$ and $(Y,d)$ be metric spaces and $f:X\to Y$ a $(L,C)$-quasi-isometry.
	Suppose that $X$ is $K$-quasi-straightenable and let $y_1,y_2\in Y$.
	Then there are $x_1,x_2\in X$ such that $d(f(x_j),y_j)\le C$ for both $j\in\{1,2\}$.
	Since $X$ is $K$-quasi-straightenable, there is a $K$-quasi-straight sequence $\zeta:\Z\to X$, and $m_1,m_2\in\Z$ with $d(\zeta(m_j),x_j)\le K$ for both $j\in\{1,2\}$.
	By Lemma~\ref{lem68711041}, the function $f\circ\zeta:\Z\to Y$ is a $K'$-quasi-straight sequence in $Y$, for some $K'\in\R$ depending on $K, L, C$.
	Notice that, for both $j\in\{1,2\}$,
	\begin{equation}
		d(f\circ\zeta(m_j),y_j)
		\le d(f\circ\zeta(m_j),f(x_j)) + d(f(x_j) , y_j) 
		\overset{\eqref{eq68ce76a1}}\le L d(\zeta(m_j),x_j) + 2C
		\le LK + 2C .
	\end{equation}
	It remains to show that $f\circ\zeta(\Z_{\ge0})$ and $f\circ\zeta(\Z_{\le0})$ are unbounded.
	Notice that, if $m\in\Z$, then
	\begin{equation}
		d(f\circ\zeta(m) , f\circ\zeta(0))
		\overset{\eqref{eq68ce76a1}}\ge \frac1L d(\zeta(m),\zeta(0)) - C.
	\end{equation}
	Since $\limsup_{m\to+\infty}d(\zeta(m),\zeta(0)) = +\infty$
	and $\limsup_{m\to-\infty}d(\zeta(m),\zeta(0)) = +\infty$,
	then $f\circ\zeta(\Z_{\ge0})$ and $f\circ\zeta(\Z_{\le0})$ are unbounded.
	
	We conclude that $Y$ is $K''$-quasi-straightenable with 
	\begin{equation}
		K'' = \max\{K',LK + 2C\} ,
	\end{equation}
	where $K'$ is given by Lemma~\ref{lem68711041}.
\end{proof}

\subsection{Lower bound to the parabolic Ferrand distance}
We use quasi-straight sequences to give a lower bound to the parabolic Ferrand distance $\ferrdPar$. The key property is an upper bound in the capacities of some distinguished sets; see  Proposition~\ref{prop68720d05}.
For the proof, we will need the following lemma.

\begin{lemma}\label{lem68712318}
	Let $(X,d)$ be a metric space and $\zeta:\Z\to X$ a $K$-quasi-straight sequence, with $K\ge0$.
	Let $m,n\in\Z$ with $m\le n$.
	Then, for every $i,j\in\Z$ with $i\le m$ and $n\le j$,
	\begin{equation}\label{eq68712307}
		d(\zeta_i,\zeta_j) 
		\ge \frac{1}{(1+K)^2} d(\zeta_m,\zeta_n) - \frac{K(K+2)}{(1+K)^2} .
	\end{equation}
\end{lemma}
\begin{proof}
	Let $i\le m\le n\le j$ be integers.
	Then
	\begin{align}
		d(\zeta_m,\zeta_n)
		&\le d(\zeta_i,\zeta_m) + d(\zeta_m,\zeta_n) + d(\zeta_n,\zeta_j) \\
		[\text{by~\eqref{eq68ce7545} and $K\ge0$}]
		&\le (K+1) d(\zeta_i,\zeta_n) + K + (K+1) d(\zeta_n,\zeta_j) \\
		[\text{by~\eqref{eq68ce7545}}]
		&\le (K+1) ( (K+1)d(\zeta_i,\zeta_j) + K) + K \\
		&= (K+1)^2 d(\zeta_i,\zeta_j) + K(K+2) .
	\end{align}
	Rearranging, we get~\eqref{eq68712307}.
\end{proof}


\begin{proposition}\label{prop68720d05}
	Let $(X,d,\mu)$ be a metric measure. 
	Let $N\in(\dimgr(X),+\infty)$, that is, there is $C_N>0$ such that
	\begin{equation}\label{eq686f6ed9}
		\mu(B(x,R)) \le C_N R^N , \qquad 	\forall x\in X,\ \forall R\ge1 .
	\end{equation}
	For every $K\ge0$ and $Q\in(N,+\infty)$
	there are $\bar R>0$ and $C>0$ such that the following holds:
	Let $x,y\in X$ with $d(x,y)\ge \bar R$.
	Let $\zeta:\Z\to X$ be a $K$-straight sequence with $d(\zeta_m,x)\le K$ and $d(\zeta_n,y)\le K$ for some $m,n\in\Z$, $m\le n$.	
	Let $E,F\subset X$ be closed sets with $x\in E$, $y\in F$, and with 
	$E\subset B(\zeta(\Z_{\le m}),K)$ and $F\subset B(\zeta(\Z_{\ge n}), K)$.
	Then 
	\begin{equation}\label{eq68713725}
		\capacity_Q(E;F) \le C \frac1{d(x,y)^{Q-N}} .
	\end{equation}
	Both $\bar R$ and $C$ depend only on the constants $Q,N,C_N,K$.
\end{proposition}
\begin{proof}
	By Lemma~\ref{lem68712318}, we have $E\cap F = \emptyset$, for $\bar R$ large enough,
	say $\bar R$ such that
	\begin{equation}\label{eq6871305c}
		\frac{1}{(1+K)^2} (\bar R -2K) - \frac{K(K+2)}{(1+K)^2} > 2K .
	\end{equation}
	
	For $S\subset X$ and $z\in X$, define $d_S(z) := \dist(z,S) = \inf\{d(z,k):k\in S\}$.
	Thus, define $u:X\to[0,1]$ by
	\begin{equation}
		u(z) := \frac{ d_E(z) }{d_E(z) + d_F(z)} ,
		\qquad\forall z\in X.
	\end{equation}
	Notice that, since $E\cap F=\emptyset$, then $u$ is well defined for all $z\in X$,
	with $E\subset\{u=0\}$ and $F\subset\{u=1\}$.
	By Lemma~\ref{lem6842ac4a}, an upper gradient of $u$ is the function $g:X\to(0,+\infty)$,
	\begin{equation}
		g(z) := \frac{1}{d_E(z) + d_F(z)} ,
		\qquad\forall z\in X.
	\end{equation}
	Let $\ell\in[m,n]\cap\Z$ and set $o:=\zeta(\ell)$.
	
	We claim that, for every $z\in X$,
	\begin{equation}\label{eq686f5a9d}
		d_E(z) + d_F(z) 
		\ge \frac{2}{K+2} d(o,z) - \frac{K(2K+5)}{K+2} .
	\end{equation}
	Indeed, if $z\in X$, $x'\in E$ and $y'\in F$, then there are $i\in\Z_{\le m}$ and $j\in\Z_{\ge n}$ with $d(\zeta_i,x') \le K$ and $d(\zeta_j,y')\le K$.
	Hence,
	\begin{align*}
	2d(o,z)
		&\le d(o,\zeta_i) + d(\zeta_i,x') + d(x',z)
			+ d(o,\zeta_j) + d(\zeta_j,y') + d(y',z) \\
		&\overset{\eqref{eq6870e535}}\le (K+1) d(\zeta_i,\zeta_j) + 3K + d(x',z) + d(y',z) \\
		&\le (K+1) ( d(\zeta_i,x') + d(x',z) + d(z,y') + d(y',\zeta_j) ) + 3K + d(x',z) + d(y',z) \\
		&\le (K+2) (d(x',z) + d(y',z)) + (3K + (K+1)2K) .
	\end{align*}
	Taking the infimum over $x'\in E$ and $y'\in F$, and then rearranging the terms, we get~\eqref{eq686f5a9d}.
	
	Next, we claim that, for every $z\in X$,
	\begin{equation}\label{eq686f617d}
		d_E(z) + d_F(z) 
		\ge \frac{d(x,y)}{K+1} - \left(2K + \frac{4K}{K+1}\right) - 2 d(o,z) .
	\end{equation}
	Indeed,  if $z\in X$, $x'\in E$ and $y'\in F$, then there are $i\in\Z_{\le m}$ and $j\in\Z_{\ge n}$ with $d(\zeta_i,x') \le K$ and $d(\zeta_j,y')\le K$.
	Hence,
	\begin{align*}
		&d(z,x') + d(z,y') \\
		&\ge d(\zeta_i,o) - d(x',\zeta_i) - d(o,z)
			+ d(\zeta_j,o) - d(y',\zeta_j) - d(o,z) \\
		&\overset{\eqref{eq6870e535}}\ge \frac{ d(\zeta_i,\zeta_m) + d(\zeta_m,o) }{ K+1 } - \frac{K}{K+1}
			+ \frac{ d(\zeta_j,\zeta_n) + d(\zeta_n,o) }{ K+1 } - \frac{K}{K+1}
			- 2 d(o,z) - 2K \\
		&= \frac{ d(\zeta_i,\zeta_m) + d(\zeta_m,o) + d(\zeta_j,\zeta_n) + d(\zeta_n,o) }{ K+1 }
			- 2 \frac{K}{K+1} - 2K - 2 d(o,z) \\
		&\ge \frac{ d(\zeta_m,\zeta_n) }{ K+1 } - 2 \frac{K}{K+1} - 2K - 2 d(o,z) \\
		&\ge \frac{ d(x,y) }{ K+1 } - \left(2K + \frac{4K}{K+1}\right) - 2 d(o,z) .
	\end{align*}
	Taking the infimum over $x'\in E$ and $y'\in F$, we get~\eqref{eq686f617d}.
		
	From~\eqref{eq686f617d} we obtain that, if
	\begin{equation}\label{eq686f6311}
		R_{xy} := \frac{d(x,y)}{4(K+1)} - \frac12\left(2K + \frac{4K}{K+1}\right),
	\end{equation}
	then
	\begin{equation}\label{eq686f636e}
		\forall z\in B(o,R_{xy}),\qquad
		d_E(z) + d_F(z) \ge \frac{d(x,y)}{2(K+1)} .
	\end{equation}
	
	Finally, we compute $\int_X g^Q\did\mu$ to get an upper bound to $\capacity_Q(E;F)$.
	For this, we assume
	\begin{equation}\label{eq686f7409}
		d(x,y) \ge \bar R 
	\end{equation}
	for some $\bar R>0$ to be chosen large enough.
	First, $\bar R$ must satisfy~\eqref{eq6871305c}, so that $E\cap F = \emptyset$.
	Second, $\bar R$ must ensure
	\begin{equation}\label{eq686f7351}
		R_{xy} \ge 1,
	\end{equation}
	thus $\bar R \ge 2(K+1)\left(2K + \frac{4K}{K+1}\right) $.
	Third, we require $\bar R$ to be large enough to ensure
	\begin{equation}\label{eq686f7391}
		R_{xy} \ge K(2K+5) ,
	\end{equation}
	hence $\bar R \ge 4(K+1) (K(2K+5) + K + 2K/(K+1))$.
	Condition~\eqref{eq686f7391} applied to~\eqref{eq686f5a9d} implies that
	\begin{equation}\label{eq687133d6}
		\forall z\in X\setminus B(o,R_{xy}),\qquad
		d_E(z) + d_F(z) \ge \frac{1}{K+2} d(o,z) .
	\end{equation}
	
	Therefore,  we have
	\begin{align}
		&\int_X g(z)^Q\did\mu(z) \\
		&= \int_{B(o,R_{xy})} \frac1{(d_E(z)+d_F(z))^Q} \did \mu(z)
			+ \int_{X\setminus B(o,R_{xy})} \frac1{(d_E(z)+d_F(z))^Q} \did \mu(z) \\
		&\overset{\eqref{eq686f636e}\&\eqref{eq687133d6}}\le  
			\int_{B(o,R_{xy})} \left( \frac{2(K+1)}{d(x,y)} \right)^Q \did \mu(z)
			+  \int_{X\setminus B(o,R_{xy})} \left(\frac{K+2}{d(o,z)}\right)^Q \did\mu(z) \\
		&= \left( \frac{2(K+1)}{d(x,y)} \right)^Q \mu(B(o,R_{xy}))
			+ (K+2)^Q \int_0^{1/R_{xy}^Q} \mu(\{z\in X\setminus B(o,R_{xy}) : \frac1{d(o,z)^Q} >t \}) \did t \\
		&= \left( \frac{2(K+1)}{d(x,y)} \right)^Q \mu(B(o,R_{xy}))
			+ (K+2)^Q \int_0^{1/R_{xy}^Q} \mu(\{z : R_{xy}\le d(o,z) < 1/t^{1/Q} \}) \did t \\
		&\le \left( \frac{2(K+1)}{d(x,y)} \right)^Q \mu(B(o,R_{xy}))
			+ (K+2)^Q \int_0^{1/R_{xy}^Q} \mu(B(o,t^{-1/Q})) \did t \\
		&\overset{\eqref{eq686f6ed9}\&\eqref{eq686f7351}}\le \left( \frac{2(K+1)}{d(x,y)} \right)^Q C_N R_{xy}^N
			+ (K+2)^Q \int_0^{R_{xy}^{-Q}} C_N t^{-N/Q} \did t \\
		&\overset{[N<Q]}= \left( \frac{2(K+1)}{d(x,y)} \right)^Q C_N R_{xy}^N
			+ (K+2)^Q \frac{Q}{Q-N} C_N \left( t^{1-N/Q} \right|_{t=0}^{R_{xy}^{-Q}} \\
		&= \left( \frac{2(K+1)}{d(x,y)} \right)^Q C_N R_{xy}^N
			+ (K+2)^Q \frac{Q}{Q-N} C_N R_{xy}^{N-Q} \\
		&\le \left( 2^Q(K+1)^Q  C_N \left( \frac{R_{xy}}{d(x,y)} \right)^N
			+ (K+2)^Q \frac{Q}{Q-N} C_N \left( \frac{R_{xy}}{d(x,y)} \right)^{N-Q} \right) d(x,y)^{N-Q} .
	\end{align}
	Since
	\begin{equation}
		\frac{1}{4(K+1)} - \frac12\left(2K + \frac{4K}{K+1}\right) \frac{1}{\bar R}
		\le \frac{R_{xy}}{d(x,y)} 
		\le \frac{1}{4(K+1)} ,
	\end{equation}
	we can also ensure 
	$\frac{1}{8(K+1)}\le \frac{R_{xy}}{d(x,y)} \le \frac{1}{4(K+1)}$ if $\bar R$ is large enough, and thus
	we have $C$ and $\bar R$ such that, if $d(x,y)\ge\bar R$, then
	\begin{equation}
		\capacity_Q(E,F) 
		\le \int_X g(z)^Q\did\mu(z) 
		\le C d(x,y)^{N-Q} ,
	\end{equation}
	that is,~\eqref{eq68713725} holds.
\end{proof}

\begin{lemma}\label{lem6842ac4a}
	Let $(X,d)$ be a metric space.
	For $K\subset X$ and $x\in X$, define $d_K(x) := \dist(x,K) = \inf\{d(x,k):k\in K\}$.
	Let $E,F\subset X$ be such that $\closure(E)\cap\closure(F)=\emptyset$.
	Define $u:X\to[0,1]$ by
	\begin{equation}\label{eq6842ac15}
		u(x) := \frac{d_E(x)}{d_E(x) + d_F(x)} ,
		\qquad\forall x\in X.
	\end{equation}
	Then $u$ is locally Lipschitz, and in particular, for every $x,y\in X$
	\begin{equation}\label{eq6842a137}
		|u(x) - u(y)| \le  d(x,y) \max\left\{ \frac{1}{d_E(y) + d_F(y)} , \frac{1}{d_E(x) + d_F(x)} \right\} .
	\end{equation}
	Moreover, for every $x\in X$,
	$u(x) = 0$ if and only if $x\in \closure(E)$, 
	and $u(x) = 1$ if and only if $x\in \closure(F)$.
	Finally, the function $g(x) := \frac{1}{d_E(x) + d_F(x)}$
	is an upper gradient of~$u$.
\end{lemma}
\begin{proof}
	For every $x,y\in X$, we have
	\begin{align}
		u(x) - u(y)
		&= \frac{d_E(x)}{d_E(x) + d_F(x)} - \frac{d_E(y)}{d_E(y) + d_F(y)} \\
		&= \frac{ d_E(x) (d_E(y) + d_F(y)) - d_E(y) (d_E(x) + d_F(x)) }{ (d_E(x) + d_F(x)) (d_E(y) + d_F(y)) } \\
		&= \frac{ d_E(x)  d_F(y) - d_E(y) d_F(x) }{ (d_E(x) + d_F(x)) (d_E(y) + d_F(y)) } \\
		&\le \frac{ d_E(x)  (d_F(x)+d(x,y)) - (d_E(x)-d(x,y)) d_F(x) }{ (d_E(x) + d_F(x)) (d_E(y) + d_F(y)) } \\
		&= d(x,y) \frac{ d_E(x) + d_F(x) }{ (d_E(x) + d_F(x)) (d_E(y) + d_F(y)) } \\
		&= d(x,y) \frac{1}{d_E(y) + d_F(y)} .
	\end{align}
	This shows~\eqref{eq6842a137}.
	
	If $z\in X$, $r>0$ and $x\in B(z,r)$, then $d_E(x) + d_F(x) \ge d_E(z) + d_F(z) - 2r \ge \max\{d_E(z) , d_F(z)\} - 2r $.
	Since $\closure(E)\cap\closure(F)=\emptyset$, then $\max\{d_E(z) , d_F(z)\}>0$.
	Therefore, there are $r>0$ and $c>0$ such that $d_E(x) + d_F(x)\ge c$ for all $x\in B(z,r)$.
	From~\eqref{eq6842a137}, we conclude that $u$ is $1/c$ Lipschitz on $B(z,r)$.
	
	The last statements are easy.
	In particular, recall that the pointwise Lipschitz constant is an upper gradient; see~\cite[Proposition 1.11]{MR1708448}.
\end{proof}


\subsection{QC implies QI in the strict parabolic case}

\begin{proposition}\label{prop68755b1a}
	Let $(X,d,\mu)$ be a metric measure space with uniformly locally $Q$-bounded geometry and growth dimension $N$.
	Assume $1\le N<Q$ and that $X$ is quasi-convex and quasi-straightenable.
Then, there is $C>1$ such that the parabolic Ferrand distance $\ferrdPar$ satisfies
	\begin{equation}\label{eq68d0ee66}
		\frac1C d(x,y)^{1-\frac{N}{Q}}
		\le \ferrdPar(x,y) 
		\le C d(x,y) ,
		\qquad\forall x,y\in X\text{ with }d(x,y)\ge C.
	\end{equation}
\end{proposition}
\begin{proof}
We equivalently check that there is $C>1$ such that for the Ferrand gauge  $\ferrg$ we have
	\begin{equation}\label{eq689ef272}
		\frac1C d(x,y)^{-Q}
		\le \ferrg(x,y) 
		\le C d(x,y)^{N-Q} ,
		\qquad\forall x,y\in X\text{ with }d(x,y)\ge C.
	\end{equation}
	The first inequality in~\eqref{eq689ef272} follows from Proposition~\ref{prop68720f23}.
	The second inequality is shown as follows.
	Let $C\ge1$ be such that $X$ is $C$-quasi-convex.
	Let $K\ge0$ be the constant for which $X$ is $K$-quasi-straightenable.
	Let $\bar R$ and $\bar C$ be the constants given by Proposition~\ref{prop68720d05} for $C(K+1)$-quasi-straight sequences, with $N$ being the growth dimension and $Q$ the Hausdorff dimension.
	Fix $x,y\in X$ with $d(x,y)\ge\bar R$.
	Let $\zeta$ be a $K$-straight sequence with $x,y\in B(\zeta(\Z),K)$.
	Quasi-convexity easily imply that there are $\infty$-continua $E$ and $F$ with $x\in E\subset B(\zeta(\Z_{\le m}),C(K+1))$ and $y\in F\subset B(\zeta(\Z_{\ge n}),C(K+1))$ for some $m,n$ in $\Z$.
	Proposition~\ref{prop68720d05} implies that
	\begin{equation}
		\ferrg(x,y) 
		\le \capacity_Q(E,F) 
		\overset{\eqref{eq68713725}}\le \bar C d(x,y)^{N-Q} .
	\end{equation}
\end{proof}

\begin{theorem}\label{thm68755ab4}
	Let $X_1$ and $X_2$ be metric measure spaces with uniformly locally $Q$-bounded geometry,
	with growth dimensions $N_1$ and $N_2$, respectively,
	and both spaces quasi-straightenable, and quasi-geodesic.
	
	If $Q>1$, $N_1<Q$ and $N_2<Q$, then every quasi-conformal map $X_1\to X_2$ is a quasi-isometry.
\end{theorem}
\begin{proof}
	Let $f:X_1\to X_2$ be a quasi-conformal map.
	We claim that $f$ is bornologous.
	Indeed, by~\eqref{eq68d0ee66}, the identity maps $(X_1,d)\to(X_1,\ferrdPar)$ and $(X_2,\ferrdPar)\to (X_2,d)$ are bornologous.
	By Proposition~\ref{prop68755c04}, since $f$ is quasi-conformal, we have that $f:(X_1,\ferrdPar)\to(X_2,\ferrdPar)$ is also bornologous.
	Thus, $f:(X_1,d)\to (X_2,d)$ is composition of bornologous functions, hence bornologous.
	
	Since $f^{-1}:X_2\to X_1$ is also quasi-conformal, $f^{-1}$ is also bornologous.
	Finally, being a bornologous map with bornologous inverse between quasi-geodesic spaces, $f$ is a quasi-isometry by Corollary~\ref{cor67f68188}.
\end{proof}

\section{Isoperimetric inequalities at large scale}

\subsection{Perimeter and coarea inequality}

For generalizing the theory of BV functions and of sets of finite perimeter,
in~\cite{MR2005202,zbMATH06320681},
Miranda Jr., Ambrosio, and Di~Marino
developed a notion of 
total variation for locally integrable functions
in the setting of metric measure spaces.
We refer to those two papers for a complete introduction.
If $X$ is a metric measure space, we denote by 
$\|\Diff f\|(U)$
the \emph{total variation} of $f\in L^1_\loc(X;\R)$ on $U\subset X$.
Of the total variation, we only recall the following facts:
The map $U\mapsto\|\Diff f\|(U)$ is a Borel outer measure on $X$
and, if $g$ is an upper gradient of $f$ then, for every open set $U\subset X$,
\begin{equation}\label{eq6864deed}
	\|\Diff f\|(U) \le \int_U g(x)\did\mu(x) .
\end{equation}
The \emph{perimeter measure} of a set $E$ is the total variation of the characteristic function $\one_E$ of $E$, that is, the measure $\Per(E;\cdot) := \|\Diff\one_E\|$.
We will denote by $\Per(E)$ the \emph{perimeter} of a measurable set $E$, i.e., 
$\Per(E) := \Per(E;X) = \|\Diff\one_E\|(X)$.

The following coarea formula is proven by Miranda Jr.~in \cite[Proposition 4.2]{MR2005202} for spaces of locally bounded geometry, and extended to general metric measure spaces by Ambrosio and Di Marino in \cite{zbMATH06320681} (see the end of the introduction there).
\begin{theorem}[Coarea formula, {\cite[Proposition 4.2]{MR2005202}}]\label{thm68458049}
	Let $(X,d,\mu)$ be a metric measure space and $f\in L^1_\loc(X;\R)$.
	For $t\in\R$, define $E_t:=\{x\in X:f(x)>t\}$.
	Then, for every $U\subset X$ open,
	\begin{equation}\label{eq684ad148}
		\int_{-\infty}^\infty \Per(E_t;U) \did t = \|\Diff f\|(U) .
	\end{equation}
	In particular, by~\eqref{eq6864deed}, if $g:X\to[0,+\infty]$ is an upper gradient of $f$, then, for every $U\subset X$ open,
	\begin{equation}\label{eq68559ec1}
		\int_{-\infty}^\infty \Per(E_t;U) \did t \le \int_U g\did\mu .
	\end{equation}
\end{theorem}

In \cite[Theorem 4.5]{MR2005202},
Miranda Jr.~proved also the following local isoperimetric inequality.

\begin{theorem}[Local isoperimetric inequality]\label{thm684c1dff}
	Let $(X,d,\mu)$ be a metric measure space with locally $Q$-bounded geometry with data 
	$(Q,R,C_{\rm A}, C_{\rm P},\sigma)$, with $Q>1$.
	Then there is $C>0$ (depending only on the geometrical data of $X$) such that, 
	for every $x\in X$, every $r\in(0,R(x)/\sigma)$, and every $E\subset X$ measurable,
	\begin{equation}\label{eq684c1e2e}
		 \min\{\mu(B\cap E) , \mu(B\setminus E)\}^{\frac{Q-1}{Q}} \le C \Per(E;\sigma B) ,
	\end{equation}
	where $B = B(x,r)$ and $\sigma B = B(x,\sigma r)$.
\end{theorem}

\begin{remark}
	Let $(X,d,\mu)$ be a metric measure space with uniformly locally $Q$-bounded geometry with data 
	$(Q,R,C_{\rm A}, C_{\rm P},\sigma)$, with $Q>1$.
	We claim that~\eqref{eq684c1e2e} implies that, 
	for every $\alpha\in[\frac{Q-1}{Q},1]$, there is a constant $C_\alpha$ such that,
	for every $x\in X$, every $r\in(0,R/\sigma)$, and every $E\subset X$ measurable,
	\begin{equation}\label{eq684c222c}
		\min\{\mu(B\cap E) , \mu(B\setminus E)\}^\alpha \le C_\alpha \Per(E;\sigma B) ,
	\end{equation}
	where $B = B(x,r)$.
	Indeed, from Definition~\ref{def6803c8e2} (and the constancy of $R$) we get that there is $M\in(0,+\infty)$ such that $\mu(B(x,R)) \le M$ for all $x\in X$.
	Therefore, 
	\begin{align*}
	 \min\{\mu(B\cap E) , \mu(B\setminus E)\}^\alpha
	&= M^\alpha \frac{ \min\{\mu(B\cap E) , \mu(B\setminus E)\}^\alpha}{M^\alpha} \\
	&\le M^\alpha \frac{ \min\{\mu(B\cap E) , \mu(B\setminus E)\} ^{\frac{Q-1}{Q}}}{M^{\frac{Q-1}{Q}}} \\
	&\overset{\eqref{eq684c1e2e}}\le M^{\alpha - \frac{Q-1}{Q} } C \Per(E;\sigma B) .
	\end{align*}
\end{remark}

\subsection{Isoperimetric inequality at large scale}\label{subs6866e18d}

\begin{definition}	
	We say that a metric measure space $(X,d,\mu)$ supports \emph{an isoperimetric inequality at large scale of order $\alpha\in[0,1]$} if there exist $V_\alpha,C_\alpha\in(0,+\infty)$ such that,
	for every $E\subset X$ measurable,
	\begin{equation}
		\mu(E) \ge V_\alpha
		\qquad\THEN\qquad
		\mu(E)^\alpha \le C_\alpha \Per(E) .
	\end{equation}
	The \emph{large-scale isoperimetric dimension} of $(X,d,\mu)$ is
	\begin{equation}
		\dimisp(X)
		:= \sup\left\{ N\in[1,\infty):
			\begin{array}{c}
			\text{$X$ supports a large-scale  }\\
			\text{isoperimetric inequality of order $\frac{N-1}{N}$ }
			\end{array}
			\right\} .
	\end{equation}
\end{definition}

\begin{remark}\label{rem685cea89}
	Notice that, if $X$ supports a large-scale isoperimetric inequality of order $\alpha\in[0,1]$, then $X$ supports a large-scale isoperimetric inequality of order $\beta$ for every $\beta\in[0,\alpha]$.
	Indeed, taking $V_\beta := \max\{V_\alpha,1\}$, we obtain that,
	if $\mu(E)\ge V_\beta$, then
	$\mu(E)^\beta \le \mu(E)^\alpha \le C_\alpha \Per(E)$.
	For this reason, for every $N < \dimisp(X)$, the space $X$ supports a large-scale isoperimetric inequality of order $\frac{N-1}{N}$.
\end{remark}

Kanai showed in \cite{zbMATH03883148} that every quasi-isometry between two Riemannian manifolds with certain curvature bounds, preserve the order of isoperimetric inequalities at large scale.
A careful adaptation of Kanai's proof leads to the following generalization to metric measure spaces with uniformly locally bounded geometry.
In fact, 
the key properties used in Kanai's setting 
are uniform bounds from above and below for the volume of small balls and a local isoperimetric inequality.

\begin{theorem}[after Kanai {\cite{zbMATH03883148}}]\label{thm684c4f0a}
	Let $X$ and $Y$ be two quasi-geodesic metric measure spaces with uniformly locally bounded geometry, with Hausdorff dimension equal to or greater than 1.
	Assume $X$ and $Y$ are quasi-isometric and let $\alpha\in[0,1]$.
	Then, $X$ supports an isoperimetric inequality at large scale of order $\alpha$ if and only if does~$Y$.
\end{theorem}

We shall prove Theorem~\ref{thm684c4f0a} in Section~\ref{subs6855aa13}.
Kanai's Theorem~\ref{thm684c4f0a} allows us to obtain Theorem~\ref{thm6862ad98}, extending the result \cite[Theorem 2.1]{zbMATH01782641} by Pittet from Riemannian Lie groups to geodesic Lie groups; see Theorem~\ref{thm68655128}.

\subsection{Proof of Kanai Theorem~\ref{thm684c4f0a}}\label{subs6855aa13}
Kanai proves Theorem~\ref{thm684c4f0a} for Riemannian manifolds in two steps.
Firstly, they show a similar statement for graphs.
Secondly, they show that the isoperimetric inequality at large scale on a ``good'' Riemannian manifold is equivalent to a similar inequality on an approximating graph.
We will follow the same strategy.

The first step goes as follows.
A graph is a pair $(\scr N,E)$ where $\scr N$ is a set and $E\subset \scr N\times \scr N$ is a symmetric relation.
We say that $(x,y)\in E$ is an edge of the graph and write $x\sim y$.
A graph \emph{has finite order $m$} if 
\[
	\sup_{x\in \scr N} \#\{y:y\sim x\} \le m .
\]
A graph $(\scr N,E)$ becomes a metric measure space $(\scr N,d,\mu)$ when endowed with the word metric $d$ and the counting measure $\mu=\#$.
The \emph{graph-boundary} of a subset $S\subset \scr N$ is
\begin{equation}
	\de_1S := \{y\in \scr N\setminus S: \exists x\in S\ x\sim y\} .
\end{equation}
We define the \emph{(upper) isoperimetric profile} of $\scr N$ as $J^{\scr N}:(0,+\infty)\to[0,+\infty]$,
\begin{equation}
	J^{\scr N}(v) := \inf\{\#(\de_1S): S\subset \scr N,\ \#S\ge v\} ,
	\qquad \forall v\in(0,+\infty) .
\end{equation}

Kanai proves the following statement:
\begin{proposition}[{\cite[Lemma 4.2]{zbMATH03883148}}]\label{prop684c58f4}
	Let $\scr N$ and $\scr M$ be two graphs with finite order.
	Suppose $\scr N$ and $\scr M$ are quasi-isometric.
	Then, for every $\alpha\in[0,1]$,
	\begin{equation}\label{eq684c59c4}
		\liminf_{v\to\infty} \frac{J^{\scr N}(v)}{v^\alpha} > 0
		\qquad\IFF\qquad
		\liminf_{v\to\infty} \frac{J^{\scr M}(v)}{v^\alpha} > 0 .
	\end{equation}
\end{proposition}
\begin{remark}
	The statement written in~\cite{zbMATH03883148} is different from our Proposition~\ref{prop684c58f4}.
	However, the proof is the same.
	Indeed, Kanai proves that there are constants $a,b\in(0,+\infty)$ such that for every $S\subset \scr N$ there is $T\subset \scr M$ with $\#T\ge a \#S$ and $\#(\de_1T)\le b \#(\de_1S)$.
	This readily implies~\eqref{eq684c59c4}.
\end{remark}

We are now ready for the second step of the proof of Theorem~\ref{thm684c4f0a}.

\begin{proposition}[{\cite[Lemma 4.5]{zbMATH03883148}}]\label{prop684afea2}
	Let $(X,d,\mu)$ be a metric measure space with uniformly locally $Q$-bounded geometry  with data 
	$(Q,R,C_{\rm A}, C_{\rm P},\sigma)$, with $Q>1$.
	Let $\epsilon\in(0,R/(5\sigma))$ and $\scr N\subset X$ a maximal $\epsilon$-separated set.
	Endow $\scr N$ with the graph structure with edges $\{(x,y)\in\scr N\times\scr N:d(x,y)\le 3\epsilon\}$.
	
	Define $\mu^+,\mu^-:(0,+\infty) = [0,+\infty]$ by
	\begin{equation}
		\mu^+(r) = \sup\{\mu(B(x,r)) : x\in X\},
		\text{ and }
		\mu^-(r) = \inf\{\mu(B(x,r)) : x\in X\},
	\end{equation}
	and assume
	\begin{equation}\label{eq684c1886}
		\mu^-(\epsilon/2) > 0
		\text{ and }
		\mu^+(\epsilon(4\sigma+1/2)) < \infty .
	\end{equation}
	
	Then, for every $\alpha\in[\frac{Q-1}{Q},1]$, the following are equivalent:
	\begin{enumerate}[label=(\roman*)]
	\item\label{prop684afea2_1}
	There are $V_\alpha,C_\alpha\in(0,+\infty)$ such that,
	for every $E\subset X$ measurable,
	\begin{equation}\label{eq684c601c}
		\mu(E) \ge V_\alpha
		\quad\THEN\quad
		\mu(E)^\alpha \le C_\alpha \Per(E) .
	\end{equation}
	\item\label{prop684afea2_2}
	There are $V_\alpha,C_\alpha\in(0,+\infty)$ such that,
	for every $S\subset\scr N$, 
	\begin{equation}\label{eq684c6658}
		\# S \ge V_\alpha
		\quad\THEN\quad
		(\# S)^\alpha \le C_\alpha \#(\de_1S) .
	\end{equation}
	\end{enumerate}
\end{proposition}
\begin{proof}
	Notice that $(x,x)$ is an edge of $\scr N$, and that balls $\{B(x,\epsilon/2)\}_{x\in\scr N}$ are pairwise disjoint, while $\{B(x,2\epsilon)\}_{x\in\scr N}$ cover $X$.
	We claim that $\scr N$ is a graph with finite order.
	Indeed, for $x\in\scr N$, set $T_x := \{y\in\scr N:y\sim x\}$.
	Then, 
	$\mu^+(4\epsilon) \ge \mu\left( \bigsqcup_{y\in T_x} B(y,\epsilon/2) \right) = \sum_{y\in T_x} \mu(B(y,\epsilon/2)) \ge \#T_x \cdot \mu^-(\epsilon/2)$.
	Therefore, $\scr N$ has order at most $m$, where
	\begin{equation}
		m:=\frac{ \mu^+(4\epsilon) }{ \mu^-(\epsilon/2) } 
		\overset{\eqref{eq684c1886}}< \infty .
	\end{equation}
	
	\framebox{$\ref{prop684afea2_1}\THEN\ref{prop684afea2_2}$}
	For every $x\in X$, let $r_x\in(\epsilon,2\epsilon)$ be such that
	\begin{equation}\label{eq684afc7e}
		\Per(B(x,r_x)) \le \frac{\mu^+(2\epsilon)}{\epsilon} ,
	\end{equation}
	which exists by Lemma~\ref{lem684ad1f8}.
	
	Let $S\subset\scr N$ be a finite set.
	Define $E = \bigcup_{x\in S} B(x,r_x)$.
	Since the balls $\{B(x,\epsilon/2)\}_{x\in\scr N}$ are disjoint and since $r_x>\epsilon/2$, then
	\begin{equation}\label{eq684c5f6d}
		\mu(E)
		\ge \sum_{x\in S} \mu(B(x,\epsilon/2)) 
		\ge \mu^-(\epsilon/2) \# S .
	\end{equation}

	Define
	\begin{equation}
		\overset{\circ}S := \{x\in S: T_x\subset S\}
		= S\setminus\de_1(\scr N\setminus S) .
	\end{equation}
	
	We claim that, if $x\in \overset{\circ}S$, then $B(x,2\epsilon)\subset E$.
	Indeed, let $z\in B(x,2\epsilon)$.
	Since $\scr N$ is a maximal $\epsilon$-separated set in $X$, there exists $y\in\scr N$ with $d(z,y)\le \epsilon$.
	It follows that $d(x,y) \le d(x,z)+d(z,y) < 3\epsilon$ and thus $y\in T_x$.
	Since $T_x\subset S$, we have $B(y,r_y)\subset E$.
	Since $d(z,y)\le \epsilon < r_y$, then $z\in B(y,r_y)\subset E$.
	
	We claim that there exists $\delta>0$ such that, for every $x\in \overset{\circ}S$, we have $B(\de E,\delta) \cap \closure B(x,r_x) = \emptyset$.
	Indeed, set
	\begin{equation}
		\delta := \frac12 \min\{2\epsilon-r_x: x\in S\text{ with }T_x\subset S\} .
	\end{equation}
	By the finiteness of $S$ and the choice of $r_x$, we have $\delta>0$.
	If $x\in S$ is such that $T_x\subset S$, then $B(x,r_x+2\delta)\subset E$ by the previous claim, and thus $B(\de E,\delta) \cap \closure B(x,r_x) = \emptyset$.
	
	We claim that
	\begin{equation}\label{eq68ceb222}
		\Per(E) \le \frac{\mu^+(2\epsilon)}{\epsilon} m \#(\de_1S).
	\end{equation}
	The perimeter measure $\Per(E;\cdot)$ is supported on $\de E$.
	Moreover, by the above claim, $E\cap B(\de E,\delta) = \bigcup_{x\in S\setminus\overset{\circ}S} B(x,r_x) \cap B(\de E,\delta)$.
	Therefore,
	\begin{align}
		\Per(E)
		&= \Per(E;B(\de E,\delta)) \\
		&= \Per(\bigcup_{x\in S\setminus\overset{\circ}S} B(x,r_x) ; B(\de E,\delta)) \\
		&\le \Per(\bigcup_{x\in S\setminus\overset{\circ}S} B(x,r_x) )\\
		&\le \sum_{x\in S\setminus\overset{\circ}S} \Per(B(x,r_x)) \\
		&\overset{\eqref{eq684afc7e}}\le \#(S\setminus\overset{\circ}S) \frac{\mu^+(2\epsilon)}{\epsilon} .
	\end{align}
	Since $S\setminus\overset{\circ}S = \de_1(\scr N\setminus S)$
	and since $\scr N$ has finite order at most $m$, then $\#\de_1(\scr N\setminus S) \le m \#(\de_1S)$.
	We conclude that~\eqref{eq68ceb222} holds.
	
	Finally, let $V_\alpha,C_\alpha\in(0,+\infty)$ as in~\ref{prop684afea2_1}.
	If $S\subset \scr N$ is such that $\#S\ge V_\alpha/\mu^-(\epsilon/2)$, 
	then $\mu(E)\ge V_\alpha$ by \eqref{eq684c5f6d} and 
	\begin{align}
		(\#S)^\alpha 
		&\overset{\eqref{eq684c5f6d}}\le \frac1{\mu^-(\epsilon/2)^\alpha} \mu(E)^\alpha \\
		&\overset{\eqref{eq684c601c}}\le \frac{C_\alpha}{\mu^-(\epsilon/2)^\alpha} \Per(E) \\
		&\overset{\eqref{eq68ceb222}}\le \frac{C_\alpha m \mu^+(2\epsilon)}{\epsilon \mu^-(\epsilon/2)^\alpha}\#(\de_1S) .
	\end{align}
	We have proven~\ref{prop684afea2_2}.
	
	\framebox{$\ref{prop684afea2_2}\THEN\ref{prop684afea2_1}$}
	In the words of Kanai, this part of the proof follows a method by Buser~\cite{zbMATH03789440}.
	From the local isoperimetric inequality, Theorem~\ref{thm684c1dff}, we know that 
	for every $\alpha\in[\frac{Q-1}{Q},1]$ there is $C^\loc_\alpha\in(0,+\infty)$ 
	such that, for every $x\in X$ and $r\in(0,R/\sigma)$, and every $E\subset X$ measurable,
	\begin{equation}\label{eq684c3265}
		\left( \min\{\mu(B\cap E) , \mu(B\setminus E)\} \right)^\alpha \le C^\loc_\alpha \Per(E;\sigma B) ,
	\end{equation}
	where $B=B(x,r)$.
	
	
	We claim that, for every $x\in X$,
	\begin{equation}\label{eq684c26b0}
		\#\{y\in \scr N : x\in B(y,4\epsilon\sigma)\} 
		\le \frac{ \mu^+(\epsilon(4\sigma+1/2)) }{ \mu^-(\epsilon/2) } =: M < \infty
	\end{equation}
	Indeed, 
	$\mu(B(x,\epsilon(4\sigma+1/2))) 
	\ge \mu\left( \bigsqcup_{y\in \scr N\cap B(x,4\epsilon\sigma)} \mu(B(y,\epsilon/2)) \right) 
	\ge \#( \scr N\cap B(x,4\epsilon\sigma) ) \mu^-(\epsilon/2)$.
	Therefore, the hypothesis~\eqref{eq684c1886} implies~\eqref{eq684c26b0}.
		
	Let $V_\alpha,C_\alpha$ be the constants from~\ref{prop684afea2_2}.
	Let $E\subset X$ be a measurable set with $\mu(E)\ge \frac{V_\alpha}{2 \mu^+(\epsilon)}$.
	Define
	\begin{equation}
	\begin{aligned}
		S &:= \left\{ x\in\scr N: \mu(B(x,\epsilon)\cap E) > \frac{ \mu(B(x,\epsilon)) }{ 2 } \right\} , \text{ and, } \\
		P &:= (\scr N\setminus S) \cap B(E,\epsilon) .
	\end{aligned}
	\end{equation}
	
	We claim that
	\begin{equation}\label{eq684c27cf}
		\Per(E) \ge \frac{ \mu^-(\epsilon) }{ 2M C^\loc_1 } \#(\de_1S) .
	\end{equation}
	Indeed, if $y\in\de_1S$, then there is $x\in S$ such that $d(x,y)\le3\epsilon$. 
	It follows that, on the one hand,
	$\mu(B(y,4\epsilon)\cap E) 
	\ge \mu(B(x,\epsilon) \cap E )
	\ge \mu(B(x,\epsilon))/2 
	\ge \mu^-(\epsilon)/2$.
	On the other hand, 
	$\mu(B(y,4\epsilon)\setminus E)
	\ge\mu(B(y,\epsilon)\setminus E)
	\ge\mu(B(y,\epsilon))/2
	\ge\mu^-(\epsilon)/2$.
	Therefore, 
	\begin{equation}
		\Per(E;B(y,4\epsilon\sigma)) 
		\overset{\eqref{eq684c3265}}\ge \frac1{C^\loc_1} \min\{\mu(B(y,4\epsilon)\cap E), \mu(B(y,4\epsilon)\setminus E)\} 
		\ge \frac1{C^\loc_1} \mu^-(\epsilon)/2 .
	\end{equation}
	We conclude that
	\begin{align}
		\Per(E)
		&\overset{\eqref{eq684c26b0}}\ge \frac1M \sum_{y\in\de_1S} \Per(E;B(y,4\epsilon\sigma))
		\ge \frac{ \mu^-(\epsilon) }{ 2MC^\loc_1 } \#(\de_1S) .
	\end{align}
	This shows our claim~\eqref{eq684c27cf}.
	
	Next, we claim that
	\begin{equation}\label{eq684c2aba}
		\sum_{x\in P} \mu(B(x,\epsilon)\cap E)
		\le (C^\loc_\alpha)^{\frac1\alpha} M^{\frac1\alpha} \Per(E)^{\frac1\alpha} .
	\end{equation}
	Indeed,
	\begin{align}
		\sum_{x\in P} \mu(B(x,\epsilon)\cap E)
		&\overset{\eqref{eq684c3265}}\le \sum_{x\in P} (C^\loc_\alpha)^{\frac1\alpha} \Per(E;B(x,\epsilon))^{\frac1\alpha} \\
		&\le (C^\loc_\alpha)^{\frac1\alpha} \left( \sum_{x\in P} \Per(E;B(x,\epsilon)) \right)^{\frac1\alpha} \\
		&\overset{\eqref{eq684c26b0}}\le (C^\loc_\alpha)^{\frac1\alpha} M^{\frac1\alpha} \Per(E;\bigcup_{x\in P} B(x,\epsilon) )^{\frac1\alpha} \\
		&\le (C^\loc_\alpha)^{\frac1\alpha} M^{\frac1\alpha} \Per(E)^{\frac1\alpha} .
	\end{align}
	So, our claim~\eqref{eq684c2aba} is shown.
	
	We can now do the simple estimate
	\begin{equation}\label{eq684c6576}
	\begin{aligned}
		\mu(E)
		&\le \sum_{x\in S} \mu(B(x,\epsilon)\cap E) + \sum_{x\in P} \mu(B(x,\epsilon)\cap E) \\
		&\overset{\eqref{eq684c2aba}}\le \mu^+(\epsilon)\#S  +  (C^\loc_\alpha)^{\frac1\alpha} M^{\frac1\alpha} \Per(E)^{\frac1\alpha} \\
		&= \mu^+(\epsilon) \#S +  (C^\loc_\alpha)^{\frac1\alpha} M^{\frac1\alpha} \left( \frac{ \Per(E)^{\frac1\alpha} }{ \mu(E) }\right) \mu(E) .
	\end{aligned}
	\end{equation}
	We have two cases.
	First, if $(C^\loc_\alpha)^{\frac1\alpha} M^{\frac1\alpha} \left( \frac{ \Per(E)^{\frac1\alpha} }{ \mu(E) }\right) \ge 1/2$, then we directly have
	\begin{equation}\label{eq684c2c37}
		\mu(E) \le 2 (C^\loc_\alpha)^{\frac1\alpha} M^{\frac1\alpha} \Per(E)^{\frac1\alpha}.
	\end{equation}
	Second, if otherwise $(C^\loc_\alpha)^{\frac1\alpha} M^{\frac1\alpha} \left( \frac{ \Per(E)^{\frac1\alpha} }{ \mu(E) }\right) < 1/2$, then we obtain from the simple estimate~\eqref{eq684c6576}
	\begin{equation}\label{eq684c2ba1}
		\mu(E) \le 2 \mu^+(\epsilon) \#S  .
	\end{equation}
	Since we assume $\mu(E)\ge \frac{V_\alpha}{2 \mu^+(\epsilon)}$, we have $\#S\ge V_\alpha$ and thus
	\begin{equation}\label{eq684c67e6}
	\begin{aligned}
		\mu(E)^\alpha
		&\overset{\eqref{eq684c2ba1}}\le (2 \mu^+(\epsilon) \#S )^\alpha \\
		&\overset{\eqref{eq684c6658}}\le (2 \mu^+(\epsilon) )^\alpha C_\alpha \#\de_1S \\
		&\overset{\eqref{eq684c27cf}}\le (2 \mu^+(\epsilon) )^\alpha C_\alpha \frac{ 2MC^\loc_1 }{ \mu^-(\epsilon) } \Per(E) .
	\end{aligned}
	\end{equation}
	Putting together the two cases~\eqref{eq684c2c37} and~\eqref{eq684c67e6}, we conclude that, if $\mu(E)\ge \frac{V_\alpha}{2 \mu^+(\epsilon)}$, then 
	\begin{equation}
		\mu(E)^\alpha \le K \Per(E) ,
	\end{equation}
	where $K=\max\left\{ (2^\alpha C^\loc_\alpha M , (2 \mu^+(\epsilon) )^\alpha C_\alpha \frac{ 2MC^\loc_1 }{ \mu^-(\epsilon) } \right\}$.
\end{proof}

\begin{lemma}\label{lem684ad1f8}
	Let $(X,d,\mu)$ be a metric measure space and $a,b\in[0,+\infty)$ with $a<b$ and set
	\begin{equation}
		\mu^+(b) := \sup\{\mu(B(x,b)) : x\in X\} .
	\end{equation}
	Then, for every $x\in X$ there exists $r\in(a,b)$ such that
	\begin{equation}
		\Per(B(x,r)) \le \frac{ \mu(B(x,b)\setminus B(x,a)) }{ b-a } .
	\end{equation}
\end{lemma}
\begin{proof}
	Let $x\in X$ and set $d_x(y)=d(x,y)$.
	Then $d_x$ is a 1-Lipschitz function: in particular, the constant function $1$ is an upper gradient of $d_x$.
	Using the coarea formula from Theorem~\ref{thm68458049}, we obtain
	\begin{align}
		\mu^+(b)
		\ge \mu(B(x,b)\setminus B(x,a))
		\overset{\eqref{eq68559ec1}}\ge \int_a^b \Per(B(x,t)) \did t .
	\end{align}
	It follows that there is $r\in(a,b)$ with $\Per(B(x,r)) \le \frac{ \mu^+(b) }{ b-a } $.
\end{proof}

\begin{remark}
	The proof of $\ref{prop684afea2_1}\THEN\ref{prop684afea2_2}$ in Proposition~\ref{prop684afea2} does not use the bounded geometry of $X$.
	In fact, this part of the above proof works in every metric measure space, under the conditions~\eqref{eq684c1886} (with $\sigma=1$).
\end{remark}

\begin{proof}[Proof of Theorem~\ref{thm684c4f0a}]
	Quasi-geodesic metric spaces are quasi-isometric to every maximal $\epsilon$-separated subset endowed with a word metric, for every $\epsilon>0$.
	It follows that, if $X$ and $Y$ are quasi isometric, then every net in $X$ is quasi-isometric to every net in $Y$.
	Moreover, on a metric measure space with locally uniform bounded geometry, there is always $\epsilon>0$ such that all the hypotheses of Proposition~\ref{prop684afea2} are met.
	Theorem~\ref{thm684c4f0a} then follows from the combination of Propositions~\ref{prop684afea2} and~\ref{prop684c58f4}.
\end{proof}

\subsection{The isoperimetric inequality at large implies the Sobolev inequality}

We will show that an isoperimetric inequality at large scale, as in~\eqref{eq67dbf890}, implies a Gagliardo--Nirenberg--Sobolev type inequality, as in~\eqref{eq67f66fec_1}.
Proposition~\ref{prop6855a6f8} is inspired by~\cite[Proposition 4.5.3]{zbMATH07060424}
and~\cite[Proof of Theorem 1.3 and Proposition 3.1]{zbMATH05956848}.

\begin{proposition}\label{prop6855a6f8}
	Let $(X,d,\mu)$ be a metric measure space.
	Suppose that $X$ supports a large scale isoperimetric inequality of order $\frac{N-1}{N}$ for some $N>1$,
	that is, that there are $C_{\rm I},V_{\rm I}$ such that, for all $E\subset X$ measurable,
	\begin{equation}\label{eq67dbf890}
		\text{if $\mu(E)\ge V_{\rm I}$, then }
		\mu(E)^{\frac{N-1}{N}} \le C_{\rm I}\cdot \Per(E) .
	\end{equation}
	Then, for every $q\in[1,N)$ there exist constants $C_q$ and $C_{q,\infty}$ such that,
	for every $u:X\to\R$ locally integrable with upper gradient $g:X\to[0,+\infty]$,
	\begin{equation}\label{eq67f66fec_1}
		\left(\int_X |u|^{\frac{Nq}{N-q}} \did\mu \right)^{\frac{N-q}{Nq}}
		\le C_q \left(\int_X g^q \did\mu \right)^{\frac1q}
			+ C_{q,\infty} \|u\|_{L^\infty(X)} .
	\end{equation}
\end{proposition}
\begin{proof}
	Before diving into the proof, recall the following formula from real analysis, see \cite[(3.34)]{MR1800917}:
	If $\delta\in(0,1)$ and $F:(0,+\infty)\to(0,+\infty)$ is non-decreasing, then 
	\begin{equation}\label{eq67dbf7e6}
		\int_0^\infty s^{\frac1\delta-1} F(s) \did s 
		\le \delta \left(\int_0^\infty F(s)^\delta \did s\right)^{1/\delta} .
\end{equation}
	
	Let $u:X\to\R$ be locally integrable with upper gradient $g:X\to[0,+\infty]$.
	We will consider the following condition:
	\begin{equation}\label{eq67dc177c}
		\text{for all $t\ge0$, either $\mu(\{|u|>t\}) \ge V_{\rm I}$ or $\mu(\{|u|>t\}) = 0$} .
	\end{equation}
	
	We claim that, if $u$ satisfies the condition~\eqref{eq67dc177c}, then
	\begin{equation}\label{eq68558ec5}
		\left( \int_X |u|^{\frac{N}{N-1}} \did\mu \right)^{\frac{N-1}{N}}
		\le C_{\rm I} \int_X g \did\mu .
	\end{equation}
	Indeed, setting $\alpha= \frac{N}{N-1}$, we have
	\begin{align*}
		&\int |u|^\alpha \did\mu \\
		[\text{Cavalieri}]
		&= \int_0^\infty \mu(\{|u|^\alpha > t \}) \did t \\
		[s=t^{1/\alpha},\ s^\alpha = t,\ \alpha s^{\alpha-1} \did s = \did t]
		&= \alpha \int_0^\infty \mu(\{|u|>s\}) s^{\alpha-1} \did s \\
		[\text{by~\eqref{eq67dbf7e6} with $\delta=1/\alpha$}]
		&\le \alpha \frac1\alpha \left(\int_0^\infty \mu(\{|u|>s\})^{1/\alpha} \did s\right)^{\alpha} \\
		[\text{since $\alpha = \frac{N}{N-1}$}]
		&= \left(\int_0^\infty \mu(\{|u|>s\})^{\frac{N-1}{N}} \did s\right)^{\frac{N}{N-1}} \\
		[\text{\eqref{eq67dc177c} and isop. ineq.~\eqref{eq67dbf890}}]
		&\le \left(\int_0^\infty C_{\rm I}\cdot \Per(\{|u|>s\}) \did s\right)^{\frac{N}{N-1}} \\
		[\text{coarea inequality~\eqref{eq68559ec1}}]
		&\le C_{\rm I}^{\frac{N}{N-1}} \left(\int_X g \did \mu \right)^{\frac{N}{N-1}} .
	\end{align*}
	We thus have obtained~\eqref{eq68558ec5} under the condition~\eqref{eq67dc177c}.
	
	Next, we claim that, if $u$ satisfies the condition~\eqref{eq67dc177c} and $q\in[1,N)$, then 
	\begin{equation}\label{eq685599b5}
		\left(\int_X |u|^{\frac{Nq}{N-q}} \did\mu \right)^{\frac{N-q}{Nq}}
		\le C_{\rm I} \frac{(N-1)q}{N-q} \left(\int_X g^q \did\mu \right)^{\frac1q} .
	\end{equation}
	Indeed, we first notice that also $|u|^p$ satisfies the condition~\eqref{eq67dc177c}, for every $p>0$.
	Second, if $p\ge 1$ then $p|u|^{p-1}g$ is an upper gradient of $|u|^p$.
	Therefore, applying~\eqref{eq68558ec5} to $|u|^p$, we obtain
	\begin{align*}
	 	\left(\int_X |u^p|^{\frac{N}{N-1}} \did\mu \right)^{\frac{N-1}{N}}
		&\overset{\eqref{eq68558ec5}}\le C_{\rm I} \int_X p|u|^{p-1}g \did\mu \\
		[\text{Hölder inequality}]
		&\le C_{\rm I}p \left(\int_X |u|^{(p-1)q'} \did\mu \right)^{\frac1{q'}}
			\left(\int_X g^q \did\mu \right)^{\frac1q} ,
	\end{align*}
	where $q' = \frac{q}{q-1}$ is the Hölder exponent conjugate to $q$.
	Since we want $(p-1)q' = p\frac{N}{N-1}$, we obtain 
	$p= \frac{N-1}{N-q} q$.
	Then we have $p\frac{N}{N-1} = \frac{Nq}{N-q}$.
	Assuming $u\neq0$, then obtain on the left-hand side the integral of $|u|^{\frac{Nq}{N-q}}$ exponentiated to $\frac{N-1}{N} - \frac1{q'} = \frac{N-q}{Nq}$.
	This concludes the proof of~\eqref{eq685599b5} under the condition~\eqref{eq67dc177c}.
	
	Finally, suppose that $u:X\to\R$ is a locally integrable function that may not satisfy~\eqref{eq67dc177c}.
	Let $\xi\ge0$ be such that $\mu(\{|u|>t\})\ge V_{\rm I}$ for $t\in(0,\xi)$ and $\mu(\{|u|>t\})\le V_{\rm I}$ for $t\ge\xi$.
	Then $|u| = u_1 + u_2$, where $u_1 = \min\{|u|,\xi\}$ and $u_2 = |u|-u_1$.
	We have that $u_1$ satisfies the condition~\eqref{eq67dc177c}: 
	indeed, if $t\in(0,\xi)$, then $\mu(\{|u_1|>t\}) = \mu(\{|u|>t\}) \ge V_{\rm I}$;
	if $t>\xi$, then $\mu(\{|u_1|>t\}) = 0$.
	Notice also that $g$ is an upper gradient of $u_1$,
	and that $u_2$ is supported on the set $\{|u|>\xi\}$ with $\mu(\{|u|>\xi\})\le V_{\rm I}$.
	Hence,
	\begin{align*}
		\left(\int_X |u|^{\frac{Nq}{N-q}} \did\mu \right)^{\frac{N-q}{Nq}}
		&\le \left(\int_X |u_1|^{\frac{Nq}{N-q}} \did\mu \right)^{\frac{N-q}{Nq}} 
			+ \left(\int_X |u_2|^{\frac{Nq}{N-q}} \did\mu \right)^{\frac{N-q}{Nq}} \\
		&\overset{\eqref{eq685599b5}}\le C_{\rm I} \frac{(N-1)q}{N-q} \left(\int_X g^q \did\mu \right)^{\frac1q} 
			+ \left(\int_{\{|u|>\xi\}} |u_2|^{\frac{Nq}{N-q}} \did\mu \right)^{\frac{N-q}{Nq}} \\
		&\le C_{\rm I} \frac{(N-1)q}{N-q} \left(\int_X g^q \did\mu \right)^{\frac1q}
			+ V_{\rm I}^{\frac{N-q}{Nq}} \|u\|_{L^\infty} .
	\end{align*}
	So, we have~\eqref{eq67f66fec_1} with $C_q=C_{\rm I} \frac{(N-1)q}{N-q}$ and $C_{q,\infty}=V_{\rm I}^{\frac{N-q}{Nq}}$.
\end{proof}

\begin{corollary}\label{cor6863f5f0}
	Let $(X,d,\mu)$ be a metric measure space
	with compact sets of arbitrarily large measure.
	If $X$ supports a large-scale isoperimetric inequality of order $\frac{N-1}{N}$ for some $N>1$,
	then $X$ is $p$-hyperbolic for every $p<N$.
	In particular,
	\begin{equation}\label{eq68780aab}
		\dimisp(X)\le \dimpar(X) .
	\end{equation}
\end{corollary}
\begin{proof}
	Combine Proposition~\ref{prop6855a6f8} and Proposition~\ref{prop67f055d9}.
\end{proof}

\begin{remark}
	Combining Corollary~\ref{cor6863f5f0} with Proposition~\ref{prop67f0559a},
	we obtain the following statement:
	If $X$ has degree of large-scale growth at most~$N_1$ and supports a large-scale isoperimetric inequality of order $\frac{N_2-1}{N_2}$, then $N_2 \le N_1$.
\end{remark}

\section{The hyperbolic case}

\subsection{The hyperbolic Ferrand distance}\label{subs67dd9a95} 

\begin{definition}\label{def of ferrdHyp}
	Let $X$ be a metric measure space with Hausdorff dimension $Q$.
	The \emph{hyperbolic Ferrand  (pseudo)distance} on $X$ is the function $\ferrdHyp :X\times X\to [0,+\infty]$,
	\begin{equation}
		\ferrdHyp(x,y) := \inf\left\{
		\capacity_Q(K) : K\text{ continuum with }x,y\in K
		\right\},
		\qquad\forall x,y\in X.
	\end{equation}
\end{definition}
Here $\capacity_Q$ denotes the $Q$-capacity as in Definition~\ref{def6803ccb2}.

The map $\ferrdHyp$ is a pseudo-distance, that is, it satisfies the triangle inequality.
Indeed, if $x,y,z\in X$ and $K_{xy}$, $K_{yz}$ are continua with $x,y\in K_{xy}$ and $y,z\in K_{yz}$, then $K_{xy}\cup K_{yz}$ is a continuum containing $x$ and $z$, and $\capacity_Q(K_{xy}\cup K_{yz}) \le \capacity_Q(K_{xy}) + \capacity_Q(K_{yz})$,
because capacity is sub-additive; see Remark~\ref{rem67f8c4ee}.
It follows that $\ferrdHyp(x,z) \le \ferrdHyp(x,y) + \ferrdHyp(y,z)$.

The Ferrand pseudo-distance $\ferrdHyp$ is not a distance when there is a pair $a,b\in X$ with $\ferrdHyp(a,b)=0$ and $a\neq b$.
The following result shows that, under certain hypothesis on the space $X$, there are only two cases: either $\ferrdHyp$ is a distance, i.e., $\ferrdHyp(a,b)>0$ whenever $a\neq b$, 
or $\capacity_Q(K)=0$ for every continuum $K\subset X$, i.e., $X$ is parabolic.

\begin{theorem}[{\cite[Theorem 6.8]{zbMATH00883879}}]\label{thm680560ae}
	Let $(X,d,\mu)$ be a path-connected metric measure space with uniformly locally $Q$-bounded geometry.
	If there are $a,b\in X$ distinct such that $\ferrdHyp(a,b)=0$, then $\capacity_Q(K)=0$ for every compact connected subset $K$ of $X$.
\end{theorem}
\begin{proof}
	$\ferrdHyp(a,b)=0$ means that, for every $\epsilon>0$ there exists a continuum $K_\epsilon\subset X$ with $a,b\in K_\epsilon$, and a function $u_\epsilon:X\to[0,1]$ continuous with compact support, $K_\epsilon\subset \{u_\epsilon\ge1\}$, and
	$\int_X g_\epsilon^Q \did x \le \epsilon$, where $g_\epsilon$ is an upper gradient of $u_\epsilon$.
	By Proposition~\ref{prop67e505c3} and Lemma~\ref{lem67e51b8c}, we can assume that $u_\epsilon$ is monotone in $X\setminus \{x\}$ for every $x\in K_\epsilon$.
	
	Let $K\subset X$ be a nonempty continuum.
	
	Suppose that $a\notin K$ but $b\in K$.
	Since $a\in K_\epsilon$, $u_\epsilon$ is monotone in $X\setminus \{a\}$.
	By Corollary~\ref{cor67efffbc}, there are $C,R\in\R$ such that, whenever $r\in(0,R]$ and $x\in X$ are such that $B(x,Cr)\subset X\setminus\{a\}$, we have 
	$\oscillation_{B(x,r)} u_\epsilon \le C\epsilon^{1/Q}$, for every $\epsilon>0$.
	Take $r\in(0,R)$ such that $Cr<{\rm dist}(a,K)$, and cover $K$ with a finite number, say $n$, of balls of radius $r$ centered in $K$.
	Then, $\oscillation_Ku_\epsilon \le nC\epsilon^{1/Q}$.
	Since $u_\epsilon (b)=1$ and $b\in K$, then 
	$u_\epsilon(x) \ge 1- nC\epsilon^{1/Q}$ for all $x\in K$.
	Take $\epsilon$ small enough so that $1-nC\epsilon^{1/Q} > 1/2$.
	Then the function $v(x) = 2u_\epsilon(x)$ is continuous with compact support and with $K\subset \{v\ge1\}$.
	Moreover, an upper-gradient for $v$ is $2g_\epsilon$.
	So,
	$\capacity_Q(K) \le \int_X (2g_\epsilon)^Q \did \mu
	\le 2^Q \epsilon$.
	Since $\epsilon$ is arbitrarily small, we conclude that $\capacity_Q(K) =0$.
	
	Suppose that $a,b\notin K$.
	Let $\gamma:[0,1]\to X$ be a path with $\gamma(0)\in K$ and $\gamma(1)=b$.
	Let $\bar t = \inf\{t\in[0,1]:\gamma(t)\in\{a,b\}\}$.
	Then, $\bar t>0$, $\gamma(\bar t)\in\{a,b\}$, and $\gamma(t)\notin\{a,b\}$ for all $t<\bar t$.
	Without loss of generality, $\gamma(\bar t)=b$ and $a\notin\gamma([0,\bar t])$.
	Then $K\cup\gamma([0,\bar t])$ is a continuum containing $b$ but not $a$, and thus, by the previous argument and the monotonicity of capacity,
	$\capacity_Q(K) \le \capacity_Q(K\cup\gamma([0,\bar t])) =0$.
		
	If $a,b\in K$, we can take $x \in X\setminus K$ and, as we did above, a path $\gamma\subset X\setminus\{a\}$ from $x$ to $b$.
	By the previous paragraph, $\capacity_Q(\gamma)=0$ and thus $\ferrdHyp(x,b)=0$.
	Applying the previous paragraph switching $a$ with $x$, we conclude $\capacity_Q(K)=0$.
\end{proof} 


\begin{corollary}\label{cor67f92736}
	Let $(X,d,\mu)$ be a path-connected metric measure space with uniformly locally $Q$-bounded geometry.
	If there exists a continuum $K\subset X$ such that $\capacity_Q(K)\neq0$, 
	then $\ferrdHyp(a,b)=0$ if and only if $a=b$.
	In particular, the pseudodistance $\ferrdHyp$ is a distance.
\end{corollary}

On spaces with locally bounded $Q$-geometry, the $Q$-capacity is bi-Lipschitz preserved by quasi-conformal maps, by the geometric Definition~\ref{def68cc0a5c}.
For this reason, the $Q$-Ferrand distance plays a role in the study of quasi-conformal maps.

\begin{proposition}\label{prop67f68305}
	Let $X,Y$ be metric measure spaces with uniformly locally $Q$-bounded geometry,
	and $f:X\to Y$ a quasi-conformal map.
	Then there is $L$ such that
	\begin{equation}
		\frac{1}{L} \ferrdHyp(p,q) \le \ferrdHyp(f(p),f(q)) \le L \ferrdHyp(p,q) .
	\end{equation}
\end{proposition}

\subsection{Upper bound to the hyperbolic Ferrand distance}

\begin{lemma}\label{lem685ce33b}
	Let $(X,d,\mu)$ be a boundedly compact metric measure space with Hausdorff dimension $Q$.
	Suppose that there is $\rho:(0,+\infty)\to (0,+\infty)$ such that for every $x,y\in X$ there is a continuum joining $x$ and $y$ in $B(x,\rho(d(x,y)))$.
	Define $\mu^+:[0,+\infty)\to[0,+\infty]$ by
	\begin{equation}
		\mu^+(r) := \sup\{\mu(B(x,r)) : x\in X\},\qquad \text{ for } r\in  [0,+\infty).
	\end{equation}
	Then, we have
	\begin{equation}\label{eq6a70fca4}
		\ferrdHyp(x,y) \le \mu^+(\rho(d(x,y))+1) ,\qquad \forall x,y\in X.
	\end{equation}
	In particular, if $\mu^+(r)<\infty$ for every $r$, then whenever $\lim_{n\to\infty}\ferrdHyp(x_n,y_n)=\infty$, we have $\lim_{n\to\infty} d(x_n,y_n)=\infty$.
\end{lemma}
\begin{proof}
	Let $x,y\in X$ and $r=d(x,y)$.
	Since $x,y$ are joined by a continuum in $B(x,\rho(r))$, then 
	\begin{equation}
		\ferrdHyp(x,y) 
		\le \capacity_Q(B(x,\rho(r))) .
	\end{equation}
	Consider the function $u(z) = \max\{\rho(r) + 1 - d(x,z), 0\}$.
	Then $B(x,\rho(r))\subset\{u\ge1\}$ and $\spt(u) \subset B(x,\rho(r)+2)$.
	Since $X$ is boundedly compact, $u$ has compact support and thus it is admissible for $B(x,\rho(r))$.
	An upper gradient of $u$ is $g= \one_{B(x,\rho(r)+1)}$.
	Therefore, 
	\begin{equation}
		\capacity_Q(B(x,\rho(r)) 
		\le \mu(B(x,\rho(r)+1)) 
		\le \mu^+(\rho(d(x,y))+1) .
	\end{equation}
\end{proof}

\subsection{Lower bound to the hyperbolic Ferrand distance}

\begin{theorem}\label{thm67e655fa}
	Let $(X,d,\mu)$ be a boundedly compact metric measure space with uniformly locally $Q$-bounded geometry, $Q>1$.
	Assume that $X$ is of hyperbolic conformal type
	and that there exist $p\ge1$, $C_Q>0$ and $C_\infty\ge0$ such that the following Gagliardo--Nirenberg–-Sobolev type inequality,
	\begin{equation}\label{eq67f66fec}
		\left(\int_X |u|^{p} \did\mu \right)^{1/p}
		\le C_Q \left(\int_X g^Q \did\mu \right)^{\frac1Q} + C_\infty \|u\|_{L^\infty(X)} ,
	\end{equation}
	holds for every $u\in L^1_{\rm loc}(X)$ with upper gradient $g$.
	
	Then, there is $C_{\rm F}\ge1$ such that 
	for every $a,b\in X$ with $d(a,b)>C_{\rm F}$,
	\begin{equation}\label{eq67f00646}
		\ferrdHyp(a,b)
		\ge \frac1{C_{\rm F}} d(a,b)^{\frac{1}{pQ}} ,
	\end{equation}
	where $\ferrdHyp$ denotes the hyperbolic Ferrand distance.
\end{theorem}
\begin{proof}
	Let $R$ and $C$ be as in Corollary~\ref{cor67efffbc}, and fix $r\in(0,R/C)$.	
	Fix $a,b\in X$ with $d(a,b) > 4CR$.
	
	Since $X$ is hyperbolic, then $\ferrdHyp(a,b)>0$ by Corollary~\ref{cor67f92736}.
	Let $K\subset X$ be a continuum with $a,b\in K$ and $\capacity_Q(K) \le 2^Q\ferrdHyp(a,b)$.
	Let $u:X\to[0,1]$ be an admissible function for $K$, monotone on $X\setminus \{z\}$ for every $z\in K$,
	and $g:X\to[0,+\infty]$ an upper-gradient of $u$,
	such that $\int_X g^Q\did\mu \le 2^Q\capacity_Q(K)$.

	Let $\{x_j\}_{j=1}^N\subset K$ be a $(2Cr)$-separated set
	with 
	\begin{equation}\label{eq67f6777f}
		\frac{d(a,b)}{3} \le NCr .
	\end{equation}
	Notice that, since $d(a,b) > 4C_1r$, for each $j$ we have $K\setminus B(x_j,C_1r) \neq\emptyset$, and thus, $u$ is monotone on $B(x_j,C_1r)$.
	
	
	CASE 1: $\int_X g^Q\did\mu \ge \frac{N}{2} \frac{1}{(2C)^Q}$.
	Then
	\begin{align*}
		\scr F(a,b) 
		&\ge 4^{-Q} \int_X g^Q\did\mu \\
		&\ge 4^{-Q} \frac{N}{2} \frac{1}{(2C)^Q} \\
		&\overset{\eqref{eq67f6777f}}\ge 4^{-Q} \frac{d(a,b)}{6Cr} \frac{1}{(2C)^Q} . 
	\end{align*}
	Since $pQ>1$, this lower bound implies \eqref{eq67f00646} for $d(a,b)\ge1$.
	
	CASE 2: $\int_X g^Q\did\mu < \frac{N}{2} \frac{1}{(2C)^Q}$.
	Then there are $N/2$ points in $\{x_j\}_{j=1}^N$ such that $\int_{B(x_j,Cr)}g^Q\did\mu \le \frac{1}{(2C)^Q}$.
	Then, if $x_j$ is one of such points, for every $y\in B(x_j,r)$ we have
	\begin{equation}\label{eq67f00c04}
	\begin{aligned}
		u(y) 
		&\ge 1 - \oscillation_{B(x_j,r)}u \\
		&\overset{\eqref{eq67efd746}}\ge 1 - C \left( \int_{B(x_j,Cr)} g^Q \did \mu \right)^{1/Q} \\
		[\text{by the choice of }j]
		&\ge 1 - C \frac{1}{2C} = \frac12 .
	\end{aligned}
	\end{equation}
	Therefore,
	\begin{align*}
	 	\scr F(a,b)^{1/Q}
		&\ge \frac14 \left( \int_X g^Q \did\mu \right)^{1/Q} \\
		&\overset{\eqref{eq67f66fec}}\ge \frac{1}{4C_Q} \left( \int_X u^p \did\mu \right)^{1/p} - \frac{C_\infty}{4C_Q} \|u\|_{L^\infty(X)} \\
	[\text{balls are disjoint and $|u|\le1$}]
		&\ge \frac{1}{4C_Q} \left( \sum_{j=1}^N \int_{B(x_j,r)} u^p \did\mu \right)^{1/p}  
			- \frac{C_\infty}{4C_Q} \\
	[\text{by Ahlfors regularity and~\eqref{eq67f00c04}}]
		&\ge \frac{1}{4C_Q} \left( \frac{N}{2} \frac1{C_{\rm A}} r^Q  \right)^{1/p} \frac12 
			 - \frac{C_\infty}{4C_Q}\\
	[\text{by~\eqref{eq67f6777f}}]
		&\ge \frac{1}{8C_Q} \left( \frac{ r^{Q-1} }{ 6C_{\rm A}C } \right)^{1/p} d(a,b)^{1/p} 
			 - \frac{C_\infty}{4C_Q} .
	\end{align*}
	This completes the proof of~\eqref{eq67f00646}.
\end{proof}

%

\begin{corollary}\label{cor6855a606}
	Let $(X,d,\mu)$ be a boundedly compact metric measure space with uniformly locally $Q$-bounded geometry, $Q>1$.
	Suppose that $X$ supports a large-scale isoperimetric inequality of order $\frac{N-1}{N}$ for some $N>Q$, as in~\eqref{eq67dbf890}.
	Suppose also that $X$ has compact sets of arbitrarily large measure.
	
	Then there is $C_{\rm F}>0$ such that, for every $a,b\in X$ with $d(a,b)>C_{\rm F}$,
	\begin{equation}\label{eq6855a684}
		\ferrdHyp(a,b)
		\ge \frac1{C_{\rm F}} d(a,b)^{\frac{1}{pQ}} ,
	\end{equation}
	where $p=\frac{NQ}{N-Q}>1$ and $\ferrdHyp$ is the hyperbolic Ferrand distance. 
%
\end{corollary}
\begin{proof}
	We apply Proposition~\ref{prop6855a6f8} to obtain~\eqref{eq67f66fec_1} for $q=1$, that is~\eqref{eq67f05467};
	thus, Proposition~\ref{prop67f055d9} implies that $X$ has hyperbolic conformal type, because $Q<N$.
	We apply again Proposition~\ref{prop6855a6f8} to obtain~\eqref{eq67f66fec} for $q=Q$,
	because $Q<N$;
	hence, Theorem~\ref{thm67e655fa} implies~\eqref{eq6855a684}.
\end{proof}

\subsection{QC implies QI in the hyperbolic case}\label{subs6866e125}

\begin{theorem}\label{thm685cec65}
	Let $X_1$ and $X_2$ be metric measure spaces 
	with uniformly locally $Q$-bounded geometry with $Q>1$.
	For both $j\in\{1,2\}$, suppose that $X_j$ is boundedly compact, and quasi-convex, 	and that the function $\mu^+_{X_j}:r\in (0,+\infty)\mapsto \sup\{\mu(B(x,r)):x\in X\}$ is finite valued and $\lim_{r\to\infty}\mu^+_{X_j}(r)=\infty$.
	If $\dimisp(X_1)>Q$ and $\dimisp(X_2)>Q$,
	then every quasi-conformal map $X_1\to X_2$ is a quasi-isometry.
\end{theorem}
\begin{proof}
	Let $\ferrdHyp_j$ be the hyperbolic Ferrand distance on $X_j$, for $j\in \{1, 2\}$.
	
	First, we claim that the identity map $(X_j,d_j)\to (X_j,\ferrdHyp_j)$ is bornologous.
	To prove this claim, fix $(X,d,\ferrdHyp)\in\{(X_1,d_1,\ferrdHyp_1),(X_2,d_2,\ferrdHyp_2)\}$.
	By assumption, there is $C>0$ such that $X$ is $C$-quasi-convex.
	Therefore, for every $r>0$ and for every $x,y\in X$ with $x,y\in B(o,r)$,
	there is a curve $\gamma$ from $x$ to $y$ with $\Length(\gamma)\le Cd(x,y)$, and thus, $\gamma\subset B(x,Cd(x,y)+1)$.
	We can then apply Lemma~\ref{lem685ce33b} with $\rho(r) := Cr+1$:
	if $\{(x_n,y_n)\}_{n\in\N}\subset X\times X$ is a sequence with $\lim_{n\to\infty}\ferrdHyp(x_n,y_n)=\infty$, then 
	 $\lim_{n\to\infty}d(x_n,y_n) = \infty$.
	This shows our claim that the identity map $(X_j,d_j)\to (X_j,\ferrdHyp_j)$ is bornologous.	
	
	Second, also the hypotheses of Corollary~\ref{cor6855a606} are met for both spaces.
	We obtain that, for $j\in\{1,2\}$, if $\{(x_n,y_n)\}_{n\in\N}\subset X_j\times X_j$ is a sequence with $\lim_{n\to\infty} d(a_n,b_n)=\infty$, then $\lim_{n\to\infty} \ferrdHyp(a_n,b_n) = \infty$.
	We conclude that the identity map $(X_j,d_j)\to (X_j,\ferrdHyp_j)$ is bornologous, for both $j\in\{1,2\}$.
	
	Next, since quasi-conformal maps are biLipschitz for the hyperbolic Ferrand distance,
	a quasi-conformal map $f:(X_1,d_1)\to (X_2,d_2)$ is composition of bornologous maps
	\begin{equation}
		(X_1,d_1)
		\overset{\Id}\to (X_1,\ferrdHyp_1) 
		\overset{f}\to (X_2,\ferrdHyp_2) 
		\overset{\Id}\to (X_2,d_2) ,
	\end{equation}
	and thus bornologous.
	Since the inverse of $f$ is also quasi-conformal, it is bornologous too. 
	
	Notice that quasi-convex spaces are quasi-geodesic. 
	Therefore, we conclude that $f$ is a quasi-isometry	 by Corollary~\ref{cor67f68188}.
\end{proof}

\section{Geodesic Lie groups}

\subsection{Geodesic Lie groups}
A \emph{geodesic Lie group} is a Lie group endowed with a left-invariant geodesic distance $d$ inducing the manifold topology, and it is equipped with a left Haar measure $\mu$.
In fact, every  geodesic Lie group is a sub-Finsler Lie group $G$:
it is equipped with a bracket-generating polarization $\Delta$ of its Lie algebra $\frk g$ and a norm $\|\cdot\|$ on $\Delta$.
The sub-Finsler structure defines the \emph{geodesic distance} (a.k.a.~\emph{Carnot-Carathéodory   distance}): 
\begin{equation*}
	d(x,y) := \inf\left\{
	\int_0^1\| {\rm d} L_{\gamma(t)}^{-1}\dot\gamma(t)\|\did t : 
	\begin{array}{l}
		\gamma:[0,1]\to G \text{ absolutely continuous},\\
		\gamma(0)=x,\ \gamma(1) =y,\\
		{\rm d}L_{\gamma(t)}^{-1}\dot\gamma(t)\in\Delta\text{ for a.e.~}t\in[0,1]
	\end{array}
	\right\} ,
\end{equation*}
which is a geodesic distance that induces the manifold topology, as $\Delta$ is assumed to be bracket generating; see \cite{donne2024metricliegroupscarnotcaratheodory} for a detailed exposition of sub-Finsler Lie groups.

Recall that a Lie group is \emph{unimodular} if every left-invariant Haar measure is also right invariant.
Lie groups with polynomial growth are unimodular.

\begin{proposition}\label{prop682dc2dc}
	A geodesic Lie group $(G,d,\mu)$ is a boundedly compact metric measure space with uniformly locally $Q$-bounded geometry,
	where $Q$ is the Hausdorff dimension of $(G,d)$ and it equals
	\begin{equation}\label{eq684e8752}
		Q = \sum_{j=1}^\infty k \dim(\Delta^{[k]}/\Delta^{[k-1]}) ,
	\end{equation}
	with $\Delta^{[0]} := \{0\}$, $\Delta^{[1]} := \Delta$, and $\Delta^{[k+1]} := \Delta^{[k]} + [\Delta,\Delta^{[k]}]$ for all $k\ge1$.
	
	If $G$ is unimodular, then a $(1,1)$-Poincaré inequality holds at all scales.
\end{proposition}
\begin{proof}
	Since geodesic Lie groups are locally compact geodesic complete metric spaces, they are boundedly compact; see \cite[Theorem 7.1.7]{donne2024metricliegroupscarnotcaratheodory}.
	
	It is well known that geodesic Lie groups have locally $Q$-bounded geometry, 
	with $Q$ as in~\eqref{eq684e8752}.
	Indeed, the local Ahlfors regularity is proven in~\cite[Ch 1, Sec 3]{MR793239}.
	Local Poincaré inequalities $(G,d,\mu)$ are proven in~\cite{MR850547}, with \cite[Theorem~11.12]{MR1683160}, for the sub-Riemannian case;	the general case of geodesic Lie groups is deduced from the fact that every geodesic Lie group is bi-Lipschitz equivalent to a sub-Riemannian Lie group.
	The uniformity of constants is a consequence of the left-invariance of both the metric $d$ and the measure $\mu$.
	So, we conclude that geodesic Lie groups have uniformly locally $Q$-bounded geometry.
	
	Another source for a proof of the (1,1)-Poincaré inequality in geodesic Lie groups is \cite[p.461]{zbMATH00819164}, which clearly shows that, when $G$ is unimodular, then the inequality holds at all scales.
\end{proof}

Pittet~\cite{zbMATH01782641} characterized the isoperimetric profiles at large scales of Riemannian Lie groups.
Kanai's Theorem~\ref{thm684c4f0a} allows us to extend the result by Pittet from Riemannian to geodesic Lie groups.
The isoperimetric inequality at large scale implies a Gagliardo–Nirenberg–Sobolev type inequality, as we have shown in Proposition~\ref{prop6855a6f8}.
Sobolev inequalities are one of our main tools to study capacity on geodesic Lie groups,
as we have shown in Proposition~\ref{prop67f055d9} and Theorem~\ref{thm67e655fa}. 
Sobolev inequalities on nilpotent geodesic Lie groups have been shown also by Varopoulos~\cite{MR839109}, or~\cite[Theorem IV.7.2]{zbMATH05521653},
and on Lie groups with polynomial growth by
Alexopoulos and Lohoué in~\cite{zbMATH00561474}.

\begin{theorem}[{after Pittet \cite[Theorem 2.1]{zbMATH01782641}}]\label{thm68655128}
	Let $(G,d,\mu)$ be a non-compact geodesic Lie group.
	Then, exactly one of the following cases happens:
	\begin{enumerate}[label=(\alph*)]
	\item
	$G$ is non-amenable or non-unimodular, with exponential volume growth,
	and it supports an isoperimetric inequality at large scale of order $\alpha$,
	for every $\alpha\in[0,1]$.
	In this case, $\dimgr(G)=\infty$.
	\item
	$G$ is amenable, unimodular, with exponential volume growth,
	and it supports an isoperimetric inequality at large scale of order $\alpha$,
	for every $\alpha\in[0,1)$.
	In this case, $\dimgr(G)=\infty$.
	\item
	$G$ is amenable, unimodular, and there is $N\in\N$ 
	such that $G$ has polynomial growth of degree $N$, 
	and it supports an isoperimetric inequality at large scale of order $\alpha$,
	for every $\alpha\in[0,\frac{N-1}{N}]$.
	In this case, $\dimgr(G)=N$.
	\end{enumerate}
	As a consequence,
	\begin{equation}\label{eq6866de65}
		\dimisp(G)
		= \dimpar(G)
		= \dimgr(G) .
	\end{equation}
\end{theorem}
\begin{proof}
	Let $\rho$ be a left-invariant Riemannian metric on $G$.
	Pittet's result \cite[Theorem 2.1]{zbMATH01782641}
	implies that one of the three cases listed in the theorem hold for $(G,\rho,\mu)$.

	We claim that the same case holds for $(G,d,\mu)$.
	Indeed, since both $\rho$ and $d$ are left-invariant (quasi-)geodesic distances on $G$, they are quasi-isometric equivalent.
	By Proposition~\ref{prop6866b133}, the volume growth of $(G,d,\mu)$ and $(G,\rho,\mu)$ are equivalent, that is, they are either both exponential or both polynomial with the same degree.
	Theorem~\ref{thm684c4f0a} tells us that, for every $\alpha\in[0,1]$,
	either both $(G,d,\mu)$ and $(G,\rho,\mu)$ support an isoperimetric inequality at large scale of order $\alpha$, or none.
	Our claim is proven.
	
	To deduce~\eqref{eq6866de65}, first recall that, by~\eqref{eq68cf049a} and~\eqref{eq68780aab}, we have
	\begin{equation}
		\dimisp(G)
		\le \dimpar(G)
		\le \dimgr(G).
	\end{equation}
	If $G$ falls in one of the first two cases, then $\dimisp(G)=\infty$ and thus equalities hold.
	If $G$ has polynomial growth, then case (c) claims that there is $N$ such that $\dimgr(G)\le N$ and $N\le\dimisp(G)$.
	Therefore, we have necessarily $\dimgr(G) = N = \dimisp(G)$ and~\eqref{eq6866de65}.
\end{proof}

\begin{corollary}\label{cor67f056ce}
	Geodesic Lie groups are boundedly compact spaces with uniformly locally bounded geometry.
	Moreover, if $G$ is a non-compact geodesic Lie group with Hausdorff dimension $Q$ and growth dimension $N\in[1,\infty]$,
	then one and only one of the following two cases is true:
	\begin{enumerate}[label=(\Roman*)]
	\item
	Either $N\le Q$, and $G$ has parabolic conformal type;
	\item
	Or $Q<N$, and $G$ has hyperbolic conformal type.
	\end{enumerate}
\end{corollary}

The identities in~\eqref{eq6866de65} have been proven for Riemannian manifolds by~Coulhon and Saloff-Coste in~\cite{zbMATH00404300}.

\subsection{Quasi-straightenable geodesic Lie groups}

\begin{proposition}\label{prop687570ac}
	Every geodesic Lie group with polynomial growth is quasi-straight\-en\-able
	in the sense of Section~\ref{subs68cf0748}.
\end{proposition}
\begin{proof}
	Every geodesic Lie group with polynomial growth is quasi-isometric to a simply connected nilpotent Riemannian Lie group;	see, for instance, \cite[Corollary 101]{zbMATH07939438}.
	By Lemma~\ref{lem6874bbe0}, simply connected nilpotent Riemannian Lie groups are quasi-straightenable.
	By Lemma~\ref{lem6874bc72}, quasi-straightenability is preserved by quasi-isometries.
	This completes the proof.
\end{proof}

\begin{lemma}\label{lem6874bbe0}
	Simply connected nilpotent Riemannian Lie groups are quasi-straightenable.
\end{lemma}
\begin{proof}
	We will make use of the fact that
	\begin{equation}\label{eq6874d78e}
		\alpha^\theta + \beta^\theta \le 2 (\alpha+\beta)^\theta , \qquad \forall \theta, \alpha,\beta\in(0,+\infty) .
	\end{equation}
	
	Let $(N,d)$ be a simply connected nilpotent Riemannian Lie group, with Lie algebra $\frk n$ and exponential map $\exp:\frk n\to N$.
	Denote by $\|\cdot\|$ the Riemannian norm on $\frk n$.
	By Guivarc’h Theorem, see for instance \cite[Theorem 10.4.3]{donne2024metricliegroupscarnotcaratheodory},
	there is a splitting $\frk n = \bigoplus_{j=1}^s V_j$ of $\frk n$ into subspaces such that, for every $v_1,\dots,v_s$ with $v_j\in V_j$ for all $j$, 
	\begin{equation}\label{eq6874da38}
		\frac1C \sum_{m=1}^s a_m \|v_m\|^{1/m} - C
		\le d(1_N,\exp(\sum_{m=1}^n v_m) ) 
		\le C \sum_{m=1}^s a_m \|v_m\|^{1/m} + C ,
	\end{equation}
	for some constants $C>0$ and $\{a_j\}_{j=1}^s\subset(0,+\infty)$.
	Moreover, up to increasing the constant $C$,
	we also assume that, if $v = \sum_{m=1}^sv_m \in\frk n$ with $v_m\in V_m$ for each $m\in\{1,\dots,s\}$, then 
	\begin{equation}\label{eq6874f517}
		\sum_{m=1}^s \|v_m\|
		\le C\| v \| 
		\le C \sum_{m=1}^s \|v_m\| .
	\end{equation}
	
	Set
	\begin{equation}\label{eq6874f56f}
		K = \max\left\{ 
		 2C^2, \,2 C^3 + C , \,\sup\{d(1_N,\exp(u)) : \|u\|\le 1\} \right\} .
	\end{equation}
	We claim that, for every vector $v\in\frk n$ with $\|v\|=1$,
	the sequence $\zeta:\Z\to N$, $\zeta_k:= \exp(kv)$, is an $K$-quasi-straight sequence. 
	Indeed, write $v = \sum_{m=1}^sv_m$ with $v_m\in V_m$ for each $m\in\{1,\dots,s\}$.
	First we have that $\zeta(\Z_{\ge0})$ and $\zeta(\Z_{\le0})$ are unbounded, because $\exp:\frk n\to N$ is a homeomorphism and $v\neq0$.
	Second, 
	\begin{equation}
		d(\zeta_k,\zeta_{k+1}) 
		= d(1_N, \exp(v)) 
		\overset{\eqref{eq6874f56f}}\le K .
	\end{equation}
	Third, if $i,j,k\in\Z$ are such that $i\le j\le k$, then
	\begin{align}
		d(\zeta_i,\zeta_j) + d(\zeta_j,\zeta_k)
		&= d(1_N,\exp((j-i)v)) + d(1_N,\exp((j-k)v)) \\
		&\overset{\eqref{eq6874da38}}\le C \sum_{m=1}^s a_m (|j-i|^{1/m} + |k-j|^{1/m}) \|v_m\|^{1/m} + C \\
		&\overset{\eqref{eq6874d78e}}\le C \sum_{m=1}^s a_m 2 |k-i|^{1/m} \|v_m\|^{1/m} + C \\
		&\overset{\eqref{eq6874da38}}\le 2 C^2 d(1_N,\exp(\sum_{m=1}^s (k-i) v_m) )
			+ 2 C^3 + C \\
		&\overset{\eqref{eq6874f56f}}\le K d(\zeta_i,\zeta_k) + K .
	\end{align}
	This shows the claim.
	
	Finally, let $x,y\in N$ be distinct points.
	Define $w := \exp^{-1}(x^{-1}y)$, so that $y=x\exp(w)$, and let $v = \frac{w}{\|w\|}$.
	Then $v$ is a unit vector and so the function $\zeta:\Z\to N$, $\zeta_k:= x\exp(kv)$, for $k\in\Z$, defines a $K$-quasi-straight sequence with $\zeta_0=x$.
	Moreover, if $k = \lfloor\|w\|\rfloor$, then 
	\begin{align}
		d(y,\zeta_k) 
		&= d(x\exp(w),x\exp(kv))
		= d(1_N, \exp((\|w\|-k)v)) \\
		&\le \sup\{d(1_N,\exp(u)) : \|u\|\le 1\} 
		\overset{\eqref{eq6874f56f}}\le K.
	\end{align}
	Therefore, $x,y\in B(\zeta(\Z),K)$.
	This shows that $(N,d)$ is quasi-straightenable.
\end{proof}

\subsection{Quasi-conformal maps between geodesic Lie groups}

We are now ready to prove our main Theorem~\ref{thm68760387}, which we spell out for the sake of clarity.

\begin{theorem}\label{thm68780db5}
	Let $G_1$ and $G_2$ be non-compact geodesic Lie groups of the same Hausdorff dimension $Q$, and let $f: G_1\to G_2$ be a quasi-conformal map.
	For each $j\in\{1,2\}$, denote by $N_j$  the growth dimension of $G_j$.
	Then $N_1=N_2$ and one of the following cases happens:
	\begin{enumerate}
	\item
	$Q > N_1 $ and $f$ is a quasi-isometry:
	This is the strictly parabolic case.
	\item
	$Q = N_1$:
	This is the liminal parabolic case.
	\item
	$Q < N_1$ and $f$ is a quasi-isometry:
	This is the hyperbolic case.
	\end{enumerate}
\end{theorem}
\begin{proof}
	If the Hausdorff dimension $Q$ equals 1, then the spaces are the Euclidean line, so we are trivially in case (2). Next, we assume $Q>1$.
	
	In cases (1) and (2), when $N_1\le Q$, $G_1$ has parabolic conformal type by Corollary~\ref{cor67f056ce}.
	Hence, $G_2$ has also parabolic conformal type, see Remark~\ref{rem689ee744}, and thus $N_2\le Q$ by Corollary~\ref{cor67f056ce}.
	Since $Q<\infty$, both $G_1$ and $G_2$ have polynomial volume growth.
	By Proposition~\ref{prop687570ac}, they are both quasi-straightenable.
	
	We claim that $Q=N_1$ if and only if $Q=N_2$.
	Notice that a geodesic Lie group with liminal parabolic type 
	is Ahlfors regular at all scales and it supports the Poincaré inequality at all scales by Proposition~\ref{prop682dc2dc}; see also \cite[p.461]{zbMATH00819164}.
	Therefore, combining Proposition~\ref{prop687658f1} and Proposition~\ref{prop68755b1a},
	we conclude that $N_1<Q$ (or $N_2<Q$) if and only if the parabolic Ferrand distance of $G_1$ (or $G_2$) takes finite values.
	The quasi-conformal invariance of the parabolic Ferrand distance, as shown in Proposition~\ref{prop68755c04}, implies that $Q=N_1$ if and only if $Q=N_2$.

	In case (1), when $N_1 < Q$ and $N_2 < Q$,
	we can apply Theorem~\ref{thm68755ab4} and obtain that $f$ is a quasi-isometry.
	Since the growth dimension is a quasi-isometric invariant by Proposition~\ref{prop6866b133}, we conclude $N_1=N_2$.
	
	In case (3), when $Q < N_1$, then $G_1$ has hyperbolic conformal type by Corollary~\ref{cor67f056ce}.
	Hence, $G_2$ has also hyperbolic conformal type and thus $Q<N_2$, again by Corollary~\ref{cor67f056ce}.
	We conclude by Theorem~\ref{thm685cec65} that $f$ is a quasi-isometry.
	Again by Proposition~\ref{prop6866b133}, we conclude $N_1=N_2$.
\end{proof}

For large classes of spaces, growth dimension is a quasi-isometric invariant, but need not be quasi-conformally invariant. 
It turns out to be quasi-conformally invariant in the realm of geodesic Lie groups, as a consequence of the following Corollary.

\begin{corollary}\label{cor67f6897a}
	Among non-compact geodesic Lie groups, the growth dimension, the large-scale isoperimetric dimension, and the parabolic dimension are quasi-conformal and quasi-isometric invariants.
\end{corollary}
\begin{proof}
	By~\eqref{eq6866de65}, the three dimensions are equal to each other for non-compact geodesic Lie groups.
	As we have shown in Proposition~\ref{prop6866b133} and Theorem~\ref{thm684c4f0a}, $\dimgr$ and $\dimisp$ are quasi-isometric invariant, hence also $\dimpar$ is quasi-isometric invariant.
	By Theorem~\ref{thm68780db5}, the existence of a quasi-conformal homeomorphism implies either that the function is a quasi-isometry, or that both domain and codomain are liminal parabolic.
	In both cases, the three dimensions are preserved.
\end{proof}

\begin{remark}
	The proof of Theorem~\ref{thm68780db5} uses, even though not explicitly, that the liminal parabolic case corresponds to the case when the geodesic Lie group is a Loewner space.
\end{remark}

\subsection{Lie groups with infinite fundamental group}

Recall that a map $\pi:X\to Y$ between metric spaces is a \emph{submetry} if 
\begin{equation}
	\pi(\bar B(x,r)) = \bar B(\pi(x),r)
	\qquad \forall x\in X, \forall r>0 .
\end{equation}
For more information on submetries, see~\cite[\S3.1.7]{donne2024metricliegroupscarnotcaratheodory}.

\begin{lemma}\label{lem68d44a76}
	For each $j\in\{1,2\}$, let $\pi_j:(\tilde X_j,\tilde d_j)\to (X_j,d_j)$ be a submetry between metric spaces.
	Let $f:X_1\to X_2$ and $\tilde f:\tilde X_1\to \tilde X_2$ be homeomorphisms such that $\pi_1\circ\tilde f = f\circ\pi_1$, i.e., the following diagram commutes:
	\begin{equation}\label{eq68d44523}
	\begin{minipage}{\textwidth}
		\xymatrix{ 
		\tilde X_1 \ar[d]_{\pi_1} \ar[r]^{\tilde f}& \tilde X_2 \ar[d]^{\pi_2} \\
		X_1 \ar[r]_{ f} & X_2 
		}
	\end{minipage}
	\end{equation}
	For every $L,C\ge0$, 
	if $\tilde f$ is an $(L,C)$-quasi-isometry, then $f$ is an $(L,C)$-quasi-isometry too.
\end{lemma}
\begin{proof}
	Let $z_1,z_2\in X_1$.
	Since $\pi_1$ is a submetry, there are $\tilde z_1,\tilde z_2\in X_1$ such that $\pi_1(\tilde z_k) = z_k$ for each $k\in\{1,2\}$ and $d_1(z_1,z_2) = \tilde d_1(\tilde z_1,\tilde z_2)$.
	Since $\tilde f$ is an $(L,C)$-quasi-isometry and since $\pi_2$ is 1-Lipschitz, we have
	\begin{align}
		d_2(f(z_1),f(z_2))
		&\overset{\eqref{eq68d44523}}= d_2(\pi_2(\tilde f(\tilde z_1)), \pi_2(\tilde f(\tilde z_2)) ) 
		\le \tilde d_2(\tilde f(\tilde z_1), \tilde f(\tilde z_2)) \\
		&\le L \tilde d_1(\tilde z_1,\tilde z_2) + C 
		= L d_1(z_1,z_2 ) + C .
	\end{align}
	Since the inverse of $f$ and $\tilde f$ satisfy the same conditions, we conclude that also the inequality $d_2(f(z_1),f(z_2)) \ge 1/L d_1(z_1,z_2 ) - C$ is also fulfilled.
	We conclude that $f$ is an $(L,C)$-quasi-isometry.
\end{proof}

\begin{theorem}\label{thm68d55353}
	Let $f:G_1\to G_2$ be a quasi-conformal map between non-compact geodesic Lie groups.
	If the fundamental group of $G_1$ is infinite, then $f$ is a quasi-isometry.
\end{theorem}
\begin{proof}
%
	Let $Q$ be the Hausdorff dimension of $G_1$ and $N$ the growth dimension of $G_1$.
	If $N\neq Q$, then $f$ is a quasi-isometry by Theorem~\ref{thm68780db5},
	so we assume $N=Q$.
	
	For $j\in\{1,2\}$, let $\tilde G_j$ be the universal covering of $G_j$,
	endowed with the geodesic distance induced by the one on $G_j$.
	The quotient maps $\pi_j:\tilde G_j\to G_j$ are submetries.
	Since $f$ is a quasi-conformal homeomorphism, then it lifts to a quasi-conformal homeomorphism $\tilde f:\tilde G_1\to\tilde G_2$.
	Let $Q$ be the Hausdorff dimension and growth dimension of $G_1$.
	Then $Q$ is also the Hausdorff dimension of $\tilde G_1$,
	but, by Lemma~\ref{lem68db76be}, the growth dimension of $\tilde G_1$ is strictly larger than $Q$ because the fundamental group of $G_1$ is infinite.
	Therefore, $\tilde G_1$ is of hyperbolic type, and thus $\tilde f$ is a quasi-isometry by Theorem~\ref{thm68780db5}.
	By Lemma~\ref{lem68d44a76}, also $f$ is a quasi-isometry.
\end{proof}

\begin{lemma}\label{lem68db76be}
	Let $\pi:\tilde G\to G$ be a covering map and submetry of connected metric Lie groups with infinite kernel.
	Let $\tilde \mu$ be a Haar measure on $\tilde G$ and $\mu$ a Haar measure on $G$.
	Then there are $C>0$ and $k\ge1$ such that
	\begin{equation}\label{eq68db7678}
		\tilde\mu(B_{\tilde G}(1_{\tilde G}, 2 R)) \ge C \mu(B_G(1_G, R)) \cdot R^k ,
		\qquad\forall R>1 .
	\end{equation}
\end{lemma}
\begin{proof}
We denote by $d$ the metric on $G$, and by  $ \tilde d$ the metric on $\tilde G$.
	Let $Z := \ker(\pi)$.
	We claim that there are $k\ge1$ and $C>0$ such that
	\begin{equation}\label{eq68db6e37}
		\#(B_{\tilde G}(1_{\tilde G},R)\cap Z) \ge C R^k,
		\qquad\forall R>1 .
	\end{equation}
	Indeed, since $Z$ is a normal and discrete subgroup of $\tilde G$, which is connected,
	then $Z$ is central and in particular abelian.
	So, there are a finite group $T\subset Z$ and $\{e_i\}_{i=1}^k\subset Z$ for some $k\in\N$,  such that $Z \simeq T \times \sum_{i=1}^k\Z\cdot e_i$.
	Since $Z$ is infinite, then $k\ge1$.
	If $z\in Z$, then there are $t\in T$ and $\{z_i\}_{i=1}^k\subset\Z$ such that
	$z = t + \sum_{i=1}^k z_i e_i$.
	Therefore, $\tilde d(1_{\tilde G},z) \le \tilde d(1_{\tilde G},t) + \sum_{i=1}^k |z_i|\tilde d(1_{\tilde G},e_i)$.
	Set $A:= \max\{\tilde d(1_{\tilde G},e_i)\}_{i=1}^k$.
	If we take $t=1_{\tilde G}$ and $|z_i|< \frac{R}{Ak}$, we obtain 
	$\tilde d(1_{\tilde G},z) < R$.
	Therefore, the set $B_{\tilde G}(1_{\tilde G},R)\cap Z$ contains at least $\left(\frac{R}{Ak}\right)^k$ elements and we get~\eqref{eq68db6e37}.

	Fix $r>0$, e.g., $r=1$.
	For $R>1$, let $\{x_j^R\}_{j=1}^{N(R)}\subset G$ be a maximal $r$-separated set in $B_G(1_G,R)$, for some $N(R)\in\N$.
	Then we have, for every $R>1$,
	\begin{equation}\label{eq68db6dd4}
		\mu(B_G(1_G,R)) \le N(R) \cdot \mu(B_G(1_G,r)) .
	\end{equation}
	Since $\pi$ is a submetry, we can lift these sets to $\{\tilde x_j^R\}_{j=1}^{N(R)}\subset \tilde G$ such that $\pi(\tilde x_j^R) =  x_j^R$ and $d(1_{\tilde G},\tilde x_j^R) = d(1_G, x_j^R)$, for every $R$ and $j$.
	For $R>1$, consider the set 
	\begin{equation}
		\cal X(R) := \left\{\tilde x_j^Rz : j\in\{1,\dots,N(R)\},\ z\in B_{\tilde G}(1_{\tilde G},R)\cap Z \right\} .
	\end{equation}
	If $i\neq j$ and $z,z'\in B_{\tilde G}(1_{\tilde G},R)\cap Z$, then 
	$\tilde d(\tilde x_i^Rz, \tilde x_j^Rz') \ge d(x_i^R,x_j^R) \ge r$,
	and, if $z\neq z'$, then $\tilde d(\tilde x_i^Rz, \tilde x_i^Rz') = \tilde d(1_{\tilde G},z^{-1}z') \ge \epsilon$,
	where $\epsilon := \min\{\tilde d(1_{\tilde G},z):z\in Z\setminus\{1_{\tilde G}\}\} > 0$.
	Moreover, $\tilde d(1_{\tilde G},\tilde x_j^Rz) \le d(1_G,x_j^R) + \tilde d(1_{\tilde G},z) \le 2R$.
	Therefore, for $\delta := \min\{r,\epsilon\}$,
	\begin{align}
		\tilde\mu(B_{\tilde G}(1_{\tilde G}, 2R + \delta))
		&\ge \sum_{y\in \cal X(R)} \tilde \mu( B_{\tilde G}(y,\delta/2) )
		= \#\cal X(R) \tilde \mu(B_{\tilde G}(1_{\tilde G},\delta/2))  \\
		&= N(R) \cdot \# (B_{\tilde G}(1_{\tilde G},R)\cap Z)
		\overset{\eqref{eq68db6e37}\&\eqref{eq68db6dd4}}\ge \frac{\mu(B_G(1_G,R))}{\mu(B_G(1_G,r))} \cdot C R^k .
	\end{align}
	We thus get~\eqref{eq68db7678}.
\end{proof}

\subsection{Nilpotent geodesic Lie groups}\label{subs68768d93}

Theorem~\ref{thm68780db5} leaves an incomplete picture for the ``QC implies QI'' for geodesic Lie groups of liminal parabolic type.
In the class of nilpotent geodesic Lie groups, we are able to fill this gap.
We show in Theorem~\ref{thm67f935d5} that nilpotent geodesic Lie groups that are quasi-conformally equivalent are also quasi-isometric.

Nilpotent groups have polynomial growth.
In particular, the growth dimension $N$ 
of a simply connected nilpotent Lie group with Lie algebra $\frk g$
can be computed algebraically by the \emph{Bass–Guivarc’h formula}
\begin{equation}\label{eq67f689d5}
	N = \sum_{k=1}^\infty \dim(\frk g^k) ,
\end{equation}
where $\frk g^1 = \frk g$, and $\frk g^{k+1} = [\frk g,\frk g^k]$;
see~\cite[Théorèm II.4]{zbMATH03460877} or \cite[Theorem 2]{zbMATH03409630}.
See also \cite[Corollary 2.9]{zbMATH06371838}.

\begin{definition}\label{def68d4ec34}
	A geodesic Lie group $(G,\Delta,\|\cdot\|)$ is a \emph{Carnot group} if $G$ is simply connected and $\Delta$ is the first layer of a stratification,
	which means that 
	the Lie algebra $\frk g$ of $G$ has a splitting $\frk g = \oplus_{j=1}^s V_j$ with $V_{j+1}=[V_1,V_j]$ for all $j$, and with $\Delta = V_1$;
	see for instance \cite[Chapter 11]{donne2024metricliegroupscarnotcaratheodory}.
\end{definition}

The following proposition describes the dichotomy for simply connected nilpotent geodesic Lie groups.
Afterwards, we will treat the general case of nilpotent geodesic Lie groups.

\begin{proposition}\label{prop67f92fca}
	Let $G$ be a simply connected nilpotent geodesic Lie group with Hausdorff dimension $Q$ and growth dimension $N$.
	Then $Q\le N$.
	Moreover,
	\begin{enumerate}[label=(\Roman*)]
	\item
	Either $Q=N$, $G$ is a Carnot group and $G$ has liminal parabolic conformal type;
	\item
	Or $Q<N$, $G$ is not a Carnot group and $G$ has hyperbolic conformal type.
	\end{enumerate}
\end{proposition}
\begin{proof}
	Let $\frk g$ be the Lie algebra of $G$ and $\Delta\subset\frk g$ the bracket-generating polarization of $G$.
	Therefore, $Q$ is given by~\eqref{eq684e8752}, while $N$ by~\eqref{eq67f689d5}.
	 
	The assumption that $\Delta$ is bracket generating implies
	\begin{equation}\label{eq67f78c42}
	\forall k\ge0,\ 
		\Delta^{[k]} + \frk g^{k+1} = \frk g .
	\end{equation}
	Indeed, if~\eqref{eq67f78c42} were false, then there would exist $k\ge 1$ with $\Delta^{[k]} + \frk g^{k+1} \neq \frk g$ (the case $k=0$ is trivial).
	Taking the quotient $\frk h := \frk g/\frk g^{k+1}$ of $\frk g$ by the ideal $\frk g^{k+1}$, we obtain a nilpotent Lie algebra $\frk h$ of step $k$ with $\Delta^{[k]}/\frk g^{k+1} \neq \frk h$.
	Since $\frk h$ has nilpotency step $k$, then $\Delta^{[m]}/\frk g^{k+1} = \Delta^{[k]}/\frk g^{k+1}$ for all $m>k$.
	Therefore, $\Delta/\frk g^{k+1}$ would be not bracket generating, in contradiction with $\Delta$ being bracket generating in $\frk g$.
	This shows our claim~\eqref{eq67f78c42}.
	
	Let $s$ be the nilpotency step of $\frk g$.
	From~\eqref{eq67f78c42}, we get $\Delta^{[s]}=\frk g$, because $\frk g^{s+1}=\{0\}$.
	If $s=1$, i.e., $\frk g$ is abelian, then clearly $Q=N$ and $G$ is a Carnot group.
	Assume $s\ge2$.
	Again from~\eqref{eq67f78c42}, we obtain for each $k\in\{0,\dots,s-1\}$
	\begin{equation}\label{eq684e87d8}
	\begin{aligned}
	 	\dim\frk g^{k+1}
		&\overset{\eqref{eq67f78c42}}\ge \dim\frk g - \dim \Delta^{[k]} 
		= \sum_{j=k+1}^s (\dim\Delta^{[j]} - \dim \Delta^{[j-1]} ) \\
		&= \sum_{j=k+1}^s \dim(\Delta^{[j]}/\Delta^{[j-1]}) .
	\end{aligned}
	\end{equation}
	Therefore,
	\begin{align*}
		Q  \overset{\eqref{eq684e8752}}= \sum_{j=1}^s j \dim(\Delta^{[j]}/\Delta^{[j-1]}) 
		= \sum_{k=1}^s \sum_{j=k}^s \dim(\Delta^{[j]}/\Delta^{[j-1]}) 
		\overset{\eqref{eq684e87d8}}\le \sum_{k=1}^s \dim\frk g^k 
		= N .
	\end{align*}
	We have thus shown that $Q\le N$.
	
	Equality $Q=N$ holds if and only if, for each $k\in\{0,\dots,s-1\}$, 
	there is equality also in~\eqref{eq684e87d8},
	i.e., $\dim\frk g = \dim \Delta^{[k]} + \dim\frk g^{k+1}$.
	Therefore, $Q=N$ holds if and only if the sum in~\eqref{eq67f78c42} is a direct sum.
	Consequently, equality $Q=N$ holds if and only if $\Delta$ is the first layer of a stratification of $\frk g$.
	
	We conclude that, if $Q=N$ then $G$ is a Carnot group,
	and $G$ is liminal parabolic by Corollary~\ref{cor67f056ce}.
	If $Q\neq N$ then $Q<N$ and $G$ is not a Carnot group,
	and $G$ is hyperbolic by Corollary~\ref{cor67f056ce}.
\end{proof}

\begin{theorem}\label{thm67f935d5}
	Let $f: G_1\to G_2$ be a metrically quasi-conformal map between nilpotent geodesic Lie groups.
	Then $G_1$ and $G_2$ have the same Hausdorff dimension, which we denote by $Q$.
	Denote by $N_j$ the growth dimension of $G_j$, for each $j\in\{1,2\}$.
	Then $N_1=N_2$ and one of the following two exclusive cases happens:
	\begin{enumerate}[label=(\Roman*)]
	\item
	Either both $G_1$ and $G_2$ are simply connected and have liminal parabolic conformal type, and thus they are algebraically isomorphic and metrically bi-Lipschitz Carnot groups.
	\item
	Or $f$ is a quasi-isometry.
	\end{enumerate}
\end{theorem}
\begin{proof}
	First of all, metrically quasi-conformally equivalent geodesic Lie groups have the same Hausdorff dimension by~\cite[Corollary~6.4]{MR1334873}. If the Hausdorff dimension $Q$ equals 1, then the spaces are either the Euclidean line or a circle, so the conclusions are trivial. Next, we assume $Q>1$.
	
	From Remark~\ref{equiv_def_QC}, the map $f$ is a geometrically quasi-conformal map.
	By Theorem~\ref{thm68780db5}, we obtain that $N:=N_1=N_2$ and that whenever $Q < N$ or $Q>N$, then $f$ is a quasi-isometry.
	If $N=Q$, that is, if $G_1$ and $G_2$ have liminal parabolic conformal type, then we distinguish two cases.
	
	On the one hand, if $G_1$ and $G_2$ are simply connected, then they are Carnot groups by Proposition~\ref{prop67f92fca};	since by \cite[Theorem 2]{MR979599} differentiation of metrically quasi-conformal maps between Carnot groups induces bi-Lipschitz isomorphisms of Lie groups,	we conclude that $G_1$ and $G_2$ are algebraically isomorphic and metrically bi-Lipschitz Carnot groups.
	
	On the other hand, suppose that $G_1$ and $G_2$ are not simply connected.
	Since they are nilpotent, their fundamental groups are infinite (this is because simply connected Lie groups have no torsion; see \cite[Proposition 9.4.26.i]{donne2024metricliegroupscarnotcaratheodory}).
	By Theorem~\ref{thm68d55353}, the quasi-conformal map $f$ is a quasi-isometry.
\end{proof}

\section{Further results on quasiconformal maps}\label{sec6a91e2a8}

\subsection{Quasiconformal maps are quasisymmetries}\label{sec6a797eb7}

It has been proven in \cite[Theorem 9.8]{MR1869604} that, on spaces with locally controlled geometry, quasiconformal maps are locally quasisymmetric.
From \cite[Corollary 4.10.]{MR1654771}, we also know that quasiconformal maps between Q-regular Loewner spaces are \emph{globally} quasisymmetric".
Here we show that in hyperbolic and parabolic spaces, which are not Loewner, quasiconformal maps are quasisymmetries too.

\begin{definition}
	Let $(X,d)$ and $(Y,d)$ be metric spaces.
	A map $f:X\to Y$ is \emph{quasisymmetric} if there exists a homeomorphism $\eta:[0,+\infty)\to[0,+\infty)$ such that, for every $x,a,b\in X$,
	\begin{equation}\label{eq6a798015}
		\frac{ d(f(x),f(a)) }{ d(f(x),f(b)) } \le \eta\left( \frac{ d(x,a) }{ d(x,b) } \right) .
	\end{equation}
\end{definition}


Notice that we can assume the homeomorphism $\eta$ to be concave on $[0,1]$.

\begin{proposition}[QC + QI = QS]\label{prop6a79864e}
	Let $(X,d)$ and $(Y,d)$ be path-connected metric spaces and $f:X\to Y$.
	Suppose that there are $R,L>0$ and a homeomorphism $\eta:[0,+\infty)\to[0,+\infty)$ 
	such that $f$ is $\eta$-quasisymmetric on balls of radius $R$ and
	$L$-biLipschitz at distance larger than $R$.
	More explicitly, suppose that
	\begin{equation}\label{eq6a798787}
		\frac{ d(f(x),f(a)) }{ d(f(x),f(b)) } \le \eta\left( \frac{ d(x,a) }{ d(x,b) } \right) 
		\qquad\forall x,a,b\in X\text{ with }d(x,a),d(x,b) \le R ;
	\end{equation}
	and
	\begin{equation}\label{eq6a79878c}
		\frac1L d(a,b) \le d(f(a),f(b)) \le L d(a,b) 
		\qquad\forall a,b\in X\text{ with }d(a,b) \ge R .
	\end{equation}
	Then $f$ is quasisymmetric.
	Quantitatively, if we assume $\eta$ to be concave on $[0,1]$, then $f$ is $\tilde\eta$-quasisymmetric with
	\begin{equation}\label{eq6a91d828}
		\tilde\eta(t) = \max\left\{\eta(t) , L^2 \cdot t \cdot \eta(t) , L^2 t \right\} , 
		\qquad\forall t\in[0,+\infty) .
	\end{equation}
\end{proposition}
\begin{proof}
	Without loss of generality, we assume $\eta$ to be concave on $[0,1]$.
	We need to show the inequality~\eqref{eq6a798015} with $\tilde\eta$ as in~\eqref{eq6a91d828} for $x,a,b\in X$ with $d(x,a)\ge R$ or $d(x,a)\ge R$.
	Given $x,a,b\in X$, we distinguish three cases.
	
	{\it Case 1: $d(x,a)\ge R$ and $d(x,b)\ge R$.}
	Then
	\begin{equation}
		\frac{ d(f(x),f(a)) }{ d(f(x),f(b)) } 
		\overset{\eqref{eq6a79878c}} \le L^2 \frac{ d(x,a) }{ d(x,b) } 
		\le \tilde\eta\left( \frac{ d(x,a) }{ d(x,b) }  \right) .
	\end{equation}
	
	{\it Case 2: $d(x,a)\ge R$ and $d(x,b)\le R$.}
	Since $X$ is path connected and $d(x,a)\ge R$, there is $a'\in X$ with $d(x,a') = R$.
	By~\eqref{eq6a79878c}, $d(f(x),f(a'))>0$.
	Hence,
	\begin{align}
		\frac{ d(f(x),f(a)) }{ d(f(x),f(b)) }
		&= \frac{ d(f(x),f(a')) }{ d(f(x),f(b)) } 
			\cdot \frac{ d(f(x),f(a)) }{ d(f(x),f(a')) }
		\overset{\eqref{eq6a798787}\&\eqref{eq6a79878c}}\le \eta\left( \frac{ d(x,a') }{ d(x,b) }  \right) 
			\cdot \frac{ L^2 d(x,a) }{ d(x,a') } \\
		&= \eta\left( \frac{ R }{d(x,a)} \frac{d(x,a)}{ d(x,b) }  \right) 
			\cdot \frac{ L^2 d(x,a) }{ d(x,b) } \frac{ d(x,b) }{ R } \\
		&\overset{(*)}\le  \eta\left( \frac{d(x,a)}{ d(x,b) }  \right) 
			\cdot L^2 \cdot \frac{ d(x,a) }{ d(x,b) } 
		\le \tilde\eta\left( \frac{ d(x,a) }{ d(x,b) }  \right) ,
	\end{align}
	where in $(*)$ we used the facts $\frac{ R }{d(x,a)} \le 1$ with the monotonicity of $\eta$, and $\frac{ d(x,b) }{ R }\le 1$.
	
	{\it Case 3: $d(x,a)\le R$ and $d(x,b)\ge R$.}
	For this case, we need $\eta$ to be concave on $[0,1]$, that is, for every $s,t,c\in[0,1]$,
	\begin{equation}\label{subs6a79cd8f}
		\eta(c t + (1-c) s) \ge c \eta(t) + (1-c)\eta(s) .
	\end{equation}
	
	Since $X$ is path connected and $d(x,b)\ge R$, there is $b'\in X$ with $d(x,b') = R$.
	Then
	\begin{align}
		\frac{ d(f(x),f(a)) }{ d(f(x),f(b)) }
		&= \frac{ d(f(x),f(a)) }{ d(f(x),f(b')) } 
			\cdot \frac{ d(f(x),f(b')) }{ d(f(x),f(b)) }
		\le \eta\left( \frac{ d(x,a) }{ d(x,b') }  \right)
			\cdot \frac{ L^2 d(x,b') }{ d(x,b) } \\
		&= \eta\left( \frac{ d(x,a) }{ R } \right)
			\cdot L^2 \cdot \frac{  R }{ d(x,b) } 
		\overset{(*)}\le L^2 \eta\left( \frac{ d(x,a) }{ d(x,b) } \right)
	\end{align}
	where in $(*)$ we have used~\eqref{subs6a79cd8f} with $c=  \frac{  R }{ d(x,b) } \le 1$, $s=0$ and $t = \frac{ d(x,a) }{ R } \le 1$.
\end{proof}

\begin{theorem}\label{thm6a79e253}
	Quasi-conformal homeomorphisms between geodesic Lie groups of dimension greater than 1 are (global) quasisymmetries.
\end{theorem}
\begin{proof}
	Let $(G_1,d_1)$ and $(G_2,d_2)$ be two geodesic Lie groups and $f:G_1\to G_2$ a quasiconformal homeomorphism.
	By Theorem~\ref{thm68780db5}, they both fall into one of three categories: 
	hyperbolic, liminal, or strictly parabolic.
	
	In the liminal case, both $(G_1,d_1)$ and $(G_2,d_2)$ are Loewner, both simultaneously bounded or unbounded, both linearly locally connected.
	Therefore, by~\cite[Corollary 4.10.]{MR1654771}, $f$ is a quasisymmetry.
	
	In the hyperbolic and strictly parabolic cases, we know $f$ is locally quasisymmetric by~\cite[Theorem 9.8]{MR1869604}.
	In fact, the quantitative result \cite[Theorem 9.8]{MR1869604} implies that $f$ is locally quasisymmetric, specifically, there are $R>0$ and $\eta$ such that $f$ is a $\eta$-quasisymmetry on every ball of radius $R$.
	Further, we know that $f$ is a quasi-isometry by Theorem~\ref{thm68780db5}:
	the definition of quasi-isometry implies that there is $L>0$ so that~\eqref{eq6a79878c} holds.
	By Proposition~\ref{prop6a79864e}, $f$ is a quasisymmetry.
\end{proof}

\subsection{1-quasiconformal maps are isometries}\label{sec6a70fe5d}

Before stating the main result of this section, we discuss a few standard notions from analysis on metric measure spaces.
Let $(X,d,\mu)$ be a metric measure space with  $Q$-bounded geometry and $f:X\to X$ a homeomorphism.
For every $x\in X$, define the \emph{pointwise Lipschitz constant}
\begin{equation}\label{eq6a70f040}
	|\Diff f|(x) = \lim_{r\to0^+} \sup_{y\in B_d(x,r)} \frac{ d(f(x),f(y)) }{ r } .
\end{equation}
The \emph{Jacobian} of $f$ (if it exists), is the the Radon–-Nikodym derivative of the pull-back measure under $f$, that is, it is the measurable function $J_f:X\to[-\infty,+\infty]$ such that, for every $h:X\to\R$ Borel
\begin{equation}\label{eq6a70f038}
	\int_X h(f(x)) J_f(x) \did\mu(x) = \int_X h(y) \did\mu(y) .
\end{equation}
It is a well known result (see \cite[Theorem 9.8.]{MR1869604} or \cite[Theorem 1.1]{zbMATH06028338}) that a homeomorphism $f:X\to X$ is $K$-quasi-conformal as in Definition~\ref{def68cc0a5c} if and only if $f\in W^{1,Q}(X;X)$ and
\begin{equation}\label{eq6a70f0f9}
	|\Diff f|^Q \le K J_f 
	\qquad\text{$\mu$-almost everywhere in }X.
\end{equation}
Other two classical results about quasiconformal maps are the following.

\begin{proposition}[{\cite[Theorems~3.4\&3.7]{MR595180}}]\label{prop6a79e0e1}
	Let $(X,d)$ and $(Y,d)$ be metric spaces, and $\eta:[0,+\infty)\to[0,+\infty)$ a homeomorphism.
	Let $\scr F$ be a family of $\eta$-quasisymmetries $X\to Y$ such that there are $a,b\in X$ and $M\ge0$ with $d(f(a),f(b))\le M$ for all $f\in\scr F$.
	Then $\scr F$ is equicontinuous.
	
	Moreover, if $\{f_n\}_{n\in\N}\subset\scr F$ is converging pointwise to a function $f:X\to Y$, then $f$ is either constant or $\eta$-quasisymmetric,
	and the convergence is uniform on every compact set.
\end{proposition}



\begin{proposition}\label{prop6a70ef2c}
	Let $(X,d,\mu)$ be a metric measure space with uniformly locally $Q$-bounded geometry with $Q>1$.
	Suppose that one of the following sets of conditions are satisfied:
	\begin{enumerate}
	\item
	$X$ has growth dimension $N$ with $1\le N<Q$, and $X$ is quasi-convex and quasi-straightenable (cfr.~Proposition~\ref{prop68755b1a}), or,
	\item
	$X$ is boundedly compact and quasi-convex, 
	the function $\mu^+_{X}:r\in (0,+\infty)\mapsto \sup\{\mu(B(x,r)):x\in X\}$ is finite valued and $\lim_{r\to\infty}\mu^+_{X}(r)=\infty$,
	and $\dimisp(X)>Q$ (cfr.~Theorem~\ref{thm685cec65}).
	\end{enumerate}
	
	Suppose that $(X,d)$ is isometrically homogeneous.
	
	Let $K\ge1$ and $\cal K$ be a group of $K$-quasi-conformal maps $X\to X$ containing the isometries.
	Then, for every $f\in\cal K$,
	\begin{equation}\label{eq6a70ef25}
		1 \le |\Diff f(x)| \le K^2,
		\qquad\text{for a.e.~}x\in X.
	\end{equation}
\end{proposition}
%

\begin{proof}
	First of all, using Proposition~\ref{prop6a79864e} and following the same strategy as for the proof of Theorem~\ref{thm6a79e253}, we obtain that all quasiconformal maps $X\to X$ are quasisymmetries.
	
	Next, we claim that, for every $f\in\cal K$,
	\begin{equation}\label{eq6a70f4a0}
		1 \le |\Diff f(x)| ,
		\qquad\text{for all }x\in X.
	\end{equation}
	Arguing by contradiction, suppose that there is $f\in\cal K$ for which~\eqref{eq6a70f4a0} does not hold.
	Then there are $x\in X$, $\bar r>0$ and $\lambda\in(0,1)$ such that 
	\begin{equation}\label{eq6a70f81c}
		f(B(x,r)) \subset B(f(x),\lambda r),
		\qquad\text{for all }r\in(0,\bar r).
	\end{equation}
	Let $L:X\to X$ be an isometry with $L(f(x)) = x$ and consider $f_L=L\circ f\in\cal K$.
	Since $f_L\in\cal K$, then $(f_L)^n\in\cal K$ is $K$-quasi-conformal for every $n\in\N$.
	 Notice that~\eqref{eq6a70f81c} implies
	 $(f_L)^nB(x,r) \subset B(x,\lambda^n r)$ for every $n\in\N$.
	 By Proposition~\ref{prop6a79e0e1}, we conclude that, 
	 $f_L^n$ converges to the constant map, 
	 as $n\to\infty$.
	 
	Seeking a contradiction, we claim that $f_L^n$ cannot converge to a constant map.
	Let $\delta$ be a Ferrand distance on $X$, depending on which of the two sets of hypotheses are satisfied.
	In the first case, $\delta$ is the parabolic Ferrand distance $\ferrdPar$;
	in the second case, $\delta$ is the hyperbolic Ferrand distance $\ferrdHyp$.
	In both cases, 
	there exist $R_0>0$ and functions $\psi_-,\psi_+:(R_0,\infty) \to (0,+\infty)$ such that, for every $n\in\N$ and every $R>KR_0$,
	\begin{equation}\label{eq6a79e4b8}
		B_d(x, \psi_-(R/K))  
		\overset{\eqref{eq68d0ee66}\text{ or }\eqref{eq6a70fca4}}\subset 
		(f_L)^n(B_\delta(x,R)) 
		\overset{\eqref{eq68d0ee66}\text{ or }\eqref{eq67f00646}}\subset 
		B_d(x , \psi_+(KR) ) .
	\end{equation}
	The second inclusion implies that $B_\delta(x,R)$ is pre-compact and the first inclusion implies that no limit of $\{(f_L)^n\}_{n\in\N}$ can be constant:
	contradiction.
	This shows our claim~\eqref{eq6a70f4a0}.
	
	
	The lower bound~\eqref{eq6a70f4a0} readily implies, for every $f\in\cal K$,
	\begin{equation}\label{eq6a70fb1f}
		\frac1K \le J_f(x)
		\qquad\text{for a.e.~}x\in X.
	\end{equation}
	
	Since $J_{f^{-1}}\circ f = \frac1{J_f}$ almost everywhere, the lower bound~\eqref{eq6a70fb1f} implies
	\begin{equation}\label{eq6a70fb6d}
		J_f(x) \le K
		\qquad\text{for a.e.~}x\in X.
	\end{equation}
	
	Plugging~\eqref{eq6a70fb6d} in the analytic definition~\eqref{eq6a70f0f9} of $K$-quasi-conformal maps, we conclude the proof. 
\end{proof}

A direct consequence of Proposition~\ref{prop6a70ef2c} for $K=1$ is the main theorem of the section.

\begin{theorem}\label{thm6a70fdfc}
	Let $(G,d,\mu)$ be a geodesic Lie group that is not liminal parabolic.
	All 1-quasi-conformal maps $G\to G$ are isometries.
\end{theorem}

In the liminal case, it is not true that 1-quasi-conformal maps are isometries.
In fact, dilations in $\R^n$, and in the other Carnot groups, are 1-quasi-conformal maps that are not isometries.
Interestingly, 1-quasi-conformal maps $\R^n\to\R^n$, $n\ge2$, are similarities by \cite[Theorem 4]{zbMATH03845917}.


\printbibliography
\end{document}